%% file: BB_arXiv.tex
\documentclass[leqno]{amsart}

\usepackage[letterpaper,margin=1.39in]{geometry} 

\usepackage{amsmath, amsfonts, amssymb, amsthm, mathrsfs, stmaryrd, wasysym, mathtools}
\usepackage{enumerate}
\usepackage{tikz, tikz-cd}
\usepackage[hidelinks]{hyperref}
\usepackage{cleveref}
\usepackage{standalone}
\usetikzlibrary{decorations.markings, arrows.meta}
\usepackage[backend=biber,style=alphabetic]{biblatex}
\usepackage{appendix}
\usepackage{bbm}
\usepackage[mathscr]{eucal}

\usepackage{yfonts}

\numberwithin{figure}{section}

\makeatletter
\renewcommand{\l@subsection}{\@tocline{2}{0pt}{2.5em}{}{}}
\makeatother

\usepackage[frame,cmtip,arrow,matrix,line,graph,curve]{xy}
\usepackage{graphpap,color,paralist}
\usepackage{tikz-3dplot}

\newtheorem{Thm}{Theorem}[section]
\newtheorem*{thm}{Theorem}
\newtheorem{Prop}[Thm]{Proposition}

\newtheorem{Lm}[Thm]{Lemma}
\newtheorem{Cor}[Thm]{Corollary}

\newtheorem{IntroThm}{Theorem} %alphabetic theorem counter: Theorem A, Theorem B, ...

\newtheorem{IntroCor}[IntroThm]{Corollary}

\theoremstyle{definition}
\newtheorem{Df}[Thm]{Definition}

\theoremstyle{remark}
\newtheorem{Rmk}[Thm]{Remark}
\newtheorem{Ex}[Thm]{Example}

\tikzset{
middle arrow/.style={
decoration={markings, mark=at position 0.7 with {\arrow{>}}},
postaction={decorate}
}
}

\tikzset{
middle arrow chain/.style={
decoration={
markings,
mark=at position 0.5 with {\arrow{>}},
mark=at position 0.75 with {\arrow{>}}
},
postaction={decorate}
}
}

\tikzset{
middle arrow cover/.style={
decoration={markings, mark=at position 0.8 with {\arrow{Latex}}},
postaction={decorate}
}
}

\newcommand{\midarrow}{
\;
\begin{tikzpicture}
\draw[middle arrow] (0,0) -- (0.4,0);
\end{tikzpicture}
\;
}

\newcommand{\middarrow}{
\;
\begin{tikzpicture}
\draw[middle arrow chain] (0,0) -- (0.4,0);
\end{tikzpicture}
\;
}

\newcommand{\midddarrow}{
\;
\begin{tikzpicture}
\draw[middle arrow cover] (0,0) -- (0.4,0);
\end{tikzpicture}
\;
}

\newcommand{\midarrowX}[1]{\ensuremath{\overset{#1}{\midarrow}}}

\newcommand{\ltl}{\ensuremath{<\hspace{-0.45em}<}}

\newcommand{\inner}[1]{\left\langle #1 \right\rangle}

\newcommand{\up}{\shortuparrow}
\newcommand{\down}{\shortdownarrow}
\newcommand{\orb}[1]{O_{#1}}
\newcommand{\Gm}[0]{\ensuremath{\mathbb{G}_m}}
\newcommand{\BB  }[0]{Bia\l ynicki-Birula }

\DeclareMathOperator{\del}{\backslash}
\DeclareMathOperator{\deq}{\vcentcolon=}
\newcommand{\R}{\mathbb{R}}
\newcommand{\Z}{\mathbb{Z}}
\newcommand{\QQ}{\mathbb{Q}}
\newcommand{\PP}{\mathbb{P}}
\DeclareMathOperator{\Pop}{Pop}
\DeclareMathOperator{\Spec}{Spec}
\DeclareMathOperator{\Hom}{Hom}
\DeclareMathOperator{\Span}{Span}
\DeclareMathOperator{\Conv}{Conv}
\DeclareMathOperator{\Cone}{Cone}

\DeclareMathOperator{\Bl}{Bl}
\DeclareMathOperator{\Div}{div}
\newcommand{\yeq}{\ensuremath{\preccurlyeq}}
\DeclareMathOperator{\Aut}{Aut}
\DeclareMathOperator{\PGL}{PGL}
\DeclareMathOperator{\GL}{GL}

\title{Structural properties of Bia\l ynicki-Birula decompositions}
\author{Teddy Gonzales}
\address{Rutgers University}
\email{tg570@math.rutgers.edu}

\author{Chayim Lowen}
\address{Princeton University}
\email{chayiml@princeton.edu}
\begin{document}
\begin{abstract}
We investigate several aspects of the \BB decomposition of a smooth complete $\Gm$-variety with finite fixed locus.
Our results include novel characterizations of when the \BB decomposition is filterable or forms a stratification, showing that these properties are invariant under reversing the $\Gm$-action.
We additionally classify the smooth projective toric varieties for which the \BB decomposition either may or must be a stratification.
Our study of $\Gm$-convexity and $\Gm$-rigidity---properties recently introduced by Buch--Chaput--Perrin---answers several questions posed in their \emph{Equivariant rigidity of Richardson varieties}.
In particular, assuming only filterability of the decomposition, we show that the \BB cell closures are determined by their $\Gm$-equivariant Chow classes.
\end{abstract}
\maketitle
\tableofcontents
\section{Introduction}\label{sec:intro}
In his seminal work entitled \emph{Some theorems on actions of algebraic groups} \cite{BB}, \BB introduced what is now known as the \BB decomposition for $\Gm$-varieties.
Our setting will be that of a smooth complete $\Gm$-variety $X$ with finite fixed locus.
In this case, his decomposition has the following simple description.
The cells of the decomposition are indexed by the fixed points of $X$.
The (positive) \BB cell of $X$ corresponding to a fixed point $p \in X^{\Gm}$ is defined to be the set
\[
X_p^+ = \left\{x \in X \;\middle|\;  \lim_{t\to 0} t \cdot x = p\right\}
\]
of points attracted to $p$ by the action of $t \in \Gm$.
The \emph{negative} \BB cell $X_p^-$ is obtained by instead taking $\lim_{t \to \infty}$.
Essential to the theory is that these cells are affine spaces.
\begin{thm}[{\cite[Theorem 4.4]{BB}}]\label[Thm]{thm:BB}
Each \BB cell of a smooth complete $\Gm$-variety with finite fixed locus is a locally closed subvariety isomorphic to some affine space $\mathbb{A}^n$.
\end{thm}
In the projective setting, the \BB decomposition is filterable (cf.\  \Cref{df:filterable}):
\begin{thm}[{\cite[Theorem 3]{BB_Example}}]\label[Thm]{BB_Filterability}
Let $X$ be a smooth projective $\Gm$-variety with finite fixed locus.
For a suitable ordering $p_1,\dots,p_n$ of the fixed points of $X$, the union $X_{p_1}^+ \sqcup \dots \sqcup X_{p_j}^+$ is closed in $X$ for each $j$.
\end{thm}
These two properties combine to give perhaps the most useful application of the \BB decomposition.
They ensure that the classes of the \BB cell closures form a basis for the Chow homology of $X$ \cite[Theorem 1]{chow_basis} (see also \cite[Appendix B, Lemma 6]{Fulton_tableaux} and \cite[Example 19.1.11(b)]{Fulton_intersection_theory}).
Similarly, they form a basis of the total equivariant Chow group of $X$ as a module over the equivariant Chow cohomology of a point \cite[\textsection3.2, Corollary 1(iii)]{Brion_1997} (see also \cite[Proposition 17.1.2]{Anderson_Fulton_2023}).

To understand when the decomposition is filterable for a smooth complete $\Gm$-variety $X$ with finite fixed locus, we consider its graph of $\Gm$-orbits.
Let $\Gamma_X$ be the directed graph with vertex set $X^{\Gm}$ and directed edges $p \to q$ for distinct $p, q \in X^{\Gm}$ whenever $X_p^- \cap X_q^+ \neq \emptyset$, i.e.\ whenever there is a $\Gm$-orbit which limits negatively to $p$ and positively to $q$.
We prove a connectedness property for this graph in \Cref{sec:graph}.
As an application, in \Cref{sec:filt_prop} we show the following.
\begin{IntroThm}[\ref{thm:filt_iff_acyclic}]
Let $X$ be a smooth complete $\Gm$-variety with finite fixed locus.
Then the (positive) \BB decomposition of $X$ is filterable if and only if $\Gamma_X$ contains no directed cycles.
\end{IntroThm}
\begin{IntroCor}[\ref{cor:filtration_reversibility}]
The positive \BB decomposition of $X$ is filterable if and only if the negative one is.
\end{IntroCor}
A strengthening of the filtration property for a decomposition is that of stratification.
Recall that a decomposition such as $X = \bigsqcup_{p \in X^{\Gm}} X_p^+$ is called a \emph{stratification} when each cell closure $\overline{X_p^+}$ is a union of cells in the decomposition.
\BB has given the following \emph{sufficient} condition for his decomposition to be a stratification.
\begin{thm}[{\cite[Theorem 5]{BB_Example}}]
Let $X$ be a smooth complete $\Gm$-variety with finite fixed locus.
If $X_p^+$ and $X_q^-$ meet transversely for all $p, q \in X^{\Gm}$, then the positive and negative \BB decompositions of $X$ are stratifications.
\end{thm}
Seeing as $X_p^+$ and $X_p^-$ are always transverse, the content of the above condition rests in the case $p \neq q$ and $X_p^+ \cap X_q^- \neq \varnothing$.
Transversality then implies $\dim X_p^+ + \dim X_q^- > \dim X$ since the intersection $X_p^+ \cap X_q^-$ is $\Gm$-stable and contains no fixed points---so has dimension at least 1.
We show in \Cref{subsec:strat_gen} that this numerical consequence of transversality is both necessary and sufficient to have a stratification.
\begin{IntroThm}[{\ref{thm:dim_strat}}]\label[Thm]{thm:numerical_trans}
Let $X$ be a smooth complete $\Gm$-variety with finite fixed locus.
The (positive) \BB decomposition of $X$ is a stratification if and only if $\dim X_p^+ + \dim X_q^- > \dim X$ for all distinct $p, q \in X^{\Gm}$ for which $X_p^+ \cap X_q^- \neq \varnothing$.
\end{IntroThm}
\Cref{ex:transversality_failure} shows that the numerical condition may hold even when transversality fails.
On the other hand, from the symmetry of the criterion it immediately follows:
\begin{IntroCor}[\ref{cor:strat_reversibility}]\label[Cor]{cor:reversible}
The positive \BB decomposition of $X$ is a stratification if and only if the negative one is.\footnotemark\footnotetext{\label{fn:foot1}
Shortly before the release of this arXiv preprint, we learned that Micha\l ek, Monin, and Wang \cite{MMW} had independently proved an analogue of Corollary D for polytopes, not necessarily smooth. This implies our statement in the special case of smooth projective toric varieties.
See also footnote~\ref{fn:foot2}.
}
\end{IntroCor}
If the positive and negative \BB decompositions of $X$ are both stratifications then the positive and negative cell closures form Poincar\'e dual bases for the Chow cohomology of $X$.
This was shown in \cite[Lemma 3.11]{BP22} in the projective setting, though the argument is completely general.
By \Cref{cor:reversible}, knowing just one of these to be a stratification suffices.

Toric varieties constitute a fruitful testing ground for the \BB theory.
Let us consider a smooth projective toric variety $X$ with torus $T \cong \Gm^n$ whose fan is the \emph{outer} normal fan of a lattice polytope $P$.
The polytope $P$ naturally sits inside $M\otimes_\mathbb{Z} \mathbb{R}$ where $M$ is the character lattice of $T$.
Recall that faces of $P$ are in natural bijection with the $T$-orbits of $X$; in particular, the vertices of $P$ correspond to the $T$-fixed points of $X$.
We call a cocharacter $v\colon \Gm \to T$ in $N \deq M^\vee$ \emph{admissible} if the corresponding fixed locus $X^{\Gm}$ is finite.
Equivalently, $X^{\Gm} = X^T$.
Note that this holds for a generally chosen $v\in N$.
Each choice of admissible $\Gm$ gives rise to a \BB decomposition.
We describe this decomposition in detail in \Cref{sec:toric_BB}.

The following theorem classifies toric varieties which admit a \BB stratification.
\Cref{prop:strat_bottom_up} refines this by identifying the cocharacters which give rise to stratifications.
\begin{IntroThm}[\ref{thm:existential_strat}, \ref{thm:universal_strat}]\label{introthm:strat}
Let $X$ be a smooth projective toric variety.
Let $P$ be a polytope representing an ample toric line bundle on $X$.
\begin{itemize}
\item The \BB decomposition of $X$ is a stratification for \emph{some} admissible cocharacter if and only if $P$ has the combinatorial type of a product of simplices.\footnotemark\footnotetext{\label{fn:foot2}The ``only if'' direction of this first half of \Cref{introthm:strat} was proven independently by Micha\l ek, Monin, and Wang (private communication) in the generality of arbitrary simple polytopes.
We do not know whether the ``if'' direction holds in this setting as well. See \cite{MMW} for the relevant definitions.
See also footnote~\ref{fn:foot1}.}
\item The \BB decomposition of $X$ is a stratification for \emph{all} admissible cocharacters if and only if $P$ is a product of simplices, up to a unimodular transformation.
\end{itemize}
\end{IntroThm}
The smooth projective toric varieties whose associated polytopes are products of simplices are exactly the products of projective spaces.
The smooth projective toric varieties whose associated polytopes are combinatorial products of simplices are still rather limited.
It turns out that they are precisely the \emph{generalized Bott towers}, i.e.\ the toric varieties obtained by iterated projectivizations of split toric vector bundles over a point.
This follows from the results in \cite{toric_classification}, \cite{quasitoric}, \cite{tuong_and_others}.
See \Cref{thm:GBT} and the discussion thereafter.

Buch--Chaput--Perrin \cite{buch} introduced a notion of equivariant homological rigidity for torus invariant subvarieties of $T$-varieties.
Given a smooth complete variety $X$ on which an algebraic torus $T$ acts with finite fixed locus, we say that a $T$-stable subvariety $Z$ is \emph{$T$-rigid} if whenever a positive $T$-invariant $\QQ$-cycle $\alpha$ in $X$ satisfies $\alpha = [Z]$ in $A_*^T(X)_{\QQ}$ (the torus-equivariant total Chow group of $X$ with $\QQ$-coefficients) then $\alpha = [Z]$ as a cycle.\footnote{Later, we refer to this notion as strong $\Gm$-rigidity.
See \Cref{rmk:differences} for the nuances of our conventions as against those of Buch--Chaput--Perrin.}
That is, $Z$ is determined by its equivariant class among positive $\mathbb{Q}$-cycles.
In \Cref{sec:rigidity}, we show that filterability of the \BB decomposition implies that all cell closures are $\Gm$-rigid.
\begin{IntroThm}[\ref{thm:BB_rigidity}]
\label[IntroThm]{introthm:rigidity}
Let $X$ be a smooth complete $\Gm$-variety with finite fixed locus whose \BB decomposition is filterable (e.g.\ $X$ is projective). Then each cell closure is $\Gm$-rigid.
\end{IntroThm}
Buch--Chaput--Perrin had proved the $T$-rigidity property for \BB cell closures of a smooth projective $T$-variety $X$ with finite fixed locus under the additional assumptions that the \BB decomposition of $X$ be a stratification and that the action be fully-definite: for each fixed point $p$ of $X$, the $T$-weights of the Zariski tangent space $T_pX$ lie strictly in a halfspace of the character lattice of $T$.
They also showed that the same assumptions guarantee that intersections of positive and negative \BB cell closures are $T$-rigid.
From this they deduce the equivariant rigidity of Schubert and Richardson subvarieties of flag varieties.
Note that for $T = \Gm$, the fully-definiteness condition is impossible to fulfill as soon as $\dim X > 1$.
As a remedy, we show that this assumption is unnecessary.

\begin{IntroThm}[\ref{thm:strong_richardson_rigidity}]
\label[IntroThm]{introthm:Richardson_rigidity}
If a smooth complete $\Gm$-variety $X$ with finite fixed locus has a \BB stratification, then for $p, q \in X^{\Gm}$ the intersection $\overline{X_p^+} \cap \overline{X_q^-}$ is $\Gm$-rigid whenever it is irreducible.
\end{IntroThm}
We give more refined sufficient conditions for $\Gm$-rigidity in \Cref{prop:weak_rigidity_sufficiency} and \Cref{cor:convexity_rigidity}. \Cref{ex:convexity_is_needed} shows that $\Gm$-rigidity need not hold for these intersections in the projective case, unlike for the \BB cell closures themselves.

A relaxation of the stratification condition was introduced by Buch--Chaput--Perrin under the banner ``$T$-convexity''.
A $T$-stable subvariety $Z$ of a $T$-variety $X$ with finite fixed locus is \emph{$T$-convex} if whenever we have containment of fixed point sets $Y^{T} \subseteq Z^{T}$---where $Y$ is a $T$-stable subvariety of $X$---we have containment of subvarieties $Y \subseteq Z$.
In their rigidity results mentioned above, the stratification condition may be weakened to the $T$-convexity of each \BB cell closure.
Specializing to $T = \Gm$, in \Cref{prop:convex_is_richardson} we show that when $X$ is filterable, a $\Gm$-convex subvariety must be an irreducible component of the intersection of a pair of opposite \BB cell closures.
As we did for stratifications, we examine more closely the $\Gm$-convexity property for smooth projective toric varieties.
This is done in \Cref{subsec:bad_convex}.
Whereas \BB cell closures are always $\Gm$-convex in curves and surfaces (\Cref{prop:uGm_convex_small}), we show that in higher dimensions $\Gm$-convexity fails quite often for toric varieties.
We note that the $\Gm$-convexity property depends sensitively on the choice of cocharacter (see \Cref{rmk:convex_non_rev}).
The following result shows that $\Gm$-convexity is a rather precarious property.
\begin{IntroThm}[\ref{prop:blowup}]\label[IntroThm]{introthm:blowup}
Suppose $X_0$ is a smooth projective toric variety with dense torus $T$ of dimension at least three.
Then the blowup $X$ of $X_0$ at two suitably chosen $T$-fixed points admits a non-$\Gm$-convex \BB cell closure for some admissible cocharacter.
\end{IntroThm}
Even more pathological behaviors are possible, as the following theorem demonstrates.
\begin{IntroThm}[\ref{prop:tetrahedron}]\label[IntroThm]{introthm:blowup_Pn}
Blow up all torus-fixed points in $\mathbb{P}^n$ to get a toric variety $X$.
If $n\geq 3$, then each admissible cocharacter gives rise to at least one non-$\Gm$-convex \BB cell closure.
\end{IntroThm}
More generally, we show that for ``most'' smooth projective toric varieties blowing up all fixed points results in the same pathology (see \Cref{prop:poponce}).

\section{Notation}\label{sec:notation}
\addtocontents{toc}{\protect\setcounter{tocdepth}{1}}
\subsection*{Varieties}
Let $\mathbb{K}$ be an algebraically closed field.
A \emph{variety} will mean an integral, separated, finite-type scheme over $\mathbb{K}$.
A closed (resp.\ open, locally closed) subvariety will mean an irreducible closed (resp.\ open, locally closed) subset equipped with the induced reduced scheme structure.
A subvariety \emph{simpliciter} will always be assumed closed.
Unless otherwise indicated, a \emph{point} in a variety means a closed point.
We follow the usual convention of identifying a variety with its set of closed points. For a subset $S$ of a variety, we write $\overline{S}$ for its Zariski closure.

We write $\Gm = \Spec \mathbb{K}[t, t^{-1}]$ for the multiplicative group over $\mathbb{K}$.
The automorphism of $\Gm$ induced by $t \mapsto t^{-1}$ we call \emph{reversal}.
An algebraic torus $T$ is an algebraic group isomorphic to $\Gm^n$ for some $n$.
A $T$-variety $X$ is a variety equipped with an algebraic action of $T$.

\subsection*{\BB subsets} It will be useful to extend the definition of \BB cells in the introduction as follows.
For any complete $\Gm$-variety $X$ and any point $x \in X$, the morphism $\Gm \to X$ defined by $t \mapsto t \cdot x$ extends uniquely to a morphism 
$\mathbb{P}^1 \to X$.
We write $\lim_{t \to 0} t \cdot x$ for the image of the point $0 \in \mathbb{P}^1$ under this morphism and $\lim_{t \to \infty} t \cdot x$ for the image of $\infty \in \mathbb{P}^1$.
For a closed point $p \in X^{\Gm}$, we define
the \emph{\BB subsets}
\[
X_p^+ = \left\{x \in X \;\middle|\;  \lim_{t\to 0} t \cdot x = p\right\}, \qquad X_p^- = \left\{x \in X \;\middle|\;  \lim_{t\to \infty} t \cdot x = p\right\}.
\]
When $X$ is smooth with finite fixed locus, this agrees with the usual definition of the \BB cells.
However, if $p$ is not an isolated fixed point these are \emph{not} \BB cells as defined in \cite[\textsection 4]{BB}.
Note that $X_p^-$ coincides with the $X_p^+$ cell for the reverse $\Gm$-action.
When $X$ is normal, the $\Gm$-action on $X$ is locally linear by a theorem of Sumihiro \cite[Lemma 8, Theorem 1]{Sumihiro}.
From this it easily follows that $X_p^+$ and $X_p^-$ are locally closed affine subschemes of $X$.
In general, applying this idea to the normalization of $X$ shows that $X_p^+$ and $X_p^-$ are constructible (a similar argument appears in \cite[Remark 6]{constructibleJ}).
We additionally use $X_p^0$ to denote the union of all irreducible components of $X^{\Gm}$ containing $p$.
When $X$ is smooth, $X^{\Gm}$ is too by \cite{Iversen1972_smooth_fixed} (see also \cite[Theorem 13.1]{Milne_2017}) and thus $X_p^0$ is a smooth subvariety of $X$.

\subsection*{Equivariant Chow groups}
For a $T$-variety $X$, we write $Z_*^T(X)$ for the free abelian group of $T$-invariant cycles on $X$, which is naturally graded by dimension.
We write $Z_*^T(X)_{\mathbb{Q}}$ for $Z_*^T(X) \otimes \mathbb{Q}$ and call its elements \emph{$T$-invariant $\QQ$-cycles}.
A $T$-invariant cycle (resp.\ $\QQ$-cycle) is called \emph{positive} if it is nonzero and all its coefficients are nonnegative.
We write $A_*^T(X)$ for the total equivariant Chow group of $X$ and $A_*^T(X)_{\QQ}$ for $A_*^T(X) \otimes \QQ$.
For the relevant definitions and properties of equivariant Chow groups, see \cite{graham} and \cite[\textsection17]{Anderson_Fulton_2023}.
For a $T$-stable subvariety $Z \subseteq X$ we write $[Z]$ for the corresponding $T$-invariant cycle in $Z_*^T(X)$ or $Z_*^T(X)_{\QQ}$.
We use the same symbol for the equivariant fundamental class of $Z$, i.e.\ the image of this cycle in either $A_*^T(X)$ or $A_*^T(X)_{\QQ}$.

For a point $p \in X$ we write $T_p X$ for its Zariski tangent space in $X$.
When $p$ is a fixed point, $T_pX$ is naturally a $T$-representation.
When $p \in X^T$ is smooth in $X$, we write $\alpha|_p$ for the pullback $\iota_p^*(\alpha) \in A_*^T(p)$ of a class $\alpha \in A_*^T(X)$ along the inclusion $\iota_p\colon p \hookrightarrow X$.
For a $T$-equivariant vector bundle $E$ on $X$ we write $c_k^T(E)$ for the $k$-th equivariant Chern class of $E$.
We write $c_{\text{top}}^T(E)$ for the equivariant Chern class $c_{\text{rank}(E)}^T(E)$ of top degree.

\subsection*{Toric varieties}
Given an algebraic torus $T$, we let $M = \Hom(T, \Gm)$ and let $N = \Hom(\Gm, T)$.
These are, respectively, the character and cocharacter lattices of $T$.
We write $M_{\R} \deq {M \otimes_{\Z} \R}$ and $N_{\R} \deq {N \otimes_{\Z} \R}$.
Given $\chi \in M_{\R}$ and $v \in N_{\R}$, we write $\langle \chi, v \rangle$ for the value of the canonical pairing $M_{\R} \times N_{\R} \to \R$.
By contrast, for $v \in N$, we write $v(t)$ for the image of $t \in \Gm$ under $v$.

Let $P$ be an \emph{integral} full-dimensional convex polytope in $M_{\R}$.
The \emph{outer} normal fan $\Sigma$ to $P$ is the complete fan in $N_{\R}$ with cones 
\[
\sigma_F \deq \{u \in N_{\R} \,\mid\, \langle p , u \rangle \leq \langle f , u \rangle \text{ for all } p \in P, f \in F\},
\]
one for each face $F$ of $P$.
This fan defines a projective toric variety $X$ (see \cite{fult93toric}, \cite{cox2011toric}).
The faces of $P$ are in an inclusion- and dimension-preserving correspondence with the $T$-stable subvarieties of $X$.
Each such subvariety is the closure of a unique $T$-orbit.
We write $\orb{F}$ for the $T$-orbit corresponding to a face $F$ of $P$.
The correspondence above gives $\overline{O_F} = \bigcup_{E \subseteq F} \orb{E}$.
We write $N_F$ for the sublattice of $N$ generated by $\sigma_F \cap N$.
The orbit $O_P$ is the (unique) dense open orbit of $X$, which we identify with $T$ itself.
For a vertex $p \in P$, the orbit $O_{p}$ consists of a single fixed point, which we identify with the point $p$.
For all projective toric varieties we consider, we assume that they have been constructed in the above manner.
Said otherwise, whenever we work with a projective toric variety, we have the choice of an ample toric line bundle in mind.

In this text, we always assume that the polytope $P$ is \emph{smooth}.
Recall that smoothness for an integral polytope $P$ means that at each vertex of $P$ the primitive edge directions form a $\Z$-basis of the lattice $M$.
This condition is precisely what is needed to have $X$ be smooth.
Note that a smooth polytope is in particular simple: each vertex of $P$ has $\dim P$ neighbors.

Recall that two polytopes are said to be \emph{combinatorially equivalent} if their face lattices are isomorphic.
Two polytopes are said to be \emph{affinely equivalent} if one may be mapped bijectively onto the other by an affine-linear map.
Two \emph{lattice} polytopes are \emph{unimodularly equivalent} if each may be mapped bijectively onto the other by a lattice-preserving affine-linear map.

For much of our discussion involving toric varieties, we implicitly assume that a choice of cocharacter $v \in N$ has been made, giving rise to a corresponding $\Gm$-action.
The cocharacter $v$ is assumed to be \emph{admissible}.
By this we mean that $v$ is not perpendicular to any edge of $P$.
This is equivalent to asking that the $T$- and $\Gm$-fixed points of $X$ coincide.

\section{The Bia\l ynicki-Birula cells of a projective toric variety}\label{sec:toric_BB}
\addtocontents{toc}{\protect\setcounter{tocdepth}{2}}

Throughout this section, $X$ will denote a smooth projective toric variety with torus $T$ defined by an integral polytope $P$ embedded in the character lattice of $T$.
We let $v$ represent some fixed admissible cocharacter, which gives $X$ the structure of a $\Gm$-variety.
\begin{Df}\label[Df]{df:P_plus}
Consistent with our identification of vertices of $P$ with the corresponding fixed points of $X$, we write 
$X_p^+$ for the \BB cell attracted to the fixed point corresponding to a vertex $p$ of $P$.
We write $P_p^+$ for the face of $P$ satisfying $\overline{\orb{P_p^+}} = \overline{X_p^+}$.
The existence of such a face is guaranteed by the fact that $\overline{X_p^+}$ is $T$-stable and irreducible.
\end{Df}
\begin{Df}
For a face $F$ of $P$, we write $F^{\up}$ for the vertex of $F$ which maximizes the linear functional $\inner{-,v}$ and $F^{\down}$ for the vertex which minimizes it.
Uniqueness in each case follows from the admissibility of $v$.
\end{Df}
The following lemma identifies the $\Gm$-limit points of a $T$-orbit $\orb{F}$.
\begin{Lm}\label[Lm]{lm:Fup}
Let $F$ be a face of $P$.
Then $\orb{F} \subseteq X_{F^\up}^+$.
\end{Lm}
\begin{proof}
Replacing $X$, $N$, $M$, and $v$ by, respectively, $\overline{\orb{F}}$, $N/N_{F}, M\cap \sigma_F^{\perp}$, and $v + N_F$, we reduce to the case $F = P$, ${O}_F = T$.
Since $v$ is admissible and pairs maximally with $P^\up$, it lies in the interior of $\sigma_{P^\up}$.
By \cite[p.\,38, Claim 1]{fult93toric}, the limit $\lim_{t\to 0} v(t)$ is given by the distinguished point corresponding to $\sigma_{P^\up}$, which is precisely the fixed point corresponding to $P^\up$.
It follows that $ \orb{P}\subseteq X_{P^\up}^+$ since $\lim_{t \to 0} (v(t) \cdot x) = x \cdot \lim_{t \to 0} v(t)$ for any $x \in T$.
\end{proof}
For any face $F$ of $P$ and any point $y \in O_F$, the lemma above shows that the $T$-fixed point ${\lim_{t\to 0} (v(t) \cdot y)}$ is the one associated to $F^\up$.
Symmetrically, $\lim_{t\to \infty} (v(t)\cdot y)$ is the $T$-fixed point associated to $F^\down$.
These two are precisely the $T$-fixed points in the closure of the $\Gm$-orbit of $y$.
In this language, we can write for any vertex $p$ of $P$
\begin{equation}\label{eq:union}
X_{p}^+ = \bigsqcup_{{F^{\up} = p}} \orb{F}\tag{$\star$}.
\end{equation}
Comparing this equality with the defining equation $\overline{\orb{P_p^+}} = \overline{X_p^+}$ for $P_p^+$, we conclude:
\begin{Cor}\label[Cor]{cor:BB_face}
Let $p$ be a vertex of $P$.
The face $P_p^+$ of $P$ is characterized combinatorially as the unique maximal face occurring in the union (\ref{eq:union}), i.e.\ the unique maximal face $F$ of $P$ satisfying $F^{\up} = p$.
\end{Cor}
\begin{Rmk}\label[Rmk]{rmk:finiteness}
It follows from \Cref{cor:BB_face} that, for fixed $X$, the \BB decomposition depends only on the preorder on the vertices of $P$ induced by the (admissible) functional $\inner{-,v}$.
In particular, only finitely many behaviors are possible.
\end{Rmk}
We next show that $\overline{Y_p^+} = Y \cap \overline{X_p^+}$ for a $T$-stable subvariety $Y \subseteq X$.
We will use this later to show that certain properties of the \BB decomposition are heritable.
\begin{Lm}\label[Lm]{lm:BB_inheritance}
Let $Y \subseteq X$ be a $T$-stable subvariety.
Let $Q$ be the face of $P$ such that $Y = \overline{\orb{Q}}$.
Then $Q_q^+ = Q \cap P_q^+$ and thus $\overline{Y_q^+} = Y \cap \overline{X_q^+}$.
Here, the $\Gm$-structure on $Y$ is induced by the image $\bar v \in N/N_Q$ of the cocharacter $v$.
\end{Lm}
\begin{proof}
By making repeated use of \Cref{cor:BB_face}, we see that $(Q_q^+)^{\up} = q$, that $Q_q^+ \subseteq P_q^+$, that $(P_q^+ \cap Q)^{\up} = q$ and finally that $P_q^+ \cap Q \subseteq Q_q^+$.
We conclude that $P_q^+ \cap Q = Q_q^+$.
\end{proof}
The following notational shorthand will become useful in later sections.
\begin{Df}\label[Df]{df:arrows_for_poly}
Given an admissible cocharacter $v$ and adjacent vertices $p, q \in P$, we will use the notation $p\midddarrow q$ to indicate $p, q$ are adjacent in the polytope and $\inner{p, v} < \inner{q, v}$.
\end{Df}
\begin{Ex}
Let $S_{n+1}$ be the permutation group on $0,\dots,n$.
It acts via 
$\sigma \cdot e_i = e_{\sigma(i)}$ on the vector space $\R^{n+1}$ with standard basis 
$e_0, e_1, \dots, e_n$.
The permutahedron $\Pi_n \subseteq \R^{n+1}$ is the convex hull of the $S_{n+1}$-orbit of the point $(0,1,\dots,n)$.
We identify the permutation $\sigma$ with the vertex $\sigma \cdot (0,1,\dots,n) = (\sigma^{-1}(0), \dots, \sigma^{-1}(n))$.

The $n$-dimensional polytope $\Pi_n$ lies in an affine hyperplane orthogonal to $\mathbbm{1} \deq (1,\dots,1)$.
The normal fan ${\Sigma_n \subseteq \mathbb{R}^{n+1}\hspace{-0.2em}/\mathbbm{1}}$ of $\Pi_n$ defines the permutahedral variety $X_n$, which is the $T$-orbit closure of a general point in the complete flag variety of type A.
We consider the \BB decomposition on $X_n$ induced by a vector $v = (v_0, \dots, v_n) \in \mathbb{Z}^{n+1}$ for which $v_0 \ltl \cdots \ltl v_n$.

\begin{figure}[ht]
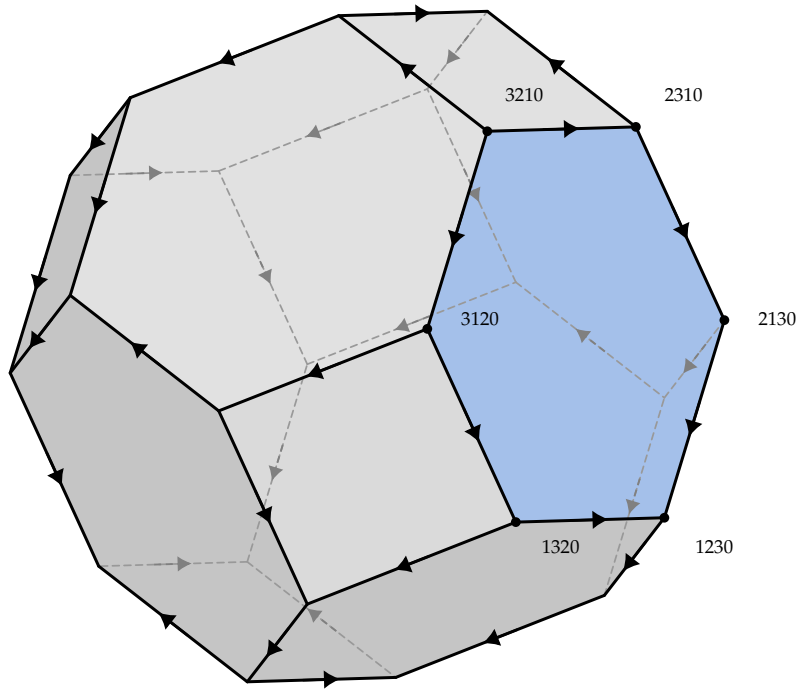

\includestandalone[width = 0.72\textwidth]{permutohedron}
\caption{
\centering
The flow on $\Pi_3$ with the facet $\overline{(\Pi_3)_{1230}^+}$ labeled, and vertices are labeled by the corresponding permutations.}\label{fig:permutahedron}
\end{figure}

The faces of $\Pi_n$ of codimension $k$ are indexed by ordered partitions $E_0 | E_1| \cdots | E_k$ of $\{0,\dots,n\}$ into $k+1$ nonempty parts.
The vertices contained in such a face correspond to the permutations in whose one-line notation all elements of $E_i$ precede all elements of $E_{i+1}$.
The maximizer of $\inner{-,v}$ on such a face is the permutation whose one-line notation is given by $\overrightarrow{E_0} \overrightarrow{E_1} \dots \overrightarrow{E_k}$ where $\overrightarrow{E_i}$ is the set $E_i$ listed in increasing order.
Similarly, the minimizer is  $\overleftarrow{E_0} \overleftarrow{E_1} \dots \overleftarrow{E_k}$ where $\overleftarrow{E_i}$ is the set $E_i$ listed in decreasing order.
By \Cref{cor:BB_face}, the \BB face ${(\Pi_n)_\sigma^+}$ of the vertex $\sigma \in S_{n+1}$ is the ordered partition obtained from the one-line word for $\sigma$ by inserting the separator $|$ at each descent of $\sigma$.
For example, the permutation $1230$ has associated positive \BB face given by the ordered partition $123|0$.
In general, the fixed points lying in the \BB cell closure of the permutation $\sigma \in S_{n+1}$ consist of the permutations in the weak Bruhat interval $[\tilde \sigma, \sigma]$ where $\tilde \sigma$ is obtained from $\sigma$ by reversing each ascending run.

As mentioned in the introduction, the \BB cell closures of a smooth projective variety form a basis for its total Chow group.
Since a permutation with $k$ descents has a \BB cell of codimension $k$, our computation above recovers the fact that the Betti numbers of the permutahedral variety $\Pi_n$ are given by the $(n+1)$-st Eulerian numbers---which by definition count the number of permutations of $\{0,\dots,n\}$ with exactly $k$ descents.
\end{Ex}

\section{The orbit graph}\label{sec:graph}
We now let $X$ represent any complete $\Gm$-variety with finite fixed locus.
\begin{Df}\label[Df]{df:arrow_notation}
Let $p, q \in X^{\Gm}$ be fixed points.
\begin{itemize}
\item For $x \in X\del X^{\Gm}$, we write $p \midarrowX{x} q$ to indicate that $p = \lim_{t \to \infty} t \cdot x$ and $q = \lim_{t \to 0} t \cdot x$.
\item We write $p \midarrow q$ to indicate that $p \midarrowX{x} q$ holds for \emph{some} $x \in X\del X^{\Gm}$.
\item We write $\middarrow$ for the transitive-reflexive closure of the  $\midarrow$ relation.
Explicitly, $p \middarrow q$ if there exists a chain of fixed points $p = p_0 \midarrow p_1 \midarrow \dots \midarrow p_k = q$ (possibly with $k = 0$).
\end{itemize}
\end{Df}
\begin{Df}\label[Df]{df:orbit_graph}
We define $\Gamma_X$ to be the directed graph with vertices $p \in X^{\Gm}$ and an edge $p \to q$ whenever $p \midarrow q$.
\end{Df}
There exist unique fixed points $p, q \in X^{\Gm}$ such that $X_p^+$ and $X_q^-$ are dense in $X$.
This follows from the constructibility of the \BB subsets of $X$ (see the discussion in \Cref{sec:notation}).
This leads to the following definition.
\begin{Df}
We let $X^{\down}, X^{\up}$ be the elements of $X^{\Gm}$ such that $X^{\down} \midarrowX{x} X^{\up}$ for general $x \in X$.
\end{Df}

The following result will be crucial for later sections.
\begin{Thm}\label[Thm]{thm:broken_trajectory}
Let $X$ be a complete $\Gm$-variety with finite fixed locus. Let $p \in X^{\Gm}$ be a fixed point.
Then in the graph $\Gamma_X$ we have $X^{\down} \middarrow p \middarrow X^{\up}$.
\end{Thm}

An elegant proof of the projective case of \Cref{thm:broken_trajectory} is given in \cite[Proposition 2.1]{totaro}.
The general case appears (in somewhat stronger forms) both as \cite[Proposition 2.3]{constructibleBB} and as \cite[Proposition 3.7]{stacky_totaro}. However, in each case there are unjustified steps.
In \cite{constructibleBB}, it is claimed without proof that there can only be one component in the fixed locus of a normal projective $\Gm$-variety with no outgoing $\Gm$-orbits. 
In \cite{stacky_totaro}, the existence of an \emph{equivariant} resolution of the rational map used there is not properly justified. Given the importance of this result to our later sections, we provide a self-contained proof here based on the ideas in \cite{stacky_totaro}. Our proof replaces the local computations there with a geometric argument.

To prove \Cref{thm:broken_trajectory}, we need a technical lemma on the behavior of smooth $\Gm$-surfaces under blowup.
%Our surface need not have finite $\Gm$-fixed locus.
We keep using the extended \BB notation from \Cref{sec:notation}.
\begin{Prop}\label[Prop]{prop:blowup_path}
Let $I$ be a smooth $\Gm$-surface.
Let $p \in I^{\Gm}$ be a fixed point.
Let $\varPhi \deq \Bl_p I$ be the blowup of $I$ at $p$.
Write $\rho\colon \varPhi \to I$ for the blowdown morphism and $O$ for the exceptional divisor.
The $\Gm$-action on $I$ lifts uniquely to $\varPhi$ and leaves $O$ invariant.
Moreover:
\begin{figure}[ht]
\includestandalone[width = 0.9\textwidth]{blowups}
\caption{The cases (II), (I), and (III) in \Cref{prop:blowup_path}.}\label{fig:blowups}
\end{figure}
\begin{enumerate}[(I)]
\item 
Suppose $\dim I_p^+, \dim I_p^- > 0$.
Then $O$ contains $\Gm$-fixed points $p^+, p^-$ such that ${\overline{\varPhi_{p^+}^-} =\overline{\varPhi_{p^-}^+} = O}$.
The strict transform of $\overline{I_p^+}$ is $\overline{\varPhi_{p^+}^+}$ and the strict transform of $\overline{I_p^-}$ is $\overline{\varPhi_{p^-}^-}$.
\item Suppose $\dim I_p^-, \dim I_p^0 > 0$.
Then $O$ contains $\Gm$-fixed points $p^-, p^0$ such that $\overline{\varPhi_{p^0}^-} = \overline{\varPhi_{p^-}^+} = O$.
The strict transform of $I_p^0$ is $\varPhi_{p^0}^0$ and the strict transform of $\overline{I_p^-}$ is $\overline{\varPhi_{p^-}^-}$.
\item Suppose $\dim I_p^+, \dim I_p^0 > 0$.
Then $O$ contains $\Gm$-fixed points $p^+, p^0$ such that $\overline{\varPhi_{p^0}^+} = \overline{\varPhi_{p^+}^-} = O$ and the strict transform of $I_p^0$ is $\varPhi_{p^0}^0$.
The strict transform of $\overline{I_p^+}$ is $\overline{\varPhi_{p^+}^+}$.
\end{enumerate}
\end{Prop}
\begin{proof}
That the action lifts uniquely follows from the universal property of blowups.
That $O$ is $\Gm$-stable follows from equivariance of the blowdown map.
Since $I$ is a smooth $\Gm$-variety, \cite[Theorem 4.1]{BB} implies that 
$T_p I_p^0$,
$T_p I_p^+$ and $T_p I_p^-$ 
are, respectively, the zero, positive, and negative $\Gm$-weight spaces of 
$T_p I$.
In particular,
$T_p I = T_p I_p^+ \oplus T_p I_p^0  \oplus T_p I_p^-$.
We have $O = \mathbb{P}(T_p I)$.

Consider case~(I).
In this case $\dim T_pI_p^+ = \dim T_pI_p^- = 1$ and $\dim T_pI_p^0 = 0$.
Let $p^+ = \mathbb{P}(T_p I_p^+)$ and $p^- = \mathbb{P}(T_p I_p^-)$.
For any smooth curve $C \subseteq I$ passing through $p$, the strict transform of $C$ in $\varPhi$ intersects $O$ at $\mathbb{P}(T_p C)$.
Let $x \in \rho^{-1}(I_p^+\del p)$ be a point.
Equivariance of $\rho$ gives $\rho(\lim_{t\to 0}t\cdot x) = \lim_{t\to 0} t\cdot\rho(x)$.
Since $\rho(x) \in {I_p^+}$, we have $\lim_{t\to 0} t\cdot x \in O$.
Thus, $\lim_{t\to 0} t\cdot x = p^+$.
It follows that the strict transform of $\overline{I_p^+}$ is contained in $\overline{\varPhi_{p^+}^+}$.
Similarly, the strict transform of $\overline{I_p^-}$ is contained in $\overline{\varPhi_{p^-}^-}$.
Since $O = \mathbb{P}(T_pI_p^- \oplus T_pI_p^+)$ where $T_pI_p^-$ has negative (hence strictly smaller) $\Gm$-weight, a general point $o \in O$ has $\lim_{t \to 0} t\cdot o = \mathbb{P}(T_p I_p^-) = p^-$ and $\lim_{t \to \infty} t\cdot o = \mathbb{P}(T_p I_p^+) = p^+$.
From the above discussion we see that $\dim \varPhi_{p^+}^+, \dim \varPhi_{p^+}^-, \dim\varPhi_{p^-}^+, \dim\varPhi_{p^-}^- \geq 1$.
Since 
\[
\dim \varPhi_{p^+}^+ + \dim \varPhi_{p^+}^- \leq \dim T_{p^+}\varPhi_{p^+}^+ + \dim T_{p^+} \varPhi_{p^+}^0 + \dim T_{p^+} \varPhi_{p^+}^- = \dim T_{p^+} \varPhi = 2,
\]
we must have $\dim \varPhi_{p^+}^+ = \dim \varPhi_{p^+}^- = 1$.
Similarly, $\dim \varPhi_{p^-}^+ = \dim \varPhi_{p^-}^- = 1$.
Hence the containments described above are actually equalities.

Cases (II) and (III) may be treated similarly.
One defines $p^0 = \mathbb{P}(T_pI_p^0)$ and uses that the strict transform of $I_p^0$ is a $\Gm$-fixed subvariety meeting $O$ at $p^0$ to conclude that it is contained in $\varPhi_{p^0}^0$.
A similar dimension argument finishes the proof.
\end{proof}

In the following proposition, by a \emph{trajectory} in a $\Gm$-variety $X$ we will mean the union of the $\Gm$-orbit closures of points $x_1, \dots, x_n \in X \del X^{\Gm}$ such that 
$
p_0 \midarrowX{x_1} p_1 \midarrowX{x_2} \dots \midarrowX{x_n} p_n
$
for some $p_0, \dots, p_n \in X^{\Gm}$ (not necessarily pairwise distinct).
We call $p_0, p_n$ the \emph{ends} of the trajectory.
\begin{Prop}\label[Prop]{prop:create_chain}
Let $f: C\to X$ be a morphism from a smooth curve $C$ whose general point maps inside $X_{X^{\up}}^+ \cap X_{X^{\down}}^-$.
The $\Gm$-action on $X$ induces a morphism $\Gm \times C\to X$, which we think of as a rational map $h: \mathbb{P}^1 \times C \dashrightarrow  X$.
This rational map may be resolved via a commutative diagram
\[
\begin{tikzcd}
I\ar[rd, "\tilde h"]\ar[d, "\rho"] & \\
\mathbb{P}^1 \times C \ar[r, "h" below, dashed] & X
\end{tikzcd}
\]
where $\rho$ is proper birational and $\tilde h (\rho^{-1}(\mathbb{P}^1 \times \{c\}))$ is a trajectory with ends $X^{\down}$ and $X^{\up}$ for each $c \in C$.
\end{Prop}
\begin{proof}
By the normality of $\mathbb{P}^1 \times C$ and the completeness of $X$, $h$ is defined on the complement of a finite set.
Restricting to an open subset on $C$, we may assume that $h$ is defined on $\mathbb{P}^1 \times (C \del c)$.
By construction, it is also well-defined on $\Gm \times \{c\}$.
By \cite[\href{https://stacks.math.columbia.edu/tag/0C5H}{Tag 0C5H}]{stacks-project}, there exists a diagram as above in which $I$ is obtained from $\mathbb{P}^1 \times C$ by a finite sequence of blowups at closed points.
Write $\mathbb{P}^1 \times C = I_0, I_1, \dots, I_n = I$ for the sequence of surfaces obtained in this way, blowing up one point at a time.
Without loss of generality, for each $j$ the point in $I_j$ blown up to get $I_{j+1}$ is a point where (the lift to $I_j$ of) $h$ is not defined.

We equip the surface $\mathbb{P}^1 \times C$ with the standard action of $\Gm$ on the first factor.
The rational map $h$ is $\Gm$-equivariant for this choice.
We prove the following statement by induction on $j$.
\begin{itemize}
\item[($\dagger$)] 
$I_j$ is a smooth $\Gm$-variety and (if $j \geq 1$) $I_{j}$ is obtained from $I_{j-1}$ by blowing up a $\Gm$-fixed point.
In particular, the blowdown $I_{j} \to I_{j-1}$ is $\Gm$-equivariant.
\end{itemize}
Let $j \geq 1$ and assume that ($\dagger$) holds for $j' < j$.
This ensures that the rational map $h: I_{j-1} \dashrightarrow X$ is $\Gm$-equivariant.
Since $I_{j-1}$ is normal, the lift of $h$ to $I_{j-1}$ is defined away from a finite set of closed points.
Since $h$ is $\Gm$-equivariant, this finite set must be $\Gm$-stable, so consists solely of $\Gm$-fixed points.
The point in $I_{j-1}$ blown up to get $I_{j}$ is one of these.
Since the blowup of a $\Gm$-variety along a $\Gm$-stable subvariety is canonically a $\Gm$-variety and the blowdown morphism is $\Gm$-equivariant, this proves the claim.

Now consider the preimage of $\mathbb{P}^1 \times \{c\}$ in $I_j$ as $j$ varies.
In $I_0$, this fiber is a single $\Gm$\nobreakdash-orbit closure whose ends $p = (\infty, c)$ and $q = (0, c)$ satisfy $\dim (I_0)_p^0, \dim (I_0)_q^0 > 0$ as well as $\overline{(I_0)_p^-} = \overline{(I_0)_q^+} = \mathbb{P}^1 \times \{c\}$.
Using ($\dagger$), we see that the changes to the fiber are described by \Cref{prop:blowup_path}.
Namely, at each stage a $\Gm$-orbit closure is added to the fiber, directed consistently with the existing trajectory.
In particular, $\rho^{-1}(\mathbb{P}^1 \times \{c\})$ is a trajectory.
The same is true for $\tilde h(\rho^{-1}(\mathbb{P}^1 \times \{c\}))$ since $\tilde h$ is $\Gm$-equivariant.
Note that orbits may be contracted in the process.
The start of the trajectory $\rho^{-1}(\mathbb{P}^1 \times \{c\})$ lies in the strict transform of $\{\infty\} \times C$ and its end lies in the strict transform of $\{0\} \times C$.
Since these map under $h$ (hence under $\tilde h$) to $X^{\down}$ and $X^{\up}$ respectively, these are the ends of the trajectory $\tilde h(\rho^{-1}(\mathbb{P}^1 \times \{c\}))$.
\end{proof}
\begin{proof}[Proof of \Cref{thm:broken_trajectory}]
Let $p \in X^{\Gm}$ be a fixed point.
Applying \cite[Corollary 1.9]{Bertini_Irred}, we can find a morphism $f\colon C \to X$ from an irreducible complete curve $C$ to $X$ with image containing $p$ and generically contained in $X_{X^\down}^- \cap X_{X^\up}^+$.
After normalizing, we may assume that $C$ is smooth.
Let $c \in C$ be such that $c \mapsto p$.
Composing with the $\Gm$-action gives a morphism $\Gm \times C \to X$.
Applying \Cref{prop:create_chain}, we get a trajectory with ends $X^{\down}, X^{\up}$.
From the construction, we see that this trajectory contains the image of $\Gm \times \{c\} \to X$, which is precisely $p$.
\end{proof}

\begin{Rmk}\label[Rmk]{rmk:loopless}
\Cref{thm:broken_trajectory} shows that $\Gamma_X$ is (weakly) connected for any complete $\Gm$-variety $X$ with finite fixed locus.
When $X$ is smooth, $\Gamma_X$ is also loopless (see \cite[Lemma 4.1]{BB}).
\end{Rmk}

\section{Filterability}\label{sec:filt_prop}
We recall what it means for the decomposition of a variety into subvarieties to be filterable.
\begin{Df}\label[Df]{df:filterable}
Let $X$ be an arbitrary variety.
Let $X = X_1 \sqcup \dots \sqcup X_n$ be a decomposition of $X$ into locally closed subvarieties.
We say that the ordering $X_1, \dots, X_n$ on the cells \emph{filters} $X$ if the partial unions $X_1 \sqcup \dots \sqcup X_j$ are closed in $X$ for all $1 \leq j \leq n$.
We say that the decomposition is filterable if it filters $X$ after suitably reindexing.
\end{Df}
In this section, we will let $X$ be any smooth complete $\Gm$-variety with finite fixed locus.
Considering the \BB decomposition of $X$, an ordering of the cells that filters $X$ is given by a total order on the fixed points of $X$.
We call this a \emph{filtering order on $X^{\Gm}$}.
\begin{Prop}\label[Prop]{prop:filter_cond}
Let $X$ be a variety and $X = X_1 \sqcup \dots \sqcup X_n$ a locally-closed decomposition of $X$.
The given order $X_1, \dots, X_n$ filters $X$ if and only if whenever $X_i \cap \overline{X_j} \neq \varnothing$ we have also $i \leq j$.
\end{Prop}
\begin{proof}
If $X_1, \dots, X_n$ is a filtering order, then $\overline{X_j} \subseteq \bigsqcup_{i \leq j} X_i$ for each $j$.
So $X_i \cap \overline{X_j} = \varnothing$ for $i > j$.
Conversely, if this condition holds then for each $j$ the closedness of $\bigsqcup_{i \leq j} X_i$ follows from
\[
\bigsqcup_{i \leq j} X_i \subseteq \overline{\bigsqcup_{i \leq j} X_i} = \bigcup_{i \leq j} \overline{X_i} \subseteq \bigsqcup_{i \leq j} {X_i} \qedhere
\]
\end{proof}
As mentioned in the introduction, if $X$ is projective then its \BB decomposition is filterable by {\cite[Theorem 3]{BB_Example}}.
This is also true if $\dim X \leq 2$ since smooth complete curves and surfaces are projective by Zariski's criterion \cite[Corollary II.2.6]{zariski}, \cite[Theorem 1.28]{smoo_sur_proj}.
In our setting, much more can be said: 
\begin{Thm}\label[Thm]{thm:smoo_sur_are_toric}
Every smooth complete $\Gm$-surface with finite fixed locus is obtained from a smooth projective toric surface with torus $T$ by restricting the action via some cocharacter $\Gm \to T$.
\end{Thm}
A proof of \Cref{thm:smoo_sur_are_toric} was sketched in \cite[\textsection4.1]{Orlik}.
\Cref{append2} fills in the necessary details.

In \cite[Example 3.3.3]{JJ}, Jurkiewicz provides an example of a smooth complete toric threefold $X$ whose \BB decomposition for a suitable (admissible) cocharacter is not filterable.
The non-filterability of the decomposition is proved by producing a cycle in the graph $\Gamma_X$.
Here we show that such a cycle exists in any non-filterable \BB decomposition.
\begin{Thm}\label[Thm]{thm:filt_iff_acyclic}
Let $X$ be a smooth complete $\Gm$-variety with finite fixed locus.
The decomposition $\{X_p^+\}_{p \in X^{\Gm}}$ is filterable if and only if $\Gamma_X$ is directed acyclic.
\end{Thm}
\begin{proof}
Let $\Gamma_X'$ be the directed graph on $X^{\Gm}$ with an edge $p \to q$ whenever $p \neq q$ and $X_p^+ \cap \overline{X_q^+} \neq \varnothing$.
\Cref{prop:filter_cond} shows that the \BB decomposition on $X$ is filterable if and only if there is a total order $\yeq$ on $X^{\Gm}$ such that $p \prec q$ whenever $p \to q$ in $\Gamma_X'$.
This holds precisely when $\Gamma_X'$ is directed acyclic.

Now we compare the two directed graphs $\Gamma_X$ and $\Gamma_X'$.
If $p \midarrow q$ then $p \neq q$ (by \cite[Lemma 4.1]{BB}) and $p \in \overline{X_q^+}$; so $p \to q$ in $\Gamma_X'$.
This shows that $\Gamma_X$ is a subgraph of $\Gamma_X'$.
On the other hand, suppose $p \to q$ in $\Gamma_X'$.
For $x \in X_p^+ \cap \overline{X_q^+}$, the limit $\lim_{t \to 0} t \cdot x = p$ shows that $p \in \overline{X_q^+}$.
Then \Cref{thm:broken_trajectory} applied to $\overline{X_q^+}$ shows that $p \middarrow q$.
So $\Gamma_X'$ is contained in the transitive closure of $\Gamma_X$.
These results show that $\Gamma_X'$ is directed acyclic if and only if $\Gamma_X$ is.
\end{proof}
\begin{Rmk}
Retracing through the proof of \Cref{thm:filt_iff_acyclic} shows more:
if the \BB decomposition of $X$ is filterable then $\middarrow$ is a partial order on $X^{\Gm}$ and a filtering order is any total order which extends it.
\end{Rmk}
\begin{Cor}\label[Cor]{cor:filtration_reversibility}
If the decomposition $\{X_p^+\}_{p \in X^{\Gm}}$ is filterable then so is the decomposition $\{X_p^-\}_{p \in X^{\Gm}}$.
\end{Cor}
\begin{proof}
Reversing the $\Gm$-action on $X$ has the effect of reversing the directed edges in $\Gamma_X$.
This preserves the acyclicity of the digraph, so the result follows from \Cref{thm:filt_iff_acyclic}.
\end{proof}
The \BB decomposition restricts nicely to subvarieties.
\begin{Lm}\label[Lm]{lm:restriction}
Let $Y \subseteq X$ be a $\Gm$-stable subvariety.
Let $p \in Y^{\Gm}$.
Then $Y_p^+ = Y \cap X_p^+$.
\end{Lm}
\begin{proof}
Let $y \in Y$.
We have $y \in X_p^+$ if and only if $\lim_{t \to 0} t\cdot y = p$ if and only if $y \in Y_p^+$.
\end{proof}
\begin{Ex}
The tempting equality $\overline{Y_p^+} = Y \cap \overline{X_p^+}$ is generally false.
For instance, let $X = \mathbb{P}^3$ and let $Y = \mathbb{P}^1 \times \mathbb{P}^1$ be embedded in $X$ via Segre:
\[
([x_0 : x_1], [y_0 : y_1]) \mapsto [x_0 y_0 : x_0 y_1 : x_1 y_0 : x_1 y_1].
\]
Let $\Gm$ act on $X$ via $t \cdot [x : y : z : w] = [x : ty : t^2 z : t^3 w]$.
Then $Y$ is $\Gm$-stable and $X^{\Gm} = Y^{\Gm}$.
Letting $p = [0 : 1 : 0 : 0]$ and $q = [0 : 0 : 1 : 0]$, we have $q \in \overline{X_p^+}$ but $q \notin \overline{Y_p^+}$.

For a special case in which the equality \emph{does} hold, see \Cref{lm:BB_inheritance}.
\end{Ex}
The filterability property is heritable in the strongest possible sense.
\begin{Prop}\label[Prop]{prop:filter_sub_gen}
Let $X$ be a variety with a decomposition $X = X_1 \sqcup \dots \sqcup X_n$ into locally closed subvarieties which filters $X$ in the given order.
Let $Y \subseteq X$ be a locally closed subvariety. If those intersections $Y_i \deq Y \cap X_i$ which are non-empty are irreducible, then the decomposition $Y = Y_1 \sqcup \dots \sqcup Y_n$ (omitting those that are empty) filters $Y$ in the given order.
\end{Prop}
\begin{proof}
For each $j$, the assumption is that $X_1 \sqcup \dots \sqcup X_j$ is closed in $X$.
It follows that $Y_1 \sqcup \dots \sqcup Y_j$, which is the intersection of this set with $Y$, is closed in $Y$.
\end{proof}
Combining \Cref{lm:restriction} with \Cref{prop:filter_sub_gen} gives:
\begin{Cor}\label[Cor]{cor:filt_sub}
Let $Y \subseteq X$ be a smooth $\Gm$-stable subvariety.
If the \BB decomposition of $X$ is filterable, then so is the one on $Y$.
\end{Cor}
We omit the straightforward proof of the following lemma.
\begin{Lm}\label[Lm]{lm:Gm_product}
Let $X, Y$ be smooth complete $\Gm$-varieties with finite fixed loci.
Then $X \times Y$ is naturally a $\Gm$-variety, using the diagonal morphism $\Gm \to \Gm \times \Gm$.
For closed points $x \in X$, $y \in Y$,
\[
\lim_{t \to 0} t \cdot (x, y) = \left(\lim_{t \to 0} t \cdot x, \lim_{t \to 0} t \cdot y\right),\qquad \lim_{t \to \infty} t \cdot (x, y) = \left(\lim_{t \to \infty} t \cdot x, \lim_{t \to \infty} t \cdot y\right).
\]
Furthermore, we have the natural identifications
\[
(X \times Y)^{\Gm} = X^{\Gm} \times Y^{\Gm}, \quad (X \times Y)_{(p, q)}^+ = X_p^+ \times Y_q^+, \quad \overline{(X \times Y)_{(p, q)}^+} = \overline{X_p^+} \times \overline{Y_q^+}.
\]
\end{Lm}
The following proposition shows that filterability is preserved under taking products.
\begin{Prop}\label[Prop]{prop:filter_prod_gen}
Let $X, Y$ be arbitrary varieties.
Let $X = X_1 \sqcup \dots \sqcup X_n$ and $Y = Y_1 \sqcup \dots \sqcup Y_k$ be locally-closed decompositions of $X$ and $Y$.
We have an induced decomposition $\{X_i \times Y_j\}_{\substack{1 \leq i \leq n\\1 \leq j \leq k}}$ of $X \times Y$.
This decomposition is filterable if and only if the decompositions on $X$ and $Y$ are.
\end{Prop}
\begin{proof}
Let $\Gamma_X'$ be the graph on $\{1, \dots, n\}$ in which the edge $i \to j$ is present whenever $i \neq j$ and $X_i \cap \overline{X_j} \neq \varnothing$.
As in the proof of \Cref{thm:filt_iff_acyclic}, we apply \Cref{prop:filter_cond} to see that the decomposition of $X$ is filterable if and only if $\Gamma_X'$ is directed acyclic.
We similarly construct $\Gamma_Y'$ and $\Gamma_{X \times Y}'$.

Using $(X_i \times Y_j) \cap \overline{X_{i'} \times Y_{j'}} = (X_i\cap \overline{X_{i'}}) \times (Y_{j} \cap \overline{Y_{j'}})$, we see that the edge $(i, j) \to (i', j')$ is present in $\Gamma_{X \times Y}'$ precisely in the following cases: (i) $i \to i'$ in $\Gamma_X'$ and $j \to j'$ in $\Gamma_Y'$, (ii)~$i \to i'$ in $\Gamma_X'$ and $j = j'$, or (iii) $i = i'$ and $j \to j'$ in $\Gamma_Y'$.
This exactly describes $\Gamma_{X \times Y}'$ as the \emph{strong product} of the directed graphs $\Gamma_X'$ and $\Gamma_Y'$ (cf.\ \cite[\textsection10.1]{graphs}).
We conclude using the fact that a strong product of nonempty directed graphs is directed acyclic if and only if each factor is.
\end{proof}
Combining \Cref{lm:Gm_product} with \Cref{prop:filter_prod_gen} gives:
\begin{Prop}\label[Prop]{prop:filt_product}
Let $X, Y$ be smooth complete $\Gm$-varieties, each with finite fixed locus.
Equip $X \times Y$ with the diagonal $\Gm$-action.
Then the \BB decomposition of $X \times Y$ is filterable if and only if the decompositions on $X$ and $Y$ are each filterable.
\end{Prop}
In \cite[Theorem 4.6]{BB}, \BB claims that for a smooth complete $\Gm$-variety with finite fixed locus his decomposition always contains at least one cell of each dimension.
As a corollary, this would extend a theorem of Rosenlicht \cite[Theorem 1]{rosenlicht}, who showed that a complete $\Gm$-variety equivariantly embedded in a projective space must have strictly more fixed points than its dimension.
Unfortunately, the key lemma \cite[Proposition 4.7]{BB} used in the proof is false, as we show in \Cref{ex:cell_count}.
In \Cref{prop:cells_of_every_dim}, we show that the strategy of the proof can be rescued if the \BB decomposition is assumed filterable (e.g.\ in the projective setting).
The present authors do not know if either the statement about fixed point counts or the more refined statement about dimensions of cells continue to hold in the non-filterable setting.
\begin{Ex}\label[Ex]{ex:cell_count}
Here we give a counterexample to \cite[Proposition 4.7]{BB}.
This proposition states that for any locally closed decomposition of a complete algebraic scheme $X$ into affine cells (i.e.\ locally closed affine subvarieties) there exists a cell of dimension $j$ for all $0 \leq j \leq \dim X$.
This is false even if one assumes irreducibility of $X$.
For instance, let $X = \mathbb{P}^2$ with coordinates $x,y,z$.
Let $Y = V(x)$ and let $Z = V(x^2 - yz)$.
They intersect at $p \deq [0:0:1]$ and $q \deq [0:1:0]$.
We have
\[
\mathbb{P}^2 = (\mathbb{P}^2\del (Y \cup Z)) \sqcup (Y \del p) \sqcup (Z \del q),
\]
which is a locally closed decomposition into affine cells and has no cells of dimension $0$.
\end{Ex}
We now present an argument that repairs \cite[Proposition 4.7]{BB} in the filterable setting.
\begin{Prop}\label[Prop]{prop:cells_of_every_dim}
Let $X$ be a complete variety.
In any filterable decomposition of $X$ into locally closed affine subvarieties, there exist cells of each dimension $0 \leq d \leq \dim X$.
\end{Prop}
\begin{proof}
Let $X_1, \dots, X_n$ be the cells in some order in which they filter $X$.
Since there is certainly a cell of dimension $\dim X$, it suffices to show that if $\dim X_i = d > 0$ for some index $i$ then there exists an index $j$ such that $\dim X_j = d-1$.
Let $i$ be minimal for the condition $\dim X_i = d$.
Consider the closed subscheme $Y \deq \overline{X_i} \del X_i$ of $X$.
By the filtering assumption, $Y \subseteq \bigcup_{j < i} X_j$.
On the other hand, by \cite[\href{https://stacks.math.columbia.edu/tag/0BCV}{Tag 0BCV}]{stacks-project} (see also \cite[Corollary 21.12.7]{EGA4}), $Y$ is pure of dimension $d-1$.
We have $Y \neq \varnothing$ since $X_i$, being affine and positive-dimensional, is not complete.
It follows that $\dim X_j \geq d-1$ for some $j < i$.
We have equality by the minimality of $i$.
\end{proof}
The following result repairs \cite[Corollary 1]{BB} in the filterable setting.
\begin{Cor}
A smooth complete $\Gm$-variety $X$ with finite fixed locus whose \BB decomposition is filterable has at least $\dim X + 1$ fixed points.
\end{Cor}
\begin{proof}
\Cref{prop:cells_of_every_dim} ensures that $X$ has at least ${\dim X + 1}$ \BB cells.
Since these cells are in bijection with the fixed points of $X$, the result follows.
\end{proof}

\section{Stratifications}

\subsection{Generalities}\label{subsec:strat_gen}
In this section, $X$ will be a smooth complete $\Gm$-variety with finite fixed locus.
We give several equivalent conditions for the \BB decomposition to be a stratification---some of them new.
We show that the decomposition $\{X_p^+\}_{p \in X^{\Gm}}$  is a stratification if and only if $\{X_p^-\}_{p \in X^{\Gm}}$ is.
\begin{Df}\label[Df]{df:stratification}
A stratification of a variety $X$ is a partition of $X$ into finitely many (disjoint) locally closed subvarieties $X_1, \dots, X_n$ such that for all $i, j$ we have either $X_i \cap \overline{X_j} = \varnothing$ or $X_i \subseteq \overline{X_j}$.
Equivalently, the closure of each cell of the partition is a union of cells.
\end{Df}
\begin{Rmk}\label[Rmk]{rmk:strat_implies_filt}
Suppose $X = X_1 \sqcup \dots \sqcup X_n$ is a stratification by locally-closed subvarieties, where the order on the cells is weakly increasing in dimension.
If $X_i \cap \overline{X_j} \neq \varnothing$ then $X_i \subseteq \overline{X_j}$. If $i \neq j$, then we must have $\dim X_i < \dim X_j$ and thus $i < j$.
By \Cref{prop:filter_cond}, this shows that stratifications are filterable.
\end{Rmk}

\begin{Thm}\label[Thm]{thm:dim_strat}
The following are equivalent for a smooth complete $\Gm$-variety $X$ with finite fixed locus.
\begin{enumerate}[(i)]
\item The decomposition $\{X_p^+\}_{p \in X^{\Gm}}$ is a stratification.
\item For $p, q \in X^{\Gm}$, if $p \in \overline{X_q^+}$ then $X_p^+ \subseteq \overline{X_q^+}$.
\item For $p, q \in X^{\Gm}$, if $p \in \overline{X_q^+}$ then either $p = q$ or $\dim X_p^+ <  \dim X_q^+$.
\item For $p,q \in X^{\Gm}$, if $p \midarrow q$ then $\dim X_p^+ < \dim X_q^+$.
\item If $X_p^- \cap X_q^+ \neq \emptyset$ and $p\neq q$ then $\dim X_p^- + \dim X_q^+ > \dim X$.
\end{enumerate}
\end{Thm}
\begin{proof}
(i)$\Rightarrow$(ii): Clear from \Cref{df:stratification}.

(ii)$\Rightarrow$(i): For any $q \in X^{\Gm}$, the assumption gives $\overline{X_q^+} \supseteq \bigcup_{p \in \overline{X_q^+}^{\Gm}} X_p^+$.
The reverse inclusion holds in general (see \cite[Lemma 5.1]{buch}).

(ii)$\Rightarrow$(iii): If $X_p^+ \subseteq \overline{X_q^+}$ and $p \neq q$, then $X_p^+$ is contained in the proper closed subscheme $\overline{X_q^+} \del X_q^+$ of the irreducible variety $\overline{X_q^+}$, so has strictly smaller dimension.

(iii)$\Rightarrow$(ii):
We proceed by induction on $\dim X_q^+$, the case $\dim X_q^+ = 0$ being trivial since it forces $p = q$.
Suppose $\dim X_q^+ > 0$ and  $p \in \overline{X_q^+}$ is a fixed point other than $q$.
By \cite[\href{https://stacks.math.columbia.edu/tag/0BCV}{Tag 0BCV}]{stacks-project}, since $X_q^+$ is affine, its ``boundary'' $W \deq \overline{X_q^+}\del X_q^+$ is pure of codimension $1$ in $\overline{X_q^+}$.
Let $C \subseteq W$ be an irreducible component of $W$ containing $p$.
By \cite[Lemma 5.1]{buch}, we have $C \subseteq \bigsqcup_{c \in C^{\Gm}} X_c^+$.
The generic point of $C$ is contained in $X_r^+$ for a unique $r \in C^{\Gm}$.
Then $C \subseteq \overline{X_r^+}$.
Since $r \in C \subseteq \overline{X_q^+}$, (iii) gives
\[
\dim X_q^+ > \dim X_r^+ \geq  \dim C = \dim W = \dim X_q^+ - 1.
\]
Since $C \subseteq \overline{X_r^+}$ and both are irreducible, this forces $\overline{X_r^+} = C \subseteq \overline{X_q^+}$.
Since $p \in \overline{X_r^+}$ and $\dim X_r^+ < \dim X_q^+$, we may apply the induction hypothesis to $p, r$ to get $X_p^+ \subseteq \overline{X_r^+} \subseteq \overline{X_q^+}$.

(iii)$\Rightarrow$(iv):
This is clear from the fact that $p \midarrow q$ implies $p \neq q$ (see \Cref{rmk:loopless}) as well as $p \in \overline{X_q^+}$.

(iv)$\Rightarrow$(iii): 
First note that (iv) implies that 
$\dim X_p^+ < \dim X_q^+$ for distinct $p, q \in X^{\Gm}$ whenever
$p \middarrow q$.
Now assume $p \in \overline{X_q^+}$ and $p \neq q$.
\Cref{thm:broken_trajectory} applied to $\overline{X_q^+}$ implies $p \middarrow q$. Hence $\dim X_p^+ < \dim X_q^+$.

(iv)$\Leftrightarrow$(v): It suffices to observe that for any $p,q \in X^{\Gm}$, $p \midarrow q$ holds if and only if $p \neq q$ and 
$X_p^- \cap X_q^+ \neq \varnothing$, and similarly that 
$\dim X_p^+ < \dim X_q^+$ holds if and only if 
$\dim X_p^- + \dim X_q^+ > \dim X$.
The latter equivalence follows from \cite[Theorem 4.1(c)]{BB}, which gives the equality $\dim X_p^+ + \dim X_p^- = \dim X$.
\end{proof}
Since condition (v) in \Cref{thm:dim_strat} is invariant under reversing the $\Gm$-action, we deduce:
\begin{Cor}\label[Cor]{cor:strat_reversibility}
Let $X$ be a smooth complete $\Gm$-variety with finite fixed locus.
Then the positive \BB decomposition of $X$ is a stratification if and only if the negative one is.
\end{Cor}
The implication (v) $\Rightarrow$ (i) allows us to easily recover the main result in \cite{BB_Example}, which is a sufficient condition for the \BB decomposition to be a stratification.
\begin{Cor}\label[Cor]{cor:recover_BB_transversality}
Let $X$ be a smooth complete $\Gm$-variety with finite fixed locus.
Assume that the positive and negative \BB cells of $X$ meet transversely.
Then the positive and negative \BB decompositions are both stratifications.
\end{Cor}
\begin{proof}
Suppose not.
By \Cref{thm:dim_strat}, we can find distinct $p,q \in X^{\Gm}$ such that $X_p^- \cap X_q^+ \neq \emptyset$ but $\dim X_p^- + \dim X_q^+ \leq \dim X$.
Transversality forces this inequality to be an \emph{equality} as well as ${\dim (X_p^- \cap X_q^+) = 0}$.
Being 0-dimensional and $\Gm$-stable, the intersection $X_p^- \cap X_q^+$ must consist of $\Gm$-fixed points.
But the only fixed point in $X_p^-$ (resp.\ $X_q^+$) is $p$ (resp.\ $q$). A contradiction.
\end{proof}
The converse to \Cref{cor:recover_BB_transversality} does not hold, as the following example shows.
\begin{Ex}\label[Ex]{ex:transversality_failure}
Let $\Gm$ act on $\mathbb{P}^2 \times \mathbb{P}^2$ by
\[
t \cdot ([x_0:x_1:x_2],[y_0:y_1:y_2]) = ([x_0:tx_1:t^2x_2],[t^4y_0:t^2y_1:y_2]).
\]
For a global section $s$ of the line bundle $\mathcal{O}(m, n)$ on ${\mathbb{P}^2 \times \mathbb{P}^2}$, we write $V(s)$ for its vanishing locus and $D(s)$ for its non-vanishing locus.
We write $V(s_1, \dots, s_k)$ for $V(s_1) \cap \dots \cap V(s_k)$.
Consider the subvariety $X = V(x_0^2 y_0 + x_1^2 y_1 + x_2^2 y_2) \subseteq \mathbb{P}^2 \times \mathbb{P}^2$.
Then $X$ is smooth and $\Gm$-stable and has fixed locus $\{([e_i], [e_j]): i \neq j\}$.
Writing $X_{ij}^{\pm}$ for $X^{\pm}_{([e_i], [e_j])}$ we have
\begin{align*}
X_{01}^+ &= V(x_0^2 y_0 + x_1^2 y_1, y_2) \cap D(x_0y_1), &  X^+_{02} &= X \cap D(x_0y_2), \\
X_{10}^+ &= V(x_0, y_1, y_2) \cap D(x_1), & X_{12}^+ &= V(x_1^2 y_1 + x_2^2 y_2, x_0)\cap D(x_1 y_2),\\
X_{20}^+ &= V(x_0, x_1, y_1, y_2), & X_{21}^+ &= V(x_0, x_1, y_2) \cap D(y_1).
\end{align*}
One checks that
\begin{align*}
\overline{X_{01}^+} &= X_{01}^+  \sqcup X_{10}^+\sqcup X_{20}^+ \sqcup  X_{21}^+,  & \overline{X_{02}^+} &= X_{01}^+ \sqcup X_{02}^+ \sqcup X_{10}^+ \sqcup X_{12}^+ \sqcup X_{20}^+  \sqcup X_{21}^+,\\
\overline{X_{10}^+} &= X_{10}^+ \sqcup X_{20}^+, & \overline{X_{12}^+} &= X_{10}^+ \sqcup X_{12}^+  \sqcup X_{20}^+ \sqcup X_{21}^+,\\
\overline{X_{20}^+} &= X_{20}^+, &\overline{X_{21}^+} &= X_{20}^+ \sqcup X_{21}^+.
\end{align*}
So this \BB decomposition is a stratification.
However,
\[
X_{21}^- = V(x_1^2 y_1 + x_2^2 y_2, y_0) \cap D(x_2y_1)
\]
leads to the non-reduced intersection
\[
X_{21}^- \cap X_{01}^+ = V(x_1^2 y_1, y_0, y_2) \cap D(x_0x_2y_1).
\]
So $X_{21}^{-}$ and $X_{01}^+$ intersect non-transversely.
\end{Ex}
\begin{Prop}\label[Prop]{prop:BB_strat_inherit}
Suppose the \BB decomposition of $X$ is a stratification.
Let $p \in X^{\Gm}$ be a fixed point and let $Y \deq \overline{X_p^+}$.
Then the \BB decomposition of $Y$ (which is not necessarily smooth) is a stratification by locally closed cells isomorphic to affine spaces.
\end{Prop}
\begin{proof}
Let $q \in Y^{\Gm} \subseteq X^{\Gm}$.
By \Cref{lm:restriction} combined with the stratification assumption, we have $Y_q^+ = Y \cap X_q^+ = \overline{X_p^+} \cap X_q^+ = X_q^+$.
Now $\overline{Y_q^+} = \overline{X_q^+}$ is equal to the union $X_{r_1}^+ \sqcup \dots \sqcup X_{r_n}^+$ for some $r_1, \dots, r_n \in X^{\Gm}$.
Since $Y$ is closed in $X$, these cells all lie in $Y$.
In particular, $r_i \in Y^{\Gm}$ for all $i$.
By the same reasoning as before, we have $Y_{r_i}^+ = X_{r_i}^+$.
So $\overline{Y_q^+} = Y_{r_1}^+ \sqcup \dots \sqcup Y_{r_n}^+$.
\end{proof}
\begin{Rmk}
The full analogue for stratifications of \Cref{cor:filt_sub} fails.
For instance, if $Y$ is any smooth projective toric variety with torus $T$ and $\mathcal{L}$ is an ample toric line bundle on $Y$ then the projective space $X \deq \mathbb{P}(H^0(Y, \mathcal{L})^{\vee})$ is a smooth projective $T$-variety with finite fixed locus.
For a general cocharacter $\mathbb{G}_m \to T$, we have $X^{\Gm} = X^T$.
The \BB decomposition of $X$ is guaranteed to be a stratification by \Cref{thm:universal_strat}.
However, \Cref{thm:existential_strat} assures us that the \BB decomposition of $Y$ cannot be a stratification if $Y$ is chosen appropriately (e.g.\ a smooth toric surface with at least five $T$-fixed points).
See \Cref{lm:toric_strat_inherit} however for another special case in which the stratification property \emph{is} inherited.
\end{Rmk}
The following is the analogue for stratifications of \Cref{prop:filt_product}.
\begin{Prop}\label[Prop]{prop:strat_product_gen}
Let $X, Y$ be arbitrary varieties.
Let $X = X_1 \sqcup \cdots \sqcup X_n$ and $Y = Y_1 \sqcup \cdots \sqcup Y_k$ be locally-closed decompositions of $X$ and $Y$.
Then the decompositions on $X$ and $Y$ are stratifications if and only if the decomposition $\{X_i \times Y_j\}_{i,j}$ is a stratification of $X \times Y$.
\end{Prop}
\begin{proof}
First assume the decompositions for $X$ and $Y$ are stratifications.
Suppose $(X_i \times Y_j) \cap \overline{X_{i'} \times Y_{j'}} \neq \varnothing$.
Using $(X_i \times Y_j) \cap \overline{X_{i'} \times Y_{j'}} = (X_i\cap \overline{X_{i'}}) \times (Y_{j} \cap \overline{Y_{j'}})$, we see that $X_i\cap \overline{X_{i'}} \neq \varnothing$ and $Y_{j} \cap \overline{Y_{j'}} \neq \varnothing$.
By the stratification assumption, $X_i \subseteq \overline{X_{i'}}$ and $Y_j \subseteq \overline{Y_{j'}}$.
Thus $X_i \times Y_j \subseteq \overline{X_{i'} \times Y_{j'}}$.

Conversely, assume $\{X_i \times Y_j\}_{i,j}$ is a stratification of $X\times Y$.
We may also assume $Y_1 \neq \varnothing$.
If $X_i \cap \overline{X_{i'}} \neq\emptyset$ then $(X_i \times Y_1)\cap \overline{X_{i'} \times Y_1} \neq \emptyset$, so by assumption $X_{i} \times Y_1 \subseteq \overline{X_{i'}} \times \overline{Y_1}$.
Hence $X_{i} \subseteq \overline{X_{i'}}$.
This shows that the decomposition on $X$ is a stratification.
The same argument applies to $Y$.
\end{proof}
Combining \Cref{lm:Gm_product} with \Cref{prop:strat_product_gen} gives:
\begin{Prop}\label[Prop]{prop:strat_product}
Let $X, Y$ be smooth complete $\Gm$-varieties, each with finite fixed locus.
Equip $X \times Y$ with the diagonal $\Gm$-action.
Then the \BB decomposition of $X \times Y$ is a stratification if and only if the \BB decompositions on $X$ and $Y$ both are.
\end{Prop}
As suggested by \cite[Example 1]{BB_Example}, \BB decompositions are rarely stratifications.
Among smooth complete $\Gm$-surfaces with finite fixed-locus (necessarily rational), the only ones that admit \BB stratifications are $\mathbb{P}^2$ and the Hirzebruch surfaces.
Since every such surface is toric by \Cref{thm:smoo_sur_are_toric}, this follows from \Cref{lm:2_dim_stratified} and the correspondence between polygons and toric surfaces (see \cite[\textsection2.5]{fult93toric}). Stratifications become even rarer in higher dimensions: Theorems \ref{thm:universal_strat} and \ref{thm:existential_strat} show that among smooth projective toric varieties, the decomposition can be a stratification only when the associated polytope is combinatorially a product of simplices.

\subsection{The toric case}
For the remainder of this section, we assume that $X$ is the projective toric variety associated to a smooth integral lattice polytope $P$.
See \Cref{sec:notation} for our conventions for toric varieties.
By default we assume that $X$ has been equipped with the structure of a $\Gm$-variety by choosing an admissible cocharacter $v\colon \Gm \to T$.
Recall that the admissibility of $v$ means precisely $X^{T} = X^{\Gm}$.
For admissible $v$, we say that $v$ \emph{stratifies} $X$ if the corresponding {positive} \BB decomposition is a stratification.
Note that the \emph{negative} \BB decomposition of $X$ coincides with the positive one for $-v$.
Rewriting condition (ii) of \Cref{thm:dim_strat} in the toric case, we obtain the following.
(See \Cref{df:P_plus} for the $P_p^+$ notation.)
\begin{Lm}\label[Lm]{lm:toric_strat_crit}
$v$ stratifies $X$ if and only if $p\in P_q^+$ implies $P_p^+ \subseteq P_q^+$ for all vertices $p,q$ of $P$.
\end{Lm}
To get a simpler criterion for stratification, we make use of the $\midddarrow$ symbol from \Cref{df:arrows_for_poly}.
\begin{Lm}\label[Lm]{lm:toric_strat_dim}
$v$ stratifies $X$ if and only if whenever $p \midddarrow q$ for vertices $p, q \in P$ then $\dim P_p^+ < \dim P_q^+$.
\end{Lm}
\begin{proof}
If $p\midddarrow q$ then $p \midarrow q$.
Since $\dim P_p^+ = \dim X_p^+$ and $\dim P_q^+ = \dim X_q^+$, the criterion is necessary by the equivalence (i)$\Leftrightarrow$(iv) in \Cref{thm:dim_strat}.

Now suppose the dimension condition holds.
We prove that criterion (iii) in \Cref{thm:dim_strat} is satisfied.
Suppose $p, q \in X^{T}$ are distinct with $p \in P_q^+$.
In the polytope $P_q^+$, we can find a path
\[
p = p_0 \midddarrow p_1 \midddarrow \dots \midddarrow p_k = q
\]
By assumption, $\dim X_p^+ = \dim P_p^+ = \dim P_{p_0}^+ < \dots < \dim P_{p_k}^+  = \dim P_{q}^+ = \dim X_q^+$.
\end{proof}
\begin{Df}
We say that $X$ is \emph{existentially stratified} if $X$ is stratified by \emph{some} admissible cocharacter and \emph{universally stratified} if it is stratified by \emph{all} admissible cocharacters.
\end{Df}
\Cref{lm:toric_strat_dim} allows us to easily see when two-dimensional toric varieties are stratified.
\begin{Lm}\label[Lm]{lm:2_dim_stratified}
Assume $X$ is two-dimensional.
Then $X$ is existentially stratified if and only if $P$ is a triangle or a quadrilateral.
It is universally stratified if and only if $P$ is a triangle or a parallelogram.
\end{Lm}
\begin{proof}
For a vertex $p$ of $P$, we have $\dim P_p^+ = 0$ or $2$ if and only if $p$ is, respectively, the unique minimizer or maximizer of $v$.
Otherwise, $\dim P_p^+ = 1$.
From this we see that the criterion in \Cref{lm:toric_strat_dim} always holds for a triangle and always fails for a polygon with at least five vertices.
For a quadrilateral, the condition is equivalent to asking that the maximizer and minimizer of $v$ lie opposite one another on $P$.
If $P$ is a parallelogram then this always holds.

Suppose $P$ is any quadrilateral, with vertices $A, B, C, D$ in cyclic order.
If $v$ is a covector perpendicular to the diagonal $AC$, then the extrema of $v$ on $P$ are achieved at $B, D$.
This shows that $X$ is existentially stratified.
On the other hand, if for instance $AB$ is \emph{not parallel} to $CD$ then we can find an admissible covector $v$ such that $v(A) < v(B)$ and $v(C) < v(D)$.
Then neither $A$ nor $C$ can be the maximum and neither $B$ nor $D$ can be the minimum.
Thus the vertices $P^{\up}$ and $P^{\down}$ cannot lie opposite one another on $P$.
\end{proof}
We now show that stratifications are inherited by closed $T$-stable subvarieties.
\begin{Lm}\label[Lm]{lm:toric_strat_inherit}
Let $Y \subseteq X$ be a $T$-stable subvariety.
Let $Q$ be the face of $P$ such that $Y = \overline{O_Q}$.
If $v \in N$ stratifies $X$ then its image $\overline{v} \in N/N_Q $ stratifies $Y$.
\end{Lm}
\begin{proof}
Let $p, q$ be vertices of $Q$ with $p \in Q_q^+$.
\Cref{lm:BB_inheritance} gives $Q_q^+ = Q\cap P_q^+$.
Hence $p \in P_q^+$.
Since $v$ stratifies $X$, this implies $P_p^+ \subseteq P_q^+$ by \Cref{lm:toric_strat_crit}.
Intersecting with $Q$ gives $Q\cap P_p^+ \subseteq Q \cap P_q^+$.
That is, $Q_p^+ \subseteq Q_q^+$, again using \Cref{lm:BB_inheritance}.
This shows that $\overline{O_Q}$ is stratified, by \Cref{lm:toric_strat_crit}.
\end{proof}
Surprisingly, this lemma has a converse.
Before stating it, we make the following observation.
Let $q$ be a vertex of $P$.
From \Cref{cor:BB_face} we know that $P_q^+$ is the unique maximal face of $P$ on which $q$ maximizes $v$.
Using the $\midddarrow$ notation, this is precisely the face of $P$ spanned by $q$ and all vertices $p$ such that $p \midddarrow q$.
It follows that
\[\tag{$\star\star$}\label{eq:arrowcount}
\dim P_q^+ = \# \{\text{$p$ a vertex of $P$} \mid p \midddarrow q\}.
\]
\begin{Prop}\label[Prop]{prop:strat_bottom_up}
Assume $v \in N$ is admissible.
Then $v$ stratifies $X$ if and only if it stratifies each 2-dimensional $T$-stable subvariety of $X$.
\end{Prop}
\begin{proof}
The ``only if'' direction is a special case of \Cref{lm:toric_strat_inherit}.
Now suppose $v$ stratifies every 2-dimensional $T$-stable subvariety of $X$.
As observed in the proof of \Cref{lm:2_dim_stratified}, this means that the maximizer and minimizer of $v$ are not adjacent on any quadrilateral 2-face of $P$.
Let $p, q$ be adjacent vertices of $P$ with $p \midddarrow q$.
It suffices to show that $\dim P_p^+ < \dim P_q^+$, by \Cref{lm:toric_strat_dim}.
By observation (\ref{eq:arrowcount}), this can be done by comparing the number of incoming edges at $p$ and $q$.

Since $P$ is simple, every edge of $P$ incident to $p$ and not $q$ lies in a unique 2-face containing both $p$ and $q$.
The same is true for every edge of $P$ incident to $q$ and not $p$.
By assumption, every such 2-face is stratified by $v$.
From the proof of \Cref{lm:2_dim_stratified}, we see that each of these must be directed in one of five ways, as shown in \Cref{fig:two_face}.
\begin{figure}[ht]
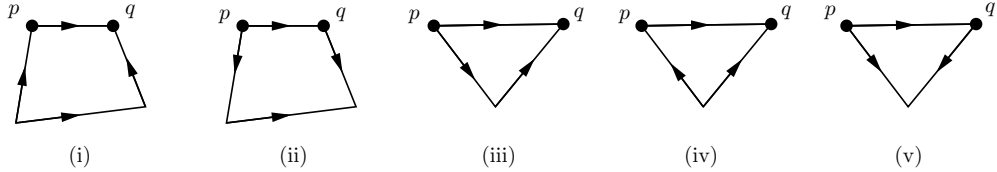

\includestandalone[width = 0.9\textwidth]{two_faced}
\caption{The five possible orientations induced by $v$ on a 2-face containing $p, q$.}\label{fig:two_face}
\end{figure}

\begin{samepage}
Excluding the edge $p \midddarrow q$, the 2-face contributes an incoming edge\dots
\begin{itemize}
\item \dots to both $p$ and $q$ in cases (i) and (iv).
\item \dots to neither $p$ nor $q$ in cases (ii) and (v).
\item \dots to $q$ but not $p$ in case (iii).
\end{itemize}
\end{samepage}
In total, $q$ gets at least as many incoming edges as $p$, and additionally has the incoming edge $p \midddarrow q$.
Thus $q$ has strictly more incoming edges than $p$.
So $\dim P_q^+ > \dim P_p^+$.
\end{proof}
Combining Lemmas \ref{lm:2_dim_stratified} and  \ref{lm:toric_strat_inherit} gives the following.
\begin{Cor}\label[Cor]{cor:strat_2_face}
If $X$ is existentially stratified, then every $2$-face of $P$ is a triangle or a quadrilateral.
If $X$ is universally stratified, then every $2$-face of $P$ is a triangle or a parallelogram.
\end{Cor}
These constraints turn out to be quite restrictive, as the following results show.
\begin{Thm}[{\cite[Appendix]{Yu_2021}}]\label[Thm]{thm:comb_prod}
If a simple convex polytope $P$ has all two-dimensional faces triangles or quadrilaterals, then $P$ is combinatorially equivalent to a product of simplices.
\end{Thm}
As noted by the authors of \cite{Yu_2021}, their result follows from a more general (and more difficult) theorem of Wiemeler \cite[Proposition 4.5]{Wiemeler_2015}.
\begin{Thm}\label[Thm]{thm:aff_prod}
If a simple convex polytope $P$ has all two-dimensional faces triangles or parallelograms, then $P$ is affinely equivalent to a product of simplices.
Moreover, if $P$ is a smooth lattice polytope, then it is unimodularly  equivalent to such a product.
\end{Thm}
We give
a self-contained proof of \Cref{thm:aff_prod} in \Cref{append1}.

Generalized Bott towers were first defined in \cite{GBT}, over $\mathbb{C}$.
Here we define them in general.
A preliminary observation:
if $X$ is a smooth projective toric variety with torus $T$ and $\mathcal{E}$ is a $T$-equivariant vector bundle over $X$ then the projectivized total space $\mathbb{P}(\mathcal{E})$ is naturally equipped with a toric structure precisely when $\mathcal{E}$ is equivariantly split (see \cite[\textsection7.3]{Cox_Caley_Poly}, \cite[41]{oda78}).
\begin{Df}\label[Df]{df:Bott} 
A smooth projective toric variety $X$ is a \emph{generalized Bott tower} if it either consists of a single point or is equivariantly isomorphic to the projectivization of a split equivariant toric vector bundle over a generalized Bott tower.
\end{Df}
\begin{Thm}[\cite{toric_classification}, \cite{quasitoric}, \cite{tuong_and_others}]\label[Thm]{thm:GBT}
Let $X$ be a smooth projective toric variety with $P$ an associated polytope.
$P$ is combinatorially equivalent to a product of simplices if and only if $X$ is a generalized Bott tower.
\end{Thm}
\begin{Rmk}
The authors in \cite{toric_classification} and \cite{quasitoric} are primarily concerned with real toric manifolds.
However, their methods with regard to \Cref{thm:GBT} are completely general and the proof extends to our setting.
\cite{tuong_and_others} provides a more direct, polytopal proof that avoids the machinery of quasitoric manifolds.
\end{Rmk}
\begin{Thm}\label[Thm]{thm:universal_strat}
Let $X$ be a smooth projective toric variety with an ample toric line bundle corresponding to a polytope $P$.
The following are equivalent.
\begin{enumerate}[(i)]
\item $X$ is universally stratified.
\item Every two-dimensional  $T$-stable subvariety of $X$ is universally stratified.
\item Every 2-face of $P$ is a triangle or parallelogram.
\item $P$ is unimodularly equivalent to a product of simplices.
\item $X$ is isomorphic as a toric variety to a product of projective spaces.
\end{enumerate}
\end{Thm}
\begin{proof}
The implication (i) $\Rightarrow$ (ii) follows from \Cref{lm:toric_strat_inherit}.
The implication (ii) $\Rightarrow$ (iii) follows from \Cref{lm:2_dim_stratified}.
The implication (iii) $\Rightarrow$ (iv) is \Cref{thm:aff_prod}.
The implication (iv) $\Rightarrow$ (v) is standard (see for instance \cite[Theorem 2.4.7]{cox2011toric} and Example 2.4.8 there).
As for {(v)~$\Rightarrow$~(i)}, note that $\mathbb{P}^n$ is universally stratified---as an easy consequence of \Cref{prop:strat_bottom_up}---and that products of universally stratified toric varieties are universally stratified by \Cref{prop:strat_product}.
\end{proof}
\begin{Thm}\label[Thm]{thm:existential_strat}
Let $X$ be a smooth projective toric variety with an ample toric line bundle corresponding to a polytope $P$.
The following are equivalent.
\begin{enumerate}[(i)]
\item $X$ is existentially stratified.
\item Every two-dimensional  $T$-stable subvariety of $X$ is existentially stratified.
\item Every 2-face of $P$ is a triangle or quadrilateral.
\item $P$ is combinatorially equivalent to a product of simplices.
\item $X$ is isomorphic as a toric variety to a generalized Bott tower.
\end{enumerate}
\end{Thm}
\begin{proof}
The implication (i) $\Rightarrow$ (ii) follows from \Cref{lm:toric_strat_inherit}.
The implication (ii) $\Rightarrow$ (iii) follows from \Cref{lm:2_dim_stratified}.
The implication (iii) $\Rightarrow$ (iv) is \Cref{thm:comb_prod}.
The implication (iv) $\Rightarrow$ (v) is \Cref{thm:GBT}.
To see (v) $\Rightarrow$ (i), we proceed by induction on the number of iterations of projective bundles, the base case being trivial.
By assumption, we may identify $X$ with the projectivized total space of an equivariantly-split toric bundle $\mathcal{E} \deq \mathcal{E}_0 \oplus \cdots \oplus \mathcal{E}_n$ over a generalized Bott tower $Y$.
Let $Q$ be a polytope in the character lattice $M$ of the dense torus of $Y$ which represents an ample toric line bundle on $Y$.
After tensoring $\mathcal{E}$ with a suitably high power of an ample toric line bundle on $Y$, we may assume that $\mathcal{E}_0, \dots, \mathcal{E}_n$ are (very) ample.
Let $Q_0, \dots, Q_n$ be the associated polytopes, all with the same normal fan as $Q$.
Then an equivariant ample line bundle on $X$ is represented by the polytope
\[
P \deq \Conv((\{e_0\} \times Q_0) \cup \cdots \cup (\{e_n\} \times Q_n)) \subseteq \R^n \times M_{\R}
\]
where $e_0 = 0$ and $e_1, \dots, e_n$ are the standard basis vectors of $\R^n$ and $\Conv$ is the convex hull operator.
See \cite[\textsection 3]{Lev_Caley_Poly}, \cite[\textsection 3]{Cox_Caley_Poly} for a discussion of this construction.
Note that we may have redefined $P$ in this process, though this has no effect on the validity of (v)$\Rightarrow$(i).
From the description of its facets in \cite[\textsection 3]{Cox_Caley_Poly}, one sees that the polytope $P$ is combinatorially equivalent to the product $\Delta_n \times Q$ where $\Delta_n$ is the simplex $\Conv(\{e_0, \dots, e_n\})$.
We may therefore identify the vertices of $P$ with the pairs $(e_i, q)$ where $q$ is a vertex of $Q$ and $0 \leq i \leq n$.

Let $v$ be a stratifying covector for $Q$.
Choose $w = (w_1,\dots,w_n) \in \Z^n$ with ${0< w_1 \ltl \cdots \ltl w_n}$.
Let $(e_i, q)$ be any vertex of $P$.
Then $P^+_{(e_i,q)}$ is a face of $P$, so corresponds---under the combinatorial isomorphism $P \cong \Delta_n \times Q$---to a product of a face in $\Delta_n$ and a face in $Q$.
By the choice of $w$, the former must be $\Conv(e_0, \dots, e_i)$; the latter is $Q_q^+$.
Thus
\[
P_{(e_i,q)}^+ = \Conv((\{e_0\} \times (Q_0)_q^+) \cup \dots \cup (\{e_i\} \times (Q_i)_q^+)).
\]    
If $(e_j, p) \in P_{(e_i, q)}^+$, then $j \leq i$ and $p \in Q_q^+$.
From the analogous expression for $P_{(e_j,p)}^+$, we see that $P_{(e_j,p)}^+ \subseteq P_{(e_i,q)}^+$.
By \Cref{lm:toric_strat_crit}, this means the vector $(w,v)$ stratifies $P$.
\end{proof}

\section{Convexity}

\subsection{Generalities}
We recall the definition of $\Gm$-convexity from \cite{buch}.
Here, $X$ will be a smooth complete $\Gm$-variety with finite fixed locus.
\begin{Df}\label[Df]{df:TConv}
Suppose $Y \subseteq X$ is a $\Gm$-stable subvariety.
We say that $Y$ is \emph{$\Gm$-convex} if, for any $\Gm$-stable subvariety $Z \subseteq X$, we have ${Z \subseteq Y \iff Z^{\Gm} \subseteq Y^{\Gm}}$.
\end{Df}
\begin{Rmk}\label[Rmk]{rmk:intersectability}
Suppose that $Y, Z \subseteq X$ are $\Gm$-stable $\Gm$-convex subvarieties.
If $Y\cap Z$ is irreducible, then it too is $\Gm$-convex.
\end{Rmk}
The following lemma shows that $\Gm$-convexity is a rather simple property to test.
Recall the $\midarrow$ notation from \Cref{df:arrow_notation}, as well as the graph $\Gamma_X$ constructed in \Cref{df:orbit_graph}.
\begin{Lm}\label[Lm]{lm:check_orbits}
Let $Y \subseteq X$ be a $\Gm$-stable subvariety of $X$.
Then $Y$ is $\Gm$-convex if and only if for all $x \in X\del X^{\Gm}$ and $p,q\in Y^{\Gm}$ it holds that
$p \midarrowX{x} q$ implies $x \in Y$.
\end{Lm}
\begin{proof}
Suppose $p \midarrowX{x} q$.
Then the $\Gm$-fixed points of $\overline{\Gm \cdot x}$ lie in $Y$.
If $Y$ is $\Gm$-convex, this forces $x \in Y$.
Conversely, suppose the criterion holds for $Y$.
Consider a subvariety $Z \subseteq X$ for which $Z^{\Gm} \subseteq Y^{\Gm}$.
Let $z \in Z$ be a non-fixed point.
Then $p \midarrowX{z} q$ for some $p, q \in Z^{\Gm} \subseteq Y^{\Gm}$.
By hypothesis, we conclude $z \in Y$.
Thus $Z \subseteq Y$.
\end{proof}
\begin{Prop}[{\cite[Proposition 5.3]{buch}}]
If the \BB decomposition of $X$ is a stratification, then each \BB cell closure in $X$ is $\Gm$-convex.
This holds more generally for irreducible intersections of positive and negative \BB cell closures.
\end{Prop}
\begin{proof}
By \Cref{rmk:intersectability}, it suffices to deal with the case of \BB cell closures.
Let $p \in X^{\Gm}$.
We apply the criterion in \Cref{lm:check_orbits} to $\overline{X_p^+}$.
Suppose $q \midarrowX{x} r$ for some $x \in X \del X^{\Gm}$ and $q, r \in \overline{X_p^+}$.
The stratification assumption implies $X_r^+ \subseteq \overline{X_p^+}$.
Thus $x \in X_r^+ \subseteq \overline{X_p^+}$.
\end{proof}
Remarkably, if we are looking for \Gm-convex subvarieties, we need not look much beyond intersections of opposite \BB cell closures.
\begin{Prop}\label[Prop]{prop:convex_is_richardson}
Assume the \BB decomposition of $X$ is filterable.
Let $Y$ be a $\Gm$-stable \Gm-convex subvariety of $X$. Then $Y$ is an irreducible component of $\overline{X_p^+} \cap \overline{X_q^-}$ for some $p, q \in X^{\Gm}$.
\end{Prop}
\begin{proof}
We have $Y \subseteq \overline{X_{Y^{\up}}^+}$ and similarly $Y \subseteq \overline{X_{Y^{\down}}^-}$.
Since $Y$ is irreducible, it is contained in some irreducible component $Z$ of ${\overline{X_{Y^{\up}}^+} \cap \overline{X_{Y^{\down}}^-}}$.
By \Cref{thm:broken_trajectory}, $Z^{\up} \in \overline{X_{Y^{\up}}^+}$ implies ${Z^{\up} \middarrow Y^{\up}}$.
Similarly, $Y^{\up} \in Z$ implies $Y^{\up} \middarrow Z^{\up}$. Since $X$ is filterable, $\middarrow$ is a partial order (see \Cref{thm:filt_iff_acyclic}), so $Y^{\up} = Z^{\up}$.
Symmetrically, $Y^{\down} = Z^{\down}$.
For a general point $z \in Z$, we have $Y^{\down} = Z^{\down} \midarrowX{z} Z^{\up} = Y^{\up}$.
The $\Gm$-convexity of $Y$ forces $z \in Y$.
We conclude $Z \subseteq Y$.
Hence $Y = Z$.
\end{proof}
The following is the analogue for the convexity property of Propositions \ref{prop:filt_product} and \ref{prop:strat_product}.
\begin{Lm}\label[Lm]{lm:product_convex}
Let $X, Y$ be smooth complete $\Gm$-varieties with finite fixed loci.
Let $Z \subseteq X$ and $W \subseteq Y$ be $\Gm$-stable subvarieties.
Equipping $X \times Y$ with the diagonal $\Gm$-action, we have that $Z \times W$ is $\Gm$-convex if and only if both $Z$ and $W$ are.
\end{Lm}
\begin{proof}
We use \Cref{lm:Gm_product} and appeal to \Cref{lm:check_orbits}.
Suppose $Z$ and $W$ are both $\Gm$-convex.
If $(p_1, q_1) \midarrowX{(x,y)} (p_2, q_2)$ holds for some $(x, y) \in X \times Y$, $p_1, p_2 \in Z^{\Gm}$, and $q_1, q_2 \in W^{\Gm}$, then either $p_1 = x = p_2$ or $p_1 \midarrowX{x} p_2$.
In either case, $x \in Z$ by $\Gm$-convexity.
Similarly, $y \in W$.
So $(x,y) \in Z \times W$.

Now suppose for instance that $Z$ is \emph{not} $\Gm$-convex.
Choose $p_1, p_2 \in Z^{\Gm}$ and $x \in X \del Z$ such that $p_1 \midarrowX{x} p_2$.
Then $(x, q) \notin Z \times W$ but $(p_1, q) \midarrowX{(x, q)} (p_2, q)$ shows that this product is not $\Gm$-convex.
\end{proof}
In low dimensions, we have the following positive result.
\begin{Prop}\label[Prop]{prop:uGm_convex_small}
If $\dim X \leq 2$ then all its \BB cell closures are $\Gm$-convex.
\end{Prop}
\begin{proof}
If $\dim X = 0 \text{ or } 1$, the result is trivial since $X$ is either a point or $\mathbb{P}^1$.
Suppose $\dim X = 2$.
By \Cref{thm:smoo_sur_are_toric}, $X$ is a smooth projective toric variety and its $\Gm$-action is given by a suitable cocharacter of its dense torus.
So the result follows from \Cref{prop:twodimensions} below.
\end{proof}
We now specialize to the case of a smooth projective toric variety $X$ with torus $T$ and associated polytope $P$.
We assume that the choice of an admissible cocharacter $v$ has been made, equipping $X$ with the structure of a $\Gm$-variety.
As pointed out in \cite[Ex.\,5.6]{buch}, the cell closures $\overline{X_p^+}$ are $T$-convex.
The authors go on to ask whether they are moreover $\Gm$-convex.
The following lemma will enable us to inspect \BB cell closures for $\Gm$-convexity.
\begin{Lm}\label[Lm]{lm:face-condition}
Let $Y \subseteq X$ be a $T$-invariant subvariety.
Let $E$ be the face of $P$ such that $Y = \overline{O_E}$.
Then $Y$ is $\Gm$-convex if and only if for every face $F$ of $P$, if $F^\up, F^\down \in E$ then $F\subseteq E$.
\end{Lm}
\begin{proof}
Each non-fixed point $x \in X$ lies in the orbit $O_F$ for some face $F$ of $P$.
By \Cref{lm:Fup}, we have $F^{\down} \midarrowX{x} F^{\up}$.
The conclusion now follows from \Cref{lm:check_orbits}, and the standard dictionary for projective toric varieties and their polytopes.
\end{proof}
For convenience, we employ the following definition.
\begin{Df}
For a given choice of $v \in N$, if every \BB cell closure in $X$ is $\Gm$-convex we will say that $X$ is \emph{$v$-well-rounded}.
We will similarly speak of the $v$-well-roundedness of the polytope $P$.
When the choice of $v$ is clear, we will simply say that $X$ (or $P$) is well-rounded.
\end{Df}
The following analogue of \Cref{lm:toric_strat_inherit} shows that well-roundedness is inherited by $T$-stable subvarieties.
\begin{Lm}\label[Lm]{lm:inheritance}
Let $Y \subseteq X$ be a $T$-stable subvariety.
Let $Q$ be the face of $P$ such that $Y = \overline{\orb{Q}}$.
If $X$ is $v$-well-rounded then $Y$ is $\overline{v}$-well-rounded where $\overline{v} \deq v + N_Q$ is the image of $v$ in $N/N_Q$.
\end{Lm}
\begin{proof}
Let $q$ be a vertex of $Q$.
\Cref{lm:BB_inheritance} states that $P_q^+ \cap Q = Q_q^+$.
Let $F$ be a face of $Q$ such that $F^\up, F^\down \in Q_q^+$.
Then $F\subseteq P_q^+$ by the $\Gm$-convexity of $\overline{\orb{P_q^+}}$.
Hence $F \subseteq Q \cap P_q^+ = Q_q^+$.
This suffices by \Cref{lm:face-condition}.
\end{proof}
\begin{Prop}\label[Prop]{prop:twodimensions}
If $\dim X \leq 2$ then $X$ is $v$-well-rounded for every admissible $v \in N$.
\end{Prop}
\begin{proof}
Let $E = P_p^+$ for some vertex $p$ of $P$.
We show that $\overline{\orb{E}} = \overline{X_p^+}$ is $\Gm$-convex.
Let $F$ be any face of $P$ and suppose $F^{\up} \in E$ and $F^{\down} \in E$.
If $\dim F = 2$ then $F^{\up} = E^{\up}$, whence $\overline{\orb{E}} = X$ is trivially $\Gm$-convex.
On the other hand if $\dim F \leq 1$ then $F$ is the convex hull of $F^{\up}$ and $F^{\down}$ and is therefore contained in $E$.
This gives $\Gm$-convexity by \Cref{lm:face-condition}.
\end{proof}
In three dimensions, we have the following test for well-roundedness.
\begin{Prop}\label[Prop]{thm:3d-iff}
Assuming $\dim X = 3$, $X$ is $v$-well-rounded if and only if for every facet $F$ of $P$, one of the following holds:
\begin{enumerate}[(i)]
\item $F^{\up} = P^{\up}$.
\item The vertices $F^{\up}$ and $F^{\down}$ are not adjacent.
\item $F^{\up}$ and $F^{\down}$ lie on the intersection of $F$ with a facet $L$ of $P$ containing $P^{\up}$.
\end{enumerate}
\end{Prop} 
\begin{proof}
We first show the ``if'' direction.
Let $E \deq P_p^+$ for some vertex $p$ of $P$, i.e.\ $E$ is maximal for the property that $E^{\up} = p$.
Let $F$ be a face of $P$ with $F^{\up}, F^{\down} \in E$.
If $\dim F \leq 1$, then $F$ is the convex hull of $\{F^{\up}, F^{\down}\}$ and so $F \subseteq E$.
If $\dim F = 3$, then $F^{\up} = P^{\up} = E^{\up}$, so the maximality condition on $E$ implies $E = P = F$.
We may therefore assume $\dim F = 2$ in the following.

If (i) holds then $P^{\up} \in E$ implies $E = P \supseteq F$ by the maximality condition.
If (ii) holds, then $F \cap E$ is at least two-dimensional and thus $F \subseteq E$.
Suppose (iii) holds.
Note that $\dim E \leq 1$ would imply $E \subseteq F$ and $E^{\up} = F^{\up}$, violating maximality.
We must therefore have either $E = P$ or $E = F$ or $E = L$.
The last case gives $E^{\up} = L^{\up} = P^{\up}$, violating maximality.

For the converse, suppose all three conditions fail for some facet $F$ of $P$.
Let $E \neq F$ be the other facet of $P$ containing $F^{\up}, F^{\down}$.
Since $P^+_{E^\up} \neq P$ by the failure of (iii), we have $P^+_{E^\up} = E$.
Thus $\overline{X_{E^\up}^+}$ is a non-$\Gm$-convex \BB cell closure by \Cref{lm:face-condition}.
\end{proof}
We will use \Cref{thm:3d-iff} in the next subsection to construct examples of non-well-rounded polytopes.
We conclude this subsection by showing that smooth zonotopes are always well-rounded.
Recall that a zonotope is by definition a polytope all of whose faces are centrally symmetric.
In particular, every face of a zonotope is again a zonotope.
\begin{Prop}\label[Prop]{prop:zonotope}
If $P$ is a smooth lattice zonotope, then, for every admissible $v \in N$, each $T$-orbit closure
(in particular, each \BB cell closure) is $\Gm$-convex.
\end{Prop}
\begin{proof}
Let $E$ be any face of $P$ and suppose $F^{\up}, F^{\down} \in E$ for some face $F$ of $P$.
Since $F$ is centrally symmetric, the vertices $F^\up$ and $F^\down$ lie opposite one another on $F$.
Thus $F\cap E$ is a face of $F$ containing opposite points of $F$ and is therefore all of $F$.
That is, $F \subseteq E$.
The result now follows by appealing to \Cref{lm:face-condition}.
\end{proof}

\subsection{Pathologies}\label{subsec:bad_convex}
In this subsection, we exhibit various examples of non-$\Gm$-convexity.
We begin by producing an example of a non-$\Gm$-convex \BB cell closure.
This answers the question posed by Buch--Chaput--Perrin at the end of \cite[Example 5.6]{buch}.

\begin{figure}[ht]
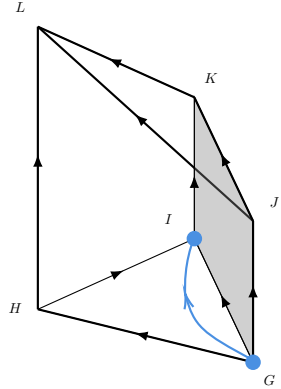

\centering
\includestandalone[width=0.25\textwidth]{with_arrows}
\caption{
A non-$\Gm$-convex \BB   cell closure $\overline{X_K^+}$, represented by the shaded facet.
Arrows indicate the $\Gm$-flow.
The curve in the face $GHI$ represents the $\Gm$-orbit of a point in the $T$-orbit corresponding to the facet $GHI$.
The $T$-fixed points $I$ and $G$---but not the $\Gm$-orbit itself---lie in $\overline{X_K^+}$.
}
\label{fig:intro}
\end{figure}

\begin{Ex}\label[Ex]{prop:firstexample}
Let $P$ be the polytope depicted in \Cref{fig:intro} in $\R^3$ with vertices
\[
G \deq (1,0,0);\; H \deq (0, 0, 0); \;  I \deq (0, 1, 0);\; J \deq (1,0,1);\; K \deq (0, 1, 1);\;  L \deq (0, 0, 2)
\]
Its facets are $GHI$, $JKL$, $HGJL$, $IHLK$, $GIKJ$.
One may verify that $P$ is smooth.
Thus its outer normal fan gives rise to a smooth projective toric variety $X$.
Let $v = (-1, 1, 2)$.
The values of $\inner{\cdot, v}$ at the above points are given respectively by
\[
\inner{G, v} = -1, \; \inner{H, v} = 0,\; \; \inner{I, v} = 1,\; \inner{J, v} = 1,\; \inner{K, v} = 3,\; \inner{L, v} = 4
\]
We have $GHI^{\up} = I$, $GHI^{\down} = G$ and, using \Cref{cor:BB_face}, $\overline{X_K^+} = \overline{\orb{GIKJ}}$.
Thus, $\overline{X_K^+}$ fails to be $\Gm$-convex by \Cref{lm:face-condition} since $GHI \nsubseteq GIKJ$.
\end{Ex}
\begin{Rmk}\label[Rmk]{rmk:convex_non_rev}
If in the above example we consider instead the vector $v' \deq -v = (1,-1,-2)$ then $X$ is $v'$-well-rounded.
This is to be contrasted with Corollaries \ref{cor:filtration_reversibility} and \ref{cor:strat_reversibility} which show that the filterability and stratification properties \emph{are} invariant under reversing the $\Gm$-action.
\end{Rmk}
The next result shows that the phenomenon in \Cref{prop:firstexample} is easy to engineer.
But first, it will be useful to recall the polytopal picture for toric blowups at smooth $T$-fixed points, which follows from the discussion in \cite[p.\,40]{fult93toric}.
Starting with a toric variety $X_0$ defined by a polytope $P_0$, to obtain a polytope $P_1$ representing the blowup of $X_0$ at the fixed point corresponding to a vertex $p \in P_0$ we truncate $P_0$ at $p$ by means of a hyperplane which meets each edge incident to $p$ at the nearest integral point to $p$.
To ensure that $P_1$ is combinatorially the truncation of $P_0$ at $p$, we may need to first replace $P_0$ by a suitable dilation.
If $P_1$ is obtained in this way from $P_0$---by dilating and truncating at $p$---we will say that it is a blowup of $P_0$ at $p$.
\begin{Prop}\label[Prop]{prop:blowup}
Suppose $X_0$ is a smooth projective toric variety of dimension at least 3 with dense torus $T$, corresponding to a polytope $P_0$.
Then the blowup $X$ of $X_0$ at two suitably chosen $T$-fixed points of $X_0$ fails to be $v$-well-rounded for some admissible cocharacter $v\colon\Gm \to T$.
\end{Prop}
\begin{proof}
By \Cref{lm:inheritance}, we may assume that $\dim X_0 = 3$.
If $P_0$ is a tetrahedron, then $X_0$ is $\PP^3$.
This case can be dealt with in a hands-on way.
Here, $X$ is defined as a toric variety by the polytope depicted in \Cref{fig:twiceblown}.

\begin{figure}[ht]
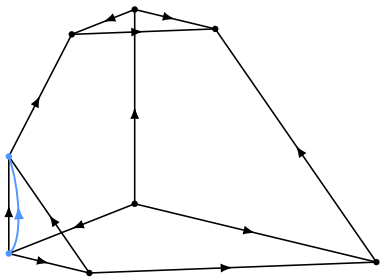

\centering
\includestandalone[width=.35\textwidth]{bad_richard}
\caption{A blowup of $\PP^3$ at two torus-fixed points.
This polytope has vertices $(0,0,0), (2,0,0), (0,0,2), (2,1,0), (2,0,1), (0,3,0), (1,0,2), (0,1,2)$. It fails to be $v$-well-rounded for $v = (1,2,3)$.
The bottom-left triangle fails the conditions in \Cref{thm:3d-iff}.
} 
\label{fig:twiceblown}
\end{figure}

Now suppose $P_0$ has at least five facets.
Then, choosing a vertex $p$ in $P_0$ arbitrarily, we can find an edge $L$ of $P_0$ such that $L$ and $p$ do not lie on a common facet of $P_0$.
Let $P$ be the successive blowup of $P_0$ at the vertex $p$ and at a vertex $\ell$ of $L$.
Let $\triangle_p, \triangle_\ell$ be the triangles in $P$ corresponding to the vertices $p$ and $\ell$.
Let $\tilde \ell$ be the vertex of $P$ lying on both $\triangle_\ell$ and the edge corresponding to $L$.
Then no edge of $\triangle_p$ lies on a common facet with $\tilde \ell$.
Choose $v$ to be an admissible outer normal vector at $\tilde \ell$.
Then $\triangle_p^{\up}$ and $\triangle_p^{\down}$ lie on a common edge of $\triangle_p$ and this edge does not lie on a facet containing $\tilde \ell$.
The claim now follows from \Cref{thm:3d-iff}.
\end{proof}
The polytope depicted in \Cref{fig:polyproof} has the remarkable property that for every admissible cocharacter $v$ it fails to be $v$-well-rounded.
This is proved in the following proposition.

\begin{figure}[hb]
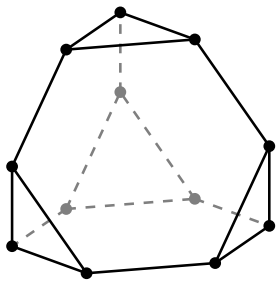

\centering
\includestandalone[width=.25\textwidth]{Pop_Projective_Space}
\caption{A polytope admitting no cocharacter making it well-rounded.}
\label{fig:polyproof}
\end{figure}

\begin{Prop}\label[Prop]{prop:tetrahedron}
Let $n \geq 3$.
If $X$ is the blowup of $\PP^n$ at all its $T$-fixed points then there is no admissible $v \in N$ making $X$ well-rounded.
\end{Prop}
\begin{proof}
By \Cref{lm:inheritance}, it suffices to prove this for $\mathbb{P}^3$.
Let $P$ be the blown-up tetrahedron depicted in \Cref{fig:polyproof}.
This polytope is obtained by truncating a suitable dilation of the convex hull of the standard basis of $\R^3$ and the origin.
$P$ is a polytope whose normal fan defines $X$.
The desired conclusion will be deduced from the following property of $P$: for any admissible $v \in N$, the vertices ${P}^{\up}$ and ${P}^{\down}$ do \emph{not} lie on a common facet of $P$.
Accepting this claim for now, we start by noting that $P^{\down}$ lies in a (unique) triangle $\triangle$.
The vertices $\triangle^{\up}$ and $\triangle^{\down} = P^{\down}$ lie in a common hexagon $H$.
The claim implies that $P^{\up} \notin \triangle$ and $P^{\up} \notin H$.
Thus \Cref{thm:3d-iff} (with $F = \triangle$) shows that $X$ is not $v$-well-rounded.

Now we proceed to prove the claim.
Note that any two vertices lying on a common facet of $P$ lie on a common hexagonal facet.
Suppose therefore that $P^{\up}$ and $P^{\down}$ lie on a hexagonal facet
$\hexagon$ of $P$.
Since $\hexagon$ is centrally symmetric, ${P}^{\up}$ and ${P}^{\down}$ must lie opposite one another on $\hexagon$.
However, it now suffices to observe that for every pair of opposite vertices on $\hexagon$, the two edges incident to these vertices and not lying on $\hexagon$ are parallel.
This precludes the simultaneous equalities $\hexagon^{\up} = P^{\up}$ and $\hexagon^{\down} = P^{\down}$.
\end{proof}
Our final construction will make use of the following definition.
\begin{Df}
Given a smooth projective toric variety $X_0$ defined by a lattice polytope $P_0$, we write $\Pop(X_0)$ for the toric variety that is the blowup of $X_0$ at all its $T$-fixed points.
Its associated polytope---denoted $\Pop(P_0)$---is the truncation (at all vertices) of a suitable dilation of $P_0$.
\end{Df}
We have just seen that if $Q$ is a smooth tetrahedron then $\Pop(Q)$ fails to be $v$-well-rounded for any admissible cocharacter $v$.
We will show that the same holds for \emph{most} three-dimensional smooth polytopes. (A higher-dimensional version can be extracted using \Cref{lm:inheritance}.)
\begin{Prop}\label[Prop]{prop:poponce}
Let $P_0$ be a three-dimensional smooth integral polytope.
Suppose that the vertices of $P_0$ do not all lie on one of two adjacent facets of $P_0$.
Then $P \deq \Pop(P_0)$ has the property that no admissible cocharacter $v \in N$ makes $P$ well-rounded.
\end{Prop}
\begin{proof}
We fix an admissible cocharacter $v \in N$ for $P$.
The vertex $P^{\up}$ of $P$ lies on a unique edge $L$ of $P_0$.
By hypothesis, there is a vertex $q$ of $P_0$ that does not lie on a common facet with $L$.
Let $\triangle$ be the triangle in $P = \Pop(P_0)$ corresponding to $q$.
No edge of $\triangle$ is on a facet containing $P^\up$.
Taking $F = \triangle$ in \Cref{thm:3d-iff} shows that $P$ is not $v$-well-rounded.
\end{proof}

\begin{figure}[ht]
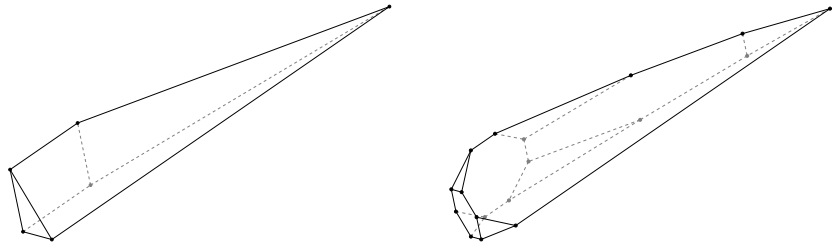

\centering
{\includestandalone[width=.35\textwidth]{pop_not_enough_1}}%
\qquad
{\includestandalone[width=.35\textwidth]{pop_not_enough_2}}%
\qquad
\caption{A polytope and its $\Pop$.}
\label{fig:badblowup}%
\end{figure}

\begin{Ex}
Consider the triangular prism depicted in \Cref{fig:badblowup}.
This polytope differs from the one in \Cref{prop:firstexample} only in that we set $L = (0,0,5)$.
Its image under $\Pop$---depicted on the right-hand side of \Cref{fig:badblowup}---is well-rounded for the vector $(-4,-3,-2)$.
This does not contradict \Cref{prop:poponce} since
any two quadrilateral facets of this polytope together contain all its vertices.
This shows that the technical requirement in \Cref{prop:poponce} is not wholly superfluous.
\end{Ex}

\section{Rigidity}\label{sec:rigidity}
Throughout this section, $X$ will be a smooth complete $T$-variety with finite fixed locus.
Our goal in this section is to answer \cite[Question 5.5]{buch} regarding the equivariant rigidity of \BB cell closures in $X$ for a choice of $\Gm \subseteq T$.
We define two related notions of equivariant homological rigidity.
\begin{Df}\label[Df]{df:weak_rigidity}
A $T$-stable subvariety $\Omega \subseteq X$ is weakly $T$-rigid if, whenever a $T$-stable subvariety $W$ in $X$ satisfies $[W] = c[\Omega]$ in the total equivariant Chow group $A_*^{T}(X)_{\QQ}$ for some $0 \neq c\in \QQ$, then $W = \Omega$ as subvarieties.
\end{Df}
\begin{Df}\label[Df]{df:strong_rigidity}
A $T$-stable subvariety $\Omega \subseteq X$ is strongly $T$-rigid if, whenever a $T$-invariant positive $\mathbb{Q}$-cycle $\alpha$ satisfies $\alpha = [\Omega]$ in the total equivariant Chow group $A_*^{T}(X)_{\QQ}$, then $\alpha = [\Omega]$ \emph{as a cycle}.
\end{Df}
\begin{Rmk}\label[Rmk]{rmk:differences}
We caution the reader that our notions differ somewhat from the notion of $T$-rigidity defined in \cite{buch}.
One difference is that they do not insist that $\Omega$ be irreducible.
In the irreducible case, their definition corresponds to our notion of strong $T$-rigidity with the added assumption that the coefficients of the prime cycles occurring in $\alpha$ are all equal.
They also work in the equivariant Chow cohomology ring, which is a purely aesthetic choice since the ambient variety is smooth.
The notion of weak $T$-rigidity appears sufficient in practice---our result for weak $T$-rigidity below is sufficient for their application of $T$-rigidity in \cite[Theorem 7.1]{buch} to a conjecture on curve neighborhoods in flag varieties.
\end{Rmk}
If the \BB decomposition on $X$ is filterable, then $A_*^T(X)$ is a free module over the equivariant Chow cohomology of a point by \cite[\textsection3.2, Corollary 1(iii)]{Brion_1997}.
In this case, moving from $A_*^T(X)$ to $A_*^T(X)_{\QQ}$ is purely a matter of convenience.
In the non-filterable case, \Cref{ex:zero_cycle} shows that a positive $\Gm$-invariant cycle may be zero in $A_*^{\Gm}(X)$.
Since this cycle may be added to any other without changing its equivariant class,
no subvariety can be strongly $\Gm$-rigid in such a case.
For this reason, we restrict all talk of strong $\Gm$-rigidity to the filterable case.

It was shown in \cite[Corollary 5.4]{buch} that the \BB cell closures of a $T$-variety (given by a choice of $\Gm \subseteq T$ with $X^{\Gm} = X^T$) are $T$-rigid under the assumption that the \BB decomposition is a stratification and that at every $T$-fixed point the $T$-weights of the tangent space lie strictly in a half-space of the character lattice.
The latter requirement is impossible when $T = \Gm$ and $X$ has dimension  at least 2.
In \Cref{thm:BB_rigidity} below, we show that $\Gm$-rigidity of \BB cells holds with no assumption other than filterability.
Recall that this condition is automatic for projective varieties.
Note also that in the setting of a $T$-variety, $\Gm$-rigidity for $T$-stable subvarieties is a stronger property than $T$-rigidity.
In the results that follow we specialize to the case $T = \Gm$.

To prove our results, we will need the following refinements of \cite[Proposition 3.1]{buch} and \cite[Corollary 3.2]{buch}.
Though our setting is slightly more general, our arguments are not essentially different.
\begin{Lm}\label[Lm]{fully_definite_non_zero}
Let $V$ be a finite-dimensional $\Gm$-representation with positive weights.
Let $\alpha \deq \sum_i a_i [W_i]$ be a positive pure-dimensional $\Gm$-invariant $\QQ$-cycle in $V$.
Then $\alpha \neq 0$ in $A_*^{\Gm}(V)_{\QQ}$.
\end{Lm}
\begin{proof}
Suppose first that $\alpha$ is a 0-cycle.
Then $\alpha = a [o]$ with $a > 0$ where $o$ is the origin of $V$.
The self-intersection formula \cite[Proposition 17.4.1]{Anderson_Fulton_2023} gives $\alpha|_o = c_{\text{top}}^{\Gm}(V) \frown a [o] \neq 0$ since the weights of $V$ are nonzero.

Now suppose the support of $\alpha$ has positive dimension.
Let $v_1, \dots, v_d$ be an eigenbasis for $V$ where $v_i$ has weight $\lambda_i$.
Let $\lambda =\prod_i \lambda_i$ and let $U = \mathbb{K}^d$ be the $\Gm$-representation defined by $t \cdot u = t^\lambda u$ for $u \in U$.
The morphism $\phi\colon V \to U$ given by
\[
\phi \left(\sum_{i}c_i v_i\right) = \sum_{i} c_i^{\lambda/\lambda_i} e_i
\]
is finite and $\Gm$-equivariant.
Moreover, the quotient map $\pi\colon U \del 0 \to \mathbb{P}(U)$ induces an isomorphism $\pi^*\colon  A_*(\mathbb{P}(U))_{\QQ} \to  A_*^{\Gm}(U\del 0)_{\QQ}$ by \cite[Theorem 2]{graham}.
The cycle
\[
\phi_*(\alpha)|_{U\del 0} = \sum_{i}a_i\deg(\phi|_{W_i})[\phi(W_i)\del 0]
\]
maps under $(\pi^*)^{-1}$ to the class of a nonempty effective $\QQ$-cycle in the ordinary total Chow group $A_*(\mathbb{P}(U))_{\QQ}$, so is nonzero in this group.
Hence ${\alpha \neq 0}$.
\end{proof}
\begin{Lm}\label[Lm]{lm:non_zero_to_point}
Let $p \in X^{\Gm}$.
Let $\alpha = \sum_{i} a_i [W_i]$ be a positive pure-dimensional $\Gm$-invariant $\QQ$-cycle in $X$ for which $p \in W_i \subseteq \overline{X_p^+}$ for all $i$.
Then the pullback $\alpha|_p$ of $\alpha$ to $A_*^{\Gm}(p)_{\QQ}$ under the regular embedding $p \hookrightarrow X$ is nonzero.
\end{Lm}
\begin{proof}
The ``basic construction'' of intersection theory gives that $[W_i]|_p$ is the pullback to $p$ of the class $[C_p^X W_i]$ of the cone of $W_i$ at $p$ in $T_p X$.
Since this pullback is an isomorphism, it is sufficient to show that $\beta \deq \sum_i a_i[C_p^X W_i]$ is nonzero in $A_*^{\Gm}(T_pX)_{\QQ}$.
Let $j\colon T_pX_p^+ \to T_pX$ be the inclusion.
Again using the self-intersection formula \cite[Proposition 17.4.1]{Anderson_Fulton_2023},
\[
j^*\beta = j^* \left(\sum_i a_i [C_p^{X} W_i]\right) = j^* j_*\left(\sum_i a_i [C_p^{X_p^+} W_i]\right) = c_{\text{top}}^{\Gm}(T_pX/T_p X_p^+)\frown \sum_i a_i [C_p^{X_p^+} W_i].
\]
Since the fixed locus of $X$ is finite, the weights of $T_pX$ are all nonzero, so $c_{\text{top}}^{\Gm}(T_pX/T_p X_p^+) \neq 0$.
Moreover, since $T_p X_p^+$ has all positive weights we have $\sum_i a_i [C_p^{X_p^+} W_i] \neq 0$ in $A^{\Gm}_*(T_pX_p^+)_{\QQ}$ by \Cref{fully_definite_non_zero}.
Hence the pullback $j^* \beta$ is nonzero in $A_*^{\Gm}(T_p X_p^+)_{\QQ}$.
So $\beta \neq 0$ in $A^{\Gm}_*(T_pX)_{\QQ}$.
\end{proof}
Recall from \Cref{df:filterable} that a filtering ordering for $X$ is given by an ordering of its fixed points $p_1, \dots, p_n$ such that each partial union $X_{p_1}^+ \sqcup \dots \sqcup X_{p_j}$ is closed.
For a subvariety $Z \subseteq X$, it will be useful to consider the maximum index $j$ for which $p_j \in Z$.
The following lemma shows that the localization of $Z$ at this fixed point in equivariant Chow homology is nonzero.
\begin{Lm}\label[Lm]{lm:effective_nonzero}
Assume the \BB decomposition of $X$ is filterable.
Let $\alpha$ be a positive $\Gm$-invariant $\QQ$-cycle in $X$.
Let $q$ be a fixed point in the support of $\alpha$ maximal with respect to a filtering order on $X^{\Gm}$.
Then the components of $\alpha$ containing $q$ are contained in $\overline{X_q^+}$ and $\alpha|_q \neq 0$ in $A_*^{\Gm}(q)_{\QQ}$.
\end{Lm}
\begin{proof}
Write 
\[
\alpha = \sum_i a_i [W_i] + \sum_j b_j [Z_j], \quad a_i, b_j > 0, \quad q \in W_i, q\notin Z_j.
\]
Let $q_i = W_i^\up$.
Then $q \in W_i \subseteq \overline{X_{q_i}^+}$, so by maximality $q = q_i$.
This shows the first claim.
For the second claim, we use  \Cref{lm:non_zero_to_point} to deduce
\[
\alpha|_q = \sum_i a_i [W_i]|_q + \sum_j b_j [Z_j]|_q = \sum_i a_i [W_i]|_q \neq 0.\qedhere
\]
\end{proof}
\begin{Cor}\label[Cor]{cor:nonvanishing}
If the \BB decomposition of $X$ is filterable, then the equivariant Chow homology class of any positive $\Gm$-invariant $\QQ$-cycle in $X$ is nonzero in $A_*^{\Gm}(X)_{\QQ}$.
\end{Cor}
\begin{Ex}\label[Ex]{ex:zero_cycle}
We note that \Cref{cor:nonvanishing} need not hold without the assumption of filterability of the \BB decomposition for $X$.
Here we give an example of a smooth, complete toric threefold with a positive torus-invariant cycle which is zero in $A_{*}^{\Gm}(X)$.
Our construction is a variation on \cite[Example 3.3.3]{JJ}.

Let $e_0, e_1, e_2$ be a basis for $\Z^3$ and let $e_0^*, e_1^*, e_2^*$ be the dual basis.
Let $\epsilon_i = -e_i$ and let $w_i = e_i - 3e_{i+1} - e_{i + 2}$ with indices viewed modulo 3.
Let $X$ be the toric variety given by the complete unimodular fan in $\R^3$ whose maximal cones are:
\begin{align*}
&\Cone(e_0, e_1, e_2), &  &\Cone(\epsilon_0, \epsilon_1, \epsilon_2), & &\{\Cone(e_i, e_{i+1}, w_i)\}_{i \in \Z\!/3\!\Z},\\
&\{\Cone(w_{i}, \epsilon_{i+1}, \epsilon_{i+2})\}_{i \in \Z\!/3\!\Z}, &  &\{\Cone(e_i, w_{i+2},\epsilon_{i+1})\}_{i \in \Z\!/3\!\Z}, & &\{\Cone(e_{i}, w_{i},\epsilon_{i+1})\}_{i \in \Z\!/3\!\Z}.
\end{align*}
For a ray generator $\rho$, we write $V(\rho)$ for the $T$-stable divisor in $X$ corresponding to $\rho$.
For ray generators $\rho, \rho'$ in the fan of $X$, we write $V(\rho, \rho')$ for $V(\rho)\cap V(\rho')$, the $T$-stable subvariety corresponding to the cone $\Cone(\rho, \rho')$.

We equip $X$ with a $\Gm$-structure using the admissible cocharacter $v = e_0 + e_1 + e_2$.
The $T$-character $\chi_i = e_{i+1}^* - e_{i+2}^*$ restricts to a rational function on the torus-invariant subvariety $V(e_i)$ which is a $T$-eigenfunction.
For our choice of $\Gm$, the rational function $\chi_i$ has weight $0$.
Thus \cite[Theorem 2.1]{Brion_1997} gives the relation $[\Div_{V(e_i)} \chi_i] = 0$ in $A^{\Gm}_*(X)$.
Using \mbox{\cite[Equation 1]{Int_on_toric}}, one computes
\[
[\Div_{V(e_i)} \chi_i] = [V(e_i, e_{i + 1})] - [V(e_i, e_{i + 2})] - 2 [V(e_i, w_i)] - [V(e_i, \epsilon_{i + 1})] - 2 [V(e_i, w_{i + 2})].
\]
Summing these relations for $i \in \{0,1,2\}$, we get in $A_*^{\Gm}(X)$ the equality
\[
0 = -\sum_{i} [\Div_{V(e_i)} \chi_i] = 2\sum_i [V(e_i, w_{i})]  + \sum_i [V(e_i, \epsilon_{i + 1})] + 2 \sum_i [V(e_i, w_{i+2})].
\]
\end{Ex}
\begin{Thm}\label[Thm]{thm:BB_rigidity}
Let $X$ be a smooth complete $\Gm$-variety with finite fixed locus whose \BB decomposition is filterable.
Then each \BB cell closure $\overline{X_p^+}$ is strongly $\Gm$-rigid.
\end{Thm}
\begin{proof}
Let $\alpha$ be a positive $\Gm$-invariant $\mathbb{Q}$-cycle in $X$ with $\alpha = [\overline{X_p^+}]$ in $A_*^{\Gm}(X)_{\QQ}$.
It follows from \Cref{cor:nonvanishing} that $\alpha$ is pure of dimension $\dim X_p^+$, since $A_*^{\Gm}(X)_{\QQ}$ is graded by dimension.
Let $q$ be a fixed point in the support of $\alpha$ maximal with respect to a filtering order on $X^{\Gm}$.
Write
\[
\alpha = \sum_{i} a_i [W_i] + \sum_j b_j [Z_j], \quad a_i> 0, b_j > 0 \quad q \in W_i, q\notin Z_j.
\]
By \Cref{lm:effective_nonzero} applied to $\alpha$ we have $\bigcup_i W_i \subseteq \overline{X_q^+}$ and $\alpha|_q \neq 0$.
So by hypothesis $[\overline{X_p^+}]|_q \neq 0$, whence  $q \in \overline{X_p^+}$.
By \Cref{lm:effective_nonzero} applied to the cycle $[\overline{X_p^+}]$, we have $[\overline{X_p^+}]|_p \neq 0$.
By hypothesis, also $\alpha|_p \neq 0$.
So $p$ lies in the support of $\alpha$.
By maximality of $q$, we conclude $q = p$ since $q \in \overline{X_p^+}$.
For each $i$, we have $W_i \subseteq \overline{X_p^+}$ and $\dim W_i = \dim X_p^+$, so $W_i = \overline{X_p^+}$.
Thus $\alpha$ as a cycle is
\[
\alpha = a [\overline{X_p^+}] + \sum_j b_j [Z_j].
\]
The equality $\alpha = [\overline{X_p^+}]$ in $A_*^{\Gm}(X)_{\QQ}$ gives
\[
(1- a)[\overline{X_p^+}] = \sum_j b_j [Z_j] \quad \text{in $A_*^{\Gm}(X)_{\QQ}$.}
\]
By definition, $p \notin Z_j$ for all $j$.
Thus $\sum_j b_j [Z_j]|_p = 0$.
Since $[\overline{X_p^+}]|_p\neq 0$, we conclude $1 = a$.
Then $\sum_j b_j [Z_j] = 0$ in $A_*^{\Gm}(X)_{\QQ}$.
By \Cref{cor:nonvanishing}, we deduce that $\sum_j b_j [Z_j]$ is the zero cycle.
So $\alpha = [\overline{X_p^+}]$, as desired.
\end{proof}
\begin{Rmk}
In light of \Cref{thm:BB_rigidity}, the abundance of examples of smooth projective toric varieties with non-\Gm-convex \BB cell closures in \Cref{subsec:bad_convex} answers the question implicit at the end of \cite[Example 4.4]{buch} about whether equivariant rigidity requires  convexity.
For comparison, see \Cref{ex:convexity_is_needed}, in which a failure of $\Gm$-convexity leads to the non-rigidity of an intersection of opposite \BB cell closures.
\end{Rmk}
We now show that the assumptions of projectivity and fully-definiteness in \cite[Theorem 4.3]{buch} are unneeded if one only wants weak rigidity.
\begin{Prop}\label[Prop]{prop:weak_rigidity_sufficiency}
Let $Y \subseteq X$ be a $\Gm$-stable subvariety.
If $Y$ is $\Gm$-convex, then it is weakly $\Gm$-rigid.
\end{Prop}
\begin{proof}
Assume $Z \subseteq X$ is a $\Gm$-stable subvariety with $[Z] = c[Y]$ in $A^{\Gm}_*(X)_{\QQ}$ for some $c \neq 0$.
Then $Z \subseteq \overline{X_{Z^{\up}}^+}$, whence $[Z]|_{Z^{\up}} \neq 0$ by \Cref{lm:non_zero_to_point}.
By hypothesis, also $[Y]|_{Z^{\up}} \neq 0$.
Hence $Z^{\up} \in Y$.
Similarly, ${Z^\down \in Y}$.
The $\Gm$-orbit of a general point in $Z$ has fixed points $Z^{\up}$ and $Z^{\down}$, so is contained in $Y$ by the $\Gm$-convexity assumption.
Thus $Z \subseteq Y$.
Since $A_*^{\Gm}(X)_{\QQ}$ is graded by dimension, the equality $[Z] = c[Y]$ implies $\dim Z = \dim Y$.
Since $Z \subseteq Y$, this forces $Z = Y$.
\end{proof}
Since the property of $\Gm$-convexity is closed under taking irreducible intersections, this gives a criterion for weak $\Gm$-rigidity of \BB cell closures.
\begin{Cor}\label[Cor]{cor:convexity_rigidity}
If the positive and negative \BB cell closures of $X$ are $\Gm$-convex, then the intersection $\overline{X_p^+} \cap \overline{X_q^-}$ is weakly $\Gm$-rigid for $p, q \in X^{\Gm}$ whenever it is irreducible.
\end{Cor}
\begin{Ex}\label[Ex]{ex:convexity_is_needed}
That the convexity assumption in \Cref{cor:convexity_rigidity} is not superfluous can be seen by analyzing the smooth projective toric variety $X$ whose polytope is pictured in \Cref{fig:twiceblown}.
It is the blowup of $\mathbb{P}^3$ at two fixed points.
The choice of cocharacter $v = (1,2,3)$ induces $\midddarrow$ relations as indicated there.
The curve inside the bottom left triangle represents the $\Gm$-orbit closure $C'$ of a general point in the corresponding exceptional divisor.
The curve $C'$ degenerates to the $T$-stable curve $C$ corresponding to the edge $(2,0,0)\midddarrow(2,0,1)$.
It follows that $[C'] = k[C]$ in $A_*^{\Gm}(X)$ for some positive integer $k$.
Thus $C$ is not weakly $\Gm$-rigid.
We conclude by observing that $C = \overline{X_{(2,0,0)}^-} \cap \overline{X_{(1,0,2)}^+}$ is an intersection of opposite \BB cell closures.
\end{Ex}
We can improve weak $\Gm$-rigidity to strong if we assume that the \BB decomposition is a stratification.
The following result may be used to show that Richardson varieties in flag varieties are strongly $\Gm$-rigid for the $\Gm$-action induced by a general cocharacter (cf.\ \cite[Lemma 6.1]{buch}).
This improves upon \cite[Theorem 6.3]{buch}, which proves $T$-rigidity.
\begin{Thm}\label[Thm]{thm:strong_richardson_rigidity}
Let $X$ be a smooth complete $\Gm$-variety with finite fixed locus.
Suppose the \BB decomposition of $X$ is a stratification.
Then the intersection $\overline{X_p^+} \cap \overline{X_q^-}$ is strongly $\Gm$-rigid for $p, q \in X^{\Gm}$ whenever it is irreducible.
\end{Thm}
\begin{proof}
First note that the \BB decomposition for $X$ is filterable (see \Cref{rmk:strat_implies_filt}).
Let $\alpha$ be a positive $\Gm$-invariant $\mathbb{Q}$-cycle in $X$ with $\alpha = [\overline{X_p^+} \cap \overline{X_q^-}]$ in $A_*^{\Gm}(X)_{\QQ}$.
It follows from \Cref{cor:nonvanishing} that $\alpha$ is pure of dimension equal to $\dim (\overline{X_p^+} \cap \overline{X_q^-})$, since $A_*^{\Gm}(X)_{\QQ}$ is graded by dimension.
Write
\[
\alpha = \sum_i a_i [W_i] + \sum_j b_j [Z_j], \quad a_i >0, b_j > 0 \quad W_i \subseteq \overline{X_p^+} , Z_j \nsubseteq \overline{X_p^+}.
\]
Assume there is at least one $Z_j$.
Let $\ell$ be a fixed point of $\bigcup_j Z_j$ maximal with respect to a filtering order on $X^{\Gm}$.
Fix $k$ such that $\ell \in Z_k$.
Then $Z_k \subseteq \overline{X_{\ell}^+}$ by the first part of \Cref{lm:effective_nonzero}.
If we had $\ell \in \overline{X_p^+}$, then also $Z_k \subseteq \overline{X_{\ell}^+} \subseteq \overline{X_p^+}$ by (i)$\Rightarrow$(ii) in \Cref{thm:dim_strat}.
This would contradict the definition of $Z_k$.
So $\ell \notin \overline{X_p^+}$.
By the second part of \Cref{lm:effective_nonzero},
\[
\alpha|_{\ell} = \sum_i a_i [W_i]|_{\ell} + \sum_j b_j[Z_j]|_{\ell} = \sum_j b_j[Z_j]|_{\ell} \neq 0. 
\]
Since $\alpha|_{\ell} = [\overline{X_p^+} \cap \overline{X_q^-}]|_{\ell} = 0$, this is a contradiction.

From this discussion, we deduce that the support of $\alpha$ is contained in $\overline{X_p^+}$.
By an analogous argument, one shows that its support is contained also in $\overline{X_q^-}$.
(Note that the negative \BB decomposition is a stratification by \Cref{cor:strat_reversibility}.)
Each irreducible component of $\alpha$ is a subvariety of $\overline{X_p^+} \cap \overline{X_q^-}$ of dimension $\dim(\overline{X_p^+} \cap \overline{X_q^-})$.
Thus as a cycle $\alpha = a[\overline{X_p^+} \cap \overline{X_q^-}]$ for some $a$.
The equality $a = 1$ follows from \Cref{cor:nonvanishing}.
\end{proof}

\section{Acknowledgments}

The authors are grateful to the Institute for Advanced Study for having provided the forum for discussion which brought these questions to our attention.
We would like to thank Harold Blum, Sergio Cristancho, Tuong Le, Leonardo Mihalcea, Botong Wang, and Shouda Wang for helpful discussions. 
We are particularly grateful to 
Anders Buch and Matt Larson for their thoughtful feedback.
We thank Mateusz Micha\l ek, Leonid Monin, and Botong Wang for a productive correspondence during the final stages of our preparation of this manuscript. We are also grateful to Seonjeong Park and JiSun Huh for their openness to discussing their research surrounding generalized Bott towers.
In the course of testing various hypotheses, the numerous examples of smooth convex polytopes in \cite{smooth_classification} were a valuable resource. 

\appendix

\section{Proof of Theorem \ref{thm:smoo_sur_are_toric}}
\label[appendix]{append2}
Here we prove \Cref{thm:smoo_sur_are_toric} based on the ideas in \cite[\textsection4.1]{Orlik}. We restate it as \Cref{thm:smoo_sur_are_toric2}.
\begin{Thm}\label{thm:smoo_sur_are_toric2}
Let $X$ be a smooth complete $\Gm$-surface with finite fixed locus.
Then $X$ admits the structure of a toric variety from which the $\Gm$-action may be recovered via some admissible cocharacter.
\end{Thm}
\begin{proof}
Note that any such surface must be projective by Zariski's criterion ({\cite[Theorem 1.28]{smoo_sur_proj}}, \cite[Corollary II.2.6]{zariski}) and rational because it contains a dense \BB cell.

Suppose first that $X$ does not contain a curve isomorphic to $\mathbb{P}^1$ with self-intersection number $-1$.
By Nagata's classification of minimal rational surfaces \cite[Theorem 2]{Nagata}, $X$ must be isomorphic as a variety to either $\mathbb{P}^2$ or a Hirzebruch surface. In either case, the automorphism group $\Aut X$ is a linear algebraic group of rank 2. Concretely, the identity component of $\Aut X$ is
$\PGL_3$ in the 
case of $\mathbb{P}^2$, 
the product
$\PGL_2 \times \PGL_2$
in the case of $\mathbb{P}^1 \times \mathbb{P}^1$, and a semidirect product 
$\mathbb{G}_a^{k+1} \rtimes \GL_2/\mu_{k}$
in the case of higher Hirzebruch surfaces (see \cite[Remarks 2.4.3, 2.4.4]{aut_group}).\footnote{The quotient $GL_2/\mu_k$ must be interpreted appropriately when $k$ is divisible by the characteristic of $\mathbb{K}$. Note that the proof of \cite[Proposition 4.1]{Orlik} incorrectly states that these automorphism groups are instead all isomorphic to 
$\PGL_2 \times \Gm$.} (Here $\mathbb{G}_a$ is the additive group over $\mathbb{K}$ and $\mu_k \subseteq \GL_2$ is the subgroup of $k$-torsion scalar matrices.) 
The $\Gm$-action on $X$ is induced by a morphism $\Gm \to \Aut X$.
From the fact that every torus in a linear algebraic group is contained in a maximal torus, we see that the
image of $\Gm \to \Aut X$ lies inside a $\Gm^2$ subtorus. Since $\dim X = \dim \Gm^2$ and $\Gm^2$ acts (scheme-theoretically) effectively on $X$, this is a toric structure extending the $\Gm$-action on $X$.

If such a curve $C$ does exist, then it must be $\Gm$-stable \cite[Proposition 1.9]{Orlik}.
By Castelnuovo's criterion \cite[Theorem V.5.7]{Hartshorne}, the curve $C$ may be blown down to a point via a morphism $\pi: X \to Y$ to a smooth projective surface.
By \cite[\href{https://stacks.math.columbia.edu/tag/0AY8}{Tag 0AY8}]{stacks-project}, we have $\pi_*(\mathcal{O}_X) = \mathcal{O}_Y$.
So by Blanchard's lemma \cite[Proposition 4.2.1]{blanchard}, there is a unique $\Gm$-action on $Y$ making the blow-down morphism equivariant.
The point $\pi(C)$ must then be a $\Gm$-fixed point of $Y$.

It follows from \cite[Proposition V.3.2]{Hartshorne} that contraction of curves decreases the Picard rank.
Since the Picard rank is finite to begin with by the N\'eron-Severi theorem,
repeatedly blowing down curves $C \cong \mathbb{P}^1$ with $C^2 = -1$ in our $\Gm$-surface eventually leads to one---call it $\check X$---that contains no such curves.
This is precisely the case handled above.

In light of the above discussion, it suffices to show that if $X, Y$ are smooth complete $\Gm$-surfaces with finite fixed loci and $\pi: X \to Y$ is the blowup of a $\Gm$-fixed point of $Y$ then the theorem holds for $X$ if it holds for $Y$. 
Let $T \subseteq Y$ be a dense torus such that the $\Gm$-action on $Y$ is induced by a morphism 
$\Gm \to T$. Since this $\Gm$-action has finite fixed locus, it is admissible. So $Y^{T} = Y^{\Gm}$.
In particular, $X$ is obtained from $Y$ by blowing up a $T$-fixed point.
Thus $X$ inherits a toric structure extending its $\Gm$-structure. 
\end{proof}
\begin{Rmk}
Whereas for smooth complete $\Gm$-surfaces with finite fixed locus the $\Gm$-action lifts to a toric structure, this fails for threefolds.
For instance, the smooth projective $\Gm$-variety in \Cref{ex:transversality_failure} cannot be toric.
This is because the \BB cell closures in a toric variety are torus-orbit closures and intersections of orbit closures in a toric variety are always reduced (cf.\ \cite[\textsection 3.1, \textsection 5.1]{fult93toric}).
\end{Rmk}

\section{Proof of Theorem \ref{thm:aff_prod}}
\label[appendix]{append1}
Our proof of \Cref{thm:aff_prod} is inspired by its analogue \cite[Lemma 2.5]{coxeter}.
\begin{Lm}\label[Lm]{lm:3d_possibilities}
Let $P$ be a simple three-dimensional polytope whose facets are all either triangles or parallelograms.
Then $P$ is either a simplex, a triangular prism, or a parallelepiped.
\end{Lm}
\begin{proof}
Write $V, E, F$ for the number of vertices, edges and facets of $P$, respectively.
Let $F_3$ be the number of triangles and $F_4$ the number of parallelograms.
We have $F = F_3 + F_4$ and $2E = 3F_3 + 4F_4$.
Also $3V = 2E$ since $P$ is simple.
Euler's formula gives $V - E + F = 2$.
These combine to give $3F_3 + 2F_4 = 12$. The possibilities for the pair $(F_3, F_4)$ are then $(0, 6), (2, 3)$, and $(4, 0)$.
The listed polytopes are the unique simple polytopes with these parameters.
\end{proof}
In \Cref{lm:vertex_eq_relation} and \Cref{prop:general}, we will repeatedly use the fact that in a simple polytope, any set of $k$ edges incident to a common vertex spans a $k$-dimensional face.
\begin{Lm}\label[Lm]{lm:vertex_eq_relation}
Let $P$ be a simple convex polytope whose two-dimensional faces are either triangles or parallelograms.
Let $v$ be a vertex of $P$.
For edges $e$ and $f$ incident to $v$, we write $e \sim_v f$ to mean that
either $e = f$ or the 2-face spanned by $e$ and $f$ is a triangle.
This relation is an equivalence relation.
\end{Lm}
\begin{proof}
Suppose $e,f,g$ are distinct edges incident to $v$.
If $e\sim_v f$ and $f \sim_v g$, the 3-face of $P$ spanned by $e, f, g$ is a simple three-dimensional polytope whose facets are all triangles or parallelograms.
By \Cref{lm:3d_possibilities}, it must be a simplex, as it has two adjacent triangles.
Thus $e \sim_v g$.
\end{proof}
We now proceed with the main result of this section.
We split up the proof of \Cref{thm:aff_prod} into two parts: the general case is \Cref{prop:general} and the smooth case is handled by \Cref{prop:unimodular}.
\begin{Prop}\label[Prop]{prop:general}
If a simple convex polytope $P$ has all two-dimensional faces triangles or parallelograms, then $P$ is affinely equivalent to a product of simplices.
\end{Prop}
\begin{proof}
Let $d \deq \dim P$.
Let $\Sigma_P^j$ denote the $j$-skeleton of $P$, i.e.\ the union of all $j$-dimensional faces of $P$.
We may assume that $P$ lies in $\R^d$ and contains the origin $o$ as a vertex.
Consider the set $E_o$ of edges of $P$ incident to $o$.
By \Cref{lm:vertex_eq_relation}, $E_o$ is partitioned into equivalence classes by $\sim_o$.
If there is only one class, then the closed neighborhood of $o$ in the graph $\Sigma^1_P$ forms the $1$-skeleton of a $d$-dimensional simplex.
Since $P$ is simple, it must coincide with this simplex.

Now suppose $\sim_o$ has at least two equivalence classes of edges.
Let $\Gamma$ be one such class and let $J$ be the face of $P$ that it spans.
$J$ is a simplex.
Let $C$ be the face of $P$ spanned by all edges incident to $o$ not in $\Gamma$.
By induction on dimension, we may assume that $C$ is affinely equivalent to a product of simplices.
Since $P$ is simple, we have ${\mathbb{R}^d=\Span(C) \oplus \Span(J)}$.
We will show that $P$ is the Minkowski sum $C + J$, hence affinely equivalent to $C \times J$.
We will achieve this by showing the inclusion $\Sigma^{1}_{C + J} \subseteq \Sigma^1_P$.
Since both of these 1-skeleta are connected $d$-regular graphs, this will imply that $P$ and $C + J$ have identical 1-skeleta, and must therefore coincide.

Still using $+$ to denote Minkowski sum, we have $\Sigma^1_{C+ J} = (\Sigma_C^1 + \Sigma_J^0) \cup (\Sigma_C^0 + \Sigma_J^1)$.
We first verify $\Sigma_C^0 + \Sigma_J^1 \subseteq \Sigma_P^1$ and thereafter that $\Sigma_C^1 + \Sigma_J^0 \subseteq \Sigma_P^1$.

\begin{enumerate}
\item 
We want to show that $\Sigma^1_J + w$ lies in the $1$-skeleton of $P$ for all vertices $w$ of $C$.
To do so, we prove that for every vertex $w$ of $C$ the following pair of conditions holds:
\begin{enumerate}[(a)]
\item $\Sigma_J^1 + w \subseteq \Sigma_P^1$.
\item The translates $e + w$ of the edges $e \in \Gamma$ form a $\sim_{w}$-equivalence class at the vertex $w$.
\end{enumerate}

The proof proceeds by traversing from vertex to vertex in $C$.
The statements are clearly true for $w = o$.
It suffices to show that if they hold at a vertex $w_1$ of $C$, then they hold for any adjacent vertex $w_2$ in $C$.
We place ourselves in this setting.
\begin{figure}[ht]
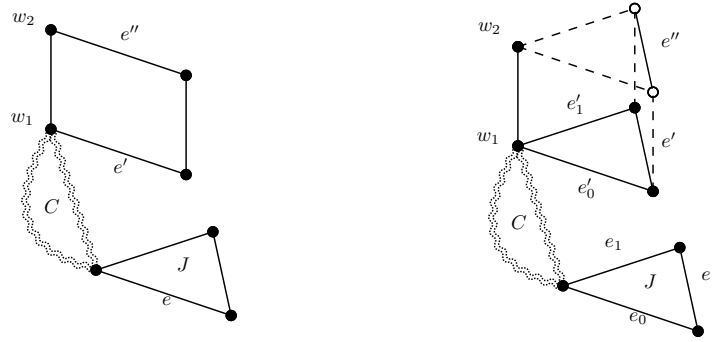

\centering\includestandalone[width=.64\textwidth]{product_part1}
\caption{Transporting $e \in \Sigma_J^1$ in cases (i) $e \in \Gamma$ and (ii) $e \notin \Gamma$.} 
\label{fig:product_part1}
\end{figure}

We begin by proving (a) for $w_2$.
Let $e$ be an edge of $\Sigma_J^1$.
By assumption, $e' \deq e + w_1$ lies in $\Sigma_P^1$.
We consider two cases.
If $e \in \Gamma$---see the left side of \Cref{fig:product_part1}---then $e' \nsim_{w_1} w_1w_2$ by the assumption of (b) for $w_1$.
The 2-face of $P$ spanned by the edges $e'$ and $w_1w_2$ is thus a parallelogram.
So the edge $e'' \deq e + w_2$ lies in $\Sigma^1_P$.
If $e \notin \Gamma$---see the right side of \Cref{fig:product_part1}---then there exists a unique pair of edges $e_0, e_1 \in \Gamma$ such that $e_0, e_1, e$ form a triangle in $J$.
Write $e' \deq e + w_1$, $e_0' \deq e_0 + w_1$ and $e_1' \deq e_1 + w_1$.
Then $e_0', e_1', e'$ form a triangle in $\Sigma_P^1$ by the assumption of (a) for $w_1$.
By the assumption of (b) for $w_1$, we have $e_0' \nsim_{w_1} w_1 w_2$.
\Cref{lm:3d_possibilities} implies that the 3-face of $P$ spanned by $e_0'$, $e_1'$ and $w_1w_2$ is a triangular prism, as depicted in the right side of \Cref{fig:product_part1}.
It follows that the translate $e'' \deq e + w_2$ is an edge of $P$. 
\begin{figure}[ht]
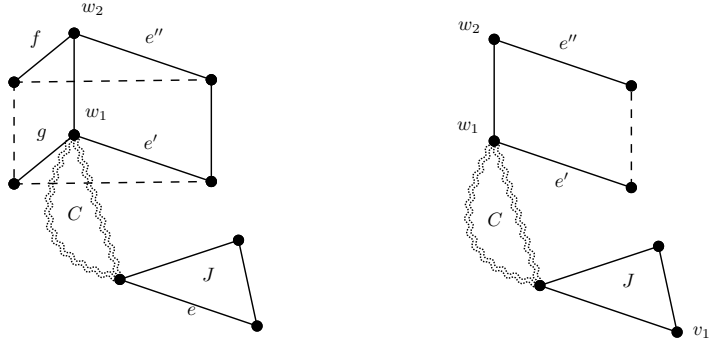

\centering\includestandalone[width=.64\textwidth]{product_part1_inductive}
\caption{
On the left: the validity of (b).
On the right: transporting $ov_1$.
} 
\label{fig:product_part1_inductive}
\end{figure}

The proof of the second case above also shows that for edges $e_0, e_1 \in \Gamma$, the corresponding edges $e_0'' \deq e_0 + w_2$ and $e_1'' \deq e_1 + w_2$ span a triangular 2-face at $w_2$.
That is, $e_0 \sim_{o} e_1$ implies $e_0'' \sim_{w_2} e_1''$.
The validity of (a) for $w_2$ combined with the simplicity of $P$ implies that all edges of $P$ incident to $w_2$ are either translates of edges in $\Gamma$ or lie in $C$.
It therefore remains only to show that if $e$ is an edge of $\Gamma$ then the edge $e'' \deq e + w_2$ satisfies $e'' \nsim_{w_2} f$ for all edges $f$ of $C$ incident to $w_2$.
Suppose instead that $e'' \sim_{w_2} f$ for some $f \in \Sigma_C^1$---see the left side of \Cref{fig:product_part1_inductive}. As observed in the proof of (a), we know that $e''\nsim_{w_2} w_2 w_1$.
The 3-face spanned by the edges $e'', f, w_2w_1$ contains both a triangle and a parallelogram and must therefore be a triangular prism, by \Cref{lm:3d_possibilities}.
The picture is as depicted in the left side of  \Cref{fig:product_part1_inductive}. In particular, this face contains an edge $g$ parallel to $f$ and incident to $w_1$.
Since both $f$ and $w_1$ lie in $C$, so too must $g$.
On the other hand, the setup forces $g \sim_{w_1} e'$ where $e' \deq e + w_1$.
By the assumption of (b) for $w_1$, this forces $g \in \Gamma + w_1$.
This is a contradiction because $g$ must lie in $C$.
\item We now verify $\Sigma_C^1 + \Sigma_J^0 \subseteq \Sigma_P^1$.
That is, for any vertex $v_1$ of $J$ (we may assume $v_1 \neq o$), and any edge $w_1 w_2$ of $C$, we show that $w_1 w_2 + v_1 \in \Sigma^1_P$.
The situation is depicted in the right side of \Cref{fig:product_part1_inductive}.
That is, we must verify that the edge $w_1w_2$ may be \emph{transported} along the vector $ov_1$.
By part (1), we know that $\Sigma_P^1$ contains the edges $e' \deq ov_1 + w_1$ and $e'' \deq ov_1 + w_2$.
These are edges adjacent to $w_1$ and $w_2$ respectively.
The 2-face at $w_1$ spanned by $w_1w_2$ and $e'$ must contain the parallel edge $e''$.
This face must therefore be a parallelogram.
In particular, $w_1w_2 + v_1$ is an edge of $\Sigma^1_P$.\qedhere
\end{enumerate}
\end{proof}
\begin{Prop}\label[Prop]{prop:unimodular}
If $P$ is a smooth lattice polytope affinely equivalent to a product of simplices then it is unimodularly equivalent to a product of smooth simplices.
\end{Prop}
\begin{proof}
After a unimodular transformation, we may assume that the origin $o$ is a vertex of $P$ and that the edges incident to $o$ lie on the positive coordinate axes.
Choose an affine isomorphism $P \cong \Delta_1 \times \dots \times \Delta_n$ where each $\Delta_i$ is a simplex.
The vertex $o$ in $P$ corresponds to a vertex $(o_1, \dots, o_n)$ of the product. Let $T_1, \dots, T_n$ be the faces of $P$ corresponding to the subproducts
\[
T_i \quad\leftrightsquigarrow\quad  \{o_1\} \times \cdots \times \{o_{i-1}\} \times \Delta_i \times \{o_{i+1}\} \times \cdots \times \{o_n\}.
\]
Then $P$ is the Minkowski sum of $T_1, \dots, T_n$ since the analogous statement is true (up to translation) for $\Delta_1 \times  \dots \times \Delta_n$.
Seeing as the linear spans of $T_1, \dots, T_n$ are coordinate subspaces in a direct sum, $P$ is in fact the product $T_1 \times \dots \times T_n$.
Since $P$ is smooth, each $T_i$ is  a smooth simplex.
\end{proof}

\printbibliography
\end{document}

%% file: permutohedron.tex
\begin{tikzpicture}[x=2.1cm,y=2.1cm]
\coordinate (P0) at (-0.63864,-1.14626);
\coordinate (P1) at (-1.42898,-1.16952);
\coordinate (P2) at (-0.31932,-0.09415);
\coordinate (P3) at (-1.90000,-0.14066);
\coordinate (P4) at (-0.79034,0.93470);
\coordinate (P5) at (-1.58068,0.91145);
\coordinate (P6) at (0.15170,-1.76165);
\coordinate (P7) at (-0.63864,-1.78490);
\coordinate (P8) at (0.79034,0.34257);
\coordinate (P9) at (-1.58068,0.27282);
\coordinate (P10) at (0.31932,1.37143);
\coordinate (P11) at (-1.26136,1.32492);
\coordinate (P12) at (1.26136,-1.32492);
\coordinate (P13) at (-0.31932,-1.37143);
\coordinate (P14) at (1.58068,-0.27282);
\coordinate (P15) at (-0.79034,-0.34257);
\coordinate (P16) at (0.63864,1.78490);
\coordinate (P17) at (-0.15170,1.76165);
\coordinate (P18) at (1.58068,-0.91145);
\coordinate (P19) at (0.79034,-0.93470);
\coordinate (P20) at (1.90000,0.14066);
\coordinate (P21) at (0.31932,0.09415);
\coordinate (P22) at (1.42898,1.16952);
\coordinate (P23) at (0.63864,1.14626);

% Faces
\fill[black!4, opacity=0.18] (P2) -- (P4) -- (P10) -- (P8) -- cycle;
\fill[black!4, opacity=0.18] (P12) -- (P6) -- (P0) -- (P2) -- (P8) -- (P14) -- cycle;
\fill[black!4, opacity=0.18] (P4) -- (P2) -- (P0) -- (P1) -- (P3) -- (P5) -- cycle;
\fill[black!4, opacity=0.18] (P14) -- (P8) -- (P10) -- (P16) -- (P22) -- (P20) -- cycle;
\fill[black!4, opacity=0.18] (P0) -- (P1) -- (P7) -- (P6) -- cycle;
\fill[black!4, opacity=0.18] (P16) -- (P10) -- (P4) -- (P5) -- (P11) -- (P17) -- cycle;
\fill[black!4, opacity=0.18] (P12) -- (P14) -- (P20) -- (P18) -- cycle;
\fill[black!24, opacity=0.95] (P3) -- (P5) -- (P11) -- (P9) -- cycle;
\fill[black!24, opacity=0.95] (P18) -- (P12) -- (P6) -- (P7) -- (P13) -- (P19) -- cycle;
\fill[black!12, opacity=0.95] (P16) -- (P17) -- (P23) -- (P22) -- cycle;
\fill[black!24, opacity=0.95] (P7) -- (P1) -- (P3) -- (P9) -- (P15) -- (P13) -- cycle;
\fill[fill={rgb,255:red,92;green,142;blue,220}, draw=none, opacity=0.55] (P22) -- (P20) -- (P18) -- (P19) -- (P21) -- (P23) -- cycle;
\fill[black!12, opacity=0.95] (P21) -- (P15) -- (P9) -- (P11) -- (P17) -- (P23) -- cycle;
\fill[black!15, opacity=0.95] (P13) -- (P15) -- (P21) -- (P19) -- cycle;

% Hidden edges and arrows
\draw[line width=0.55pt, dash pattern=on 2.2pt off 1.5pt, draw=black!40, line cap=round] (P0) -- (P1);
\draw[line width=0.55pt, dash pattern=on 2.2pt off 1.5pt, draw=black!40, line cap=round] (P0) -- (P2);
\draw[line width=0.55pt, dash pattern=on 2.2pt off 1.5pt, draw=black!40, line cap=round] (P0) -- (P6);
\draw[line width=0.55pt, dash pattern=on 2.2pt off 1.5pt, draw=black!40, line cap=round] (P2) -- (P4);
\draw[line width=0.55pt, dash pattern=on 2.2pt off 1.5pt, draw=black!40, line cap=round] (P2) -- (P8);
\draw[line width=0.55pt, dash pattern=on 2.2pt off 1.5pt, draw=black!40, line cap=round] (P4) -- (P5);
\draw[line width=0.55pt, dash pattern=on 2.2pt off 1.5pt, draw=black!40, line cap=round] (P4) -- (P10);
\draw[line width=0.55pt, dash pattern=on 2.2pt off 1.5pt, draw=black!40, line cap=round] (P8) -- (P10);
\draw[line width=0.55pt, dash pattern=on 2.2pt off 1.5pt, draw=black!40, line cap=round] (P8) -- (P14);
\draw[line width=0.55pt, dash pattern=on 2.2pt off 1.5pt, draw=black!40, line cap=round] (P10) -- (P16);
\draw[line width=0.55pt, dash pattern=on 2.2pt off 1.5pt, draw=black!40, line cap=round] (P12) -- (P14);
\draw[line width=0.55pt, dash pattern=on 2.2pt off 1.5pt, draw=black!40, line cap=round] (P14) -- (P20);
\draw[line width=0.6pt, draw=black!50, dash pattern=on 2.2pt off 1.5pt, -{Latex[length=4.5pt,width=5pt]}] (-1.13376,-1.16083) -- (-0.93385,-1.15495);
\draw[line width=0.6pt, draw=black!50, dash pattern=on 2.2pt off 1.5pt, -{Latex[length=4.5pt,width=5pt]}] (-0.44994,-0.52452) -- (-0.50802,-0.71590);
\draw[line width=0.6pt, draw=black!50, dash pattern=on 2.2pt off 1.5pt, -{Latex[length=4.5pt,width=5pt]}] (-0.16456,-1.51539) -- (-0.32237,-1.39252);
\draw[line width=0.6pt, draw=black!50, dash pattern=on 2.2pt off 1.5pt, -{Latex[length=4.5pt,width=5pt]}] (-0.59646,0.51120) -- (-0.51320,0.32935);
\draw[line width=0.6pt, draw=black!50, dash pattern=on 2.2pt off 1.5pt, -{Latex[length=4.5pt,width=5pt]}] (0.32856,0.16083) -- (0.14246,0.08759);
\draw[line width=0.6pt, draw=black!50, dash pattern=on 2.2pt off 1.5pt, -{Latex[length=4.5pt,width=5pt]}] (-1.28547,0.92014) -- (-1.08555,0.92602);
\draw[line width=0.6pt, draw=black!50, dash pattern=on 2.2pt off 1.5pt, -{Latex[length=4.5pt,width=5pt]}] (-0.14246,1.18969) -- (-0.32856,1.11644);
\draw[line width=0.6pt, draw=black!50, dash pattern=on 2.2pt off 1.5pt, -{Latex[length=4.5pt,width=5pt]}] (0.51320,0.94792) -- (0.59646,0.76607);
\draw[line width=0.6pt, draw=black!50, dash pattern=on 2.2pt off 1.5pt, -{Latex[length=4.5pt,width=5pt]}] (1.26441,-0.02656) -- (1.10661,0.09631);
\draw[line width=0.6pt, draw=black!50, dash pattern=on 2.2pt off 1.5pt, -{Latex[length=4.5pt,width=5pt]}] (0.54010,1.65731) -- (0.41785,1.49902);
\draw[line width=0.6pt, draw=black!50, dash pattern=on 2.2pt off 1.5pt, -{Latex[length=4.5pt,width=5pt]}] (1.45006,-0.70318) -- (1.39198,-0.89456);
\draw[line width=0.6pt, draw=black!50, dash pattern=on 2.2pt off 1.5pt, -{Latex[length=4.5pt,width=5pt]}] (1.80146,0.01307) -- (1.67922,-0.14522);

% Visible edges and arrows
\draw[line width=0.95pt, line cap=round] (P1) -- (P3);
\draw[line width=0.95pt, line cap=round] (P1) -- (P7);
\draw[line width=0.95pt, line cap=round] (P3) -- (P5);
\draw[line width=0.95pt, line cap=round] (P3) -- (P9);
\draw[line width=0.95pt, line cap=round] (P5) -- (P11);
\draw[line width=0.95pt, line cap=round] (P6) -- (P7);
\draw[line width=0.95pt, line cap=round] (P6) -- (P12);
\draw[line width=0.95pt, line cap=round] (P7) -- (P13);
\draw[line width=0.95pt, line cap=round] (P9) -- (P11);
\draw[line width=0.95pt, line cap=round] (P9) -- (P15);
\draw[line width=0.95pt, line cap=round] (P11) -- (P17);
\draw[line width=0.95pt, line cap=round] (P12) -- (P18);
\draw[line width=0.95pt, line cap=round] (P13) -- (P15);
\draw[line width=0.95pt, line cap=round] (P13) -- (P19);
\draw[line width=0.95pt, line cap=round] (P15) -- (P21);
\draw[line width=0.95pt, line cap=round] (P16) -- (P17);
\draw[line width=0.95pt, line cap=round] (P16) -- (P22);
\draw[line width=0.95pt, line cap=round] (P17) -- (P23);
\draw[line width=0.95pt, line cap=round] (P18) -- (P19);
\draw[line width=0.95pt, line cap=round] (P18) -- (P20);
\draw[line width=0.95pt, line cap=round] (P19) -- (P21);
\draw[line width=0.95pt, line cap=round] (P20) -- (P22);
\draw[line width=0.95pt, line cap=round] (P21) -- (P23);
\draw[line width=0.95pt, line cap=round] (P22) -- (P23);
\draw[line width=0.8pt, -{Latex[length=5.2pt,width=6pt]}] (-1.70612,-0.56416) -- (-1.62286,-0.74601);
\draw[line width=0.8pt, -{Latex[length=5.2pt,width=6pt]}] (-0.95490,-1.53864) -- (-1.11271,-1.41577);
\draw[line width=0.8pt, -{Latex[length=5.2pt,width=6pt]}] (-1.71130,0.48109) -- (-1.76938,0.28971);
\draw[line width=0.8pt, -{Latex[length=5.2pt,width=6pt]}] (-1.67922,0.14522) -- (-1.80146,-0.01307);
\draw[line width=0.8pt, -{Latex[length=5.2pt,width=6pt]}] (-1.35990,1.19733) -- (-1.48215,1.03904);
\draw[line width=0.8pt, -{Latex[length=5.2pt,width=6pt]}] (-0.34342,-1.77622) -- (-0.14351,-1.77033);
\draw[line width=0.8pt, -{Latex[length=5.2pt,width=6pt]}] (0.79959,-1.50666) -- (0.61348,-1.57991);
\draw[line width=0.8pt, -{Latex[length=5.2pt,width=6pt]}] (-0.41785,-1.49902) -- (-0.54010,-1.65731);
\draw[line width=0.8pt, -{Latex[length=5.2pt,width=6pt]}] (-1.39198,0.89456) -- (-1.45006,0.70318);
\draw[line width=0.8pt, -{Latex[length=5.2pt,width=6pt]}] (-1.10661,-0.09631) -- (-1.26441,0.02656);
\draw[line width=0.8pt, -{Latex[length=5.2pt,width=6pt]}] (-0.61348,1.57991) -- (-0.79959,1.50666);
\draw[line width=0.8pt, -{Latex[length=5.2pt,width=6pt]}] (1.48215,-1.03904) -- (1.35990,-1.19733);
\draw[line width=0.8pt, -{Latex[length=5.2pt,width=6pt]}] (-0.59646,-0.76607) -- (-0.51320,-0.94792);
\draw[line width=0.8pt, -{Latex[length=5.2pt,width=6pt]}] (0.32856,-1.11644) -- (0.14246,-1.18969);
\draw[line width=0.8pt, -{Latex[length=5.2pt,width=6pt]}] (-0.14246,-0.08759) -- (-0.32856,-0.16083);
\draw[line width=0.8pt, -{Latex[length=5.2pt,width=6pt]}] (0.14351,1.77033) -- (0.34342,1.77622);
\draw[line width=0.8pt, -{Latex[length=5.2pt,width=6pt]}] (1.11271,1.41577) -- (0.95490,1.53864);
\draw[line width=0.8pt, -{Latex[length=5.2pt,width=6pt]}] (0.32237,1.39252) -- (0.16456,1.51539);
\draw[line width=0.8pt, -{Latex[length=5.2pt,width=6pt]}] (1.08555,-0.92602) -- (1.28547,-0.92014);
\draw[line width=0.8pt, -{Latex[length=5.2pt,width=6pt]}] (1.76938,-0.28971) -- (1.71130,-0.48109);
\draw[line width=0.8pt, -{Latex[length=5.2pt,width=6pt]}] (0.51320,-0.32935) -- (0.59646,-0.51120);
\draw[line width=0.8pt, -{Latex[length=5.2pt,width=6pt]}] (1.62286,0.74601) -- (1.70612,0.56416);
\draw[line width=0.8pt, -{Latex[length=5.2pt,width=6pt]}] (0.50802,0.71590) -- (0.44994,0.52452);
\draw[line width=0.8pt, -{Latex[length=5.2pt,width=6pt]}] (0.93385,1.15495) -- (1.13376,1.16083);

% Highlighted face vertices
\fill[black] (P18) circle[radius=1.55pt];
\fill[black] (P19) circle[radius=1.55pt];
\fill[black] (P20) circle[radius=1.55pt];
\fill[black] (P21) circle[radius=1.55pt];
\fill[black] (P22) circle[radius=1.55pt];
\fill[black] (P23) circle[radius=1.55pt];

% Selected inverse-permutation labels
\node[anchor=north west, font=\fontsize{6}{6.0}\selectfont] at (1.68655,-0.97250) {1230};
\node[anchor=north west, font=\fontsize{6}{6.0}\selectfont] at (0.86925,-0.97250) {1320};
\node[anchor=west, font=\fontsize{6}{6.0}\selectfont] at (2.02188,0.14969) {2130};
\node[anchor=west, font=\fontsize{6}{6.0}\selectfont] at (0.44104,0.14969) {3120};
\node[anchor=south west, font=\fontsize{6}{6.0}\selectfont] at (1.52357,1.24693) {2310};
\node[anchor=south west, font=\fontsize{6}{6.0}\selectfont] at (0.67615,1.24693) {3210};
\end{tikzpicture}

%% file: blowups.tex
% Pattern Info

\tikzset{every picture/.style={line width=0.75pt}} %set default line width to 0.75pt        

\begin{tikzpicture}[x=0.75pt,y=0.75pt,yscale=-1,xscale=1]

%Straight Lines [id:da8937470716641651] 
\draw    (60,100) -- (120,60) ;
%Straight Lines [id:da6291956829038424] 
\draw    (120,60) -- (190,60) ;
\draw [shift={(120,60)}, rotate = 0] [color={rgb, 255:red, 0; green, 0; blue, 0 }  ][fill={rgb, 255:red, 0; green, 0; blue, 0 }  ][line width=0.75]      (0, 0) circle [x radius= 3.35, y radius= 3.35]   ;
%Straight Lines [id:da1522877509456526] 
\draw    (120,60) -- (158,60) ;
\draw [shift={(160,60)}, rotate = 180] [fill={rgb, 255:red, 0; green, 0; blue, 0 }  ][line width=0.08]  [draw opacity=0] (12,-3) -- (0,0) -- (12,3) -- cycle    ;
\draw [shift={(120,60)}, rotate = 0] [color={rgb, 255:red, 0; green, 0; blue, 0 }  ][fill={rgb, 255:red, 0; green, 0; blue, 0 }  ][line width=0.75]      (0, 0) circle [x radius= 3.35, y radius= 3.35]   ;
%Straight Lines [id:da254847792633499] 
\draw    (270,100) -- (330,60) ;
%Straight Lines [id:da5966338627055563] 
\draw    (330,60) -- (400,60) ;
\draw [shift={(330,60)}, rotate = 0] [color={rgb, 255:red, 0; green, 0; blue, 0 }  ][fill={rgb, 255:red, 0; green, 0; blue, 0 }  ][line width=0.75]      (0, 0) circle [x radius= 3.35, y radius= 3.35]   ;
%Straight Lines [id:da603843333835518] 
\draw    (330,60) -- (368,60) ;
\draw [shift={(370,60)}, rotate = 180] [fill={rgb, 255:red, 0; green, 0; blue, 0 }  ][line width=0.08]  [draw opacity=0] (12,-3) -- (0,0) -- (12,3) -- cycle    ;
\draw [shift={(330,60)}, rotate = 0] [color={rgb, 255:red, 0; green, 0; blue, 0 }  ][fill={rgb, 255:red, 0; green, 0; blue, 0 }  ][line width=0.75]      (0, 0) circle [x radius= 3.35, y radius= 3.35]   ;
%Straight Lines [id:da43839650655728546] 
\draw    (480,100) -- (540,60) ;
%Straight Lines [id:da6560226540420092] 
\draw    (540,60) -- (610,60) ;
\draw [shift={(540,60)}, rotate = 0] [color={rgb, 255:red, 0; green, 0; blue, 0 }  ][fill={rgb, 255:red, 0; green, 0; blue, 0 }  ][line width=0.75]      (0, 0) circle [x radius= 3.35, y radius= 3.35]   ;
%Straight Lines [id:da5870278063605779] 
\draw    (60,220) -- (90,190) ;
\draw [shift={(90,190)}, rotate = 315] [color={rgb, 255:red, 0; green, 0; blue, 0 }  ][fill={rgb, 255:red, 0; green, 0; blue, 0 }  ][line width=0.75]      (0, 0) circle [x radius= 3.35, y radius= 3.35]   ;
%Straight Lines [id:da1504030798196322] 
\draw    (150,180) -- (190,180) ;
\draw [shift={(150,180)}, rotate = 0] [color={rgb, 255:red, 0; green, 0; blue, 0 }  ][fill={rgb, 255:red, 0; green, 0; blue, 0 }  ][line width=0.75]      (0, 0) circle [x radius= 3.35, y radius= 3.35]   ;
%Curve Lines [id:da8050956334643667] 
\draw    (90,190) .. controls (89.67,189.67) and (91.16,163.2) .. (118.31,160.16) ;
\draw [shift={(120,160)}, rotate = 175.8] [fill={rgb, 255:red, 0; green, 0; blue, 0 }  ][line width=0.08]  [draw opacity=0] (12,-3) -- (0,0) -- (12,3) -- cycle    ;
%Curve Lines [id:da7263609573108573] 
\draw    (120,160) .. controls (141.5,159.75) and (149,181) .. (150,180) ;
%Straight Lines [id:da9148412257639051] 
\draw    (270,220) -- (300,190) ;
\draw [shift={(300,190)}, rotate = 315] [color={rgb, 255:red, 0; green, 0; blue, 0 }  ][fill={rgb, 255:red, 0; green, 0; blue, 0 }  ][line width=0.75]      (0, 0) circle [x radius= 3.35, y radius= 3.35]   ;
%Straight Lines [id:da6045083153297797] 
\draw    (360,180) -- (400,180) ;
\draw [shift={(360,180)}, rotate = 0] [color={rgb, 255:red, 0; green, 0; blue, 0 }  ][fill={rgb, 255:red, 0; green, 0; blue, 0 }  ][line width=0.75]      (0, 0) circle [x radius= 3.35, y radius= 3.35]   ;
%Curve Lines [id:da10860597497450641] 
\draw    (300,190) .. controls (299.67,189.68) and (301.16,163.2) .. (328.31,160.16) ;
\draw [shift={(330,160)}, rotate = 175.8] [fill={rgb, 255:red, 0; green, 0; blue, 0 }  ][line width=0.08]  [draw opacity=0] (12,-3) -- (0,0) -- (12,3) -- cycle    ;
%Curve Lines [id:da7950986248491687] 
\draw    (330,160) .. controls (351.5,159.75) and (359,181) .. (360,180) ;
%Straight Lines [id:da9553585067952773] 
\draw    (480,220) -- (510,190) ;
\draw [shift={(510,190)}, rotate = 315] [color={rgb, 255:red, 0; green, 0; blue, 0 }  ][fill={rgb, 255:red, 0; green, 0; blue, 0 }  ][line width=0.75]      (0, 0) circle [x radius= 3.35, y radius= 3.35]   ;
%Straight Lines [id:da678533750023786] 
\draw    (570,180) -- (610,180) ;
\draw [shift={(570,180)}, rotate = 0] [color={rgb, 255:red, 0; green, 0; blue, 0 }  ][fill={rgb, 255:red, 0; green, 0; blue, 0 }  ][line width=0.75]      (0, 0) circle [x radius= 3.35, y radius= 3.35]   ;
%Curve Lines [id:da7137459375649567] 
\draw    (510,190) .. controls (509.67,189.68) and (511.16,163.2) .. (538.31,160.16) ;
\draw [shift={(540,160)}, rotate = 175.8] [fill={rgb, 255:red, 0; green, 0; blue, 0 }  ][line width=0.08]  [draw opacity=0] (12,-3) -- (0,0) -- (12,3) -- cycle    ;
%Curve Lines [id:da38717832859170187] 
\draw    (540,160) .. controls (561.5,159.75) and (569,181) .. (570,180) ;
%Straight Lines [id:da4100724335655571] 
\draw    (150,180) -- (178,180) ;
\draw [shift={(180,180)}, rotate = 180] [fill={rgb, 255:red, 0; green, 0; blue, 0 }  ][line width=0.08]  [draw opacity=0] (12,-3) -- (0,0) -- (12,3) -- cycle    ;
\draw [shift={(150,180)}, rotate = 0] [color={rgb, 255:red, 0; green, 0; blue, 0 }  ][fill={rgb, 255:red, 0; green, 0; blue, 0 }  ][line width=0.75]      (0, 0) circle [x radius= 3.35, y radius= 3.35]   ;
%Straight Lines [id:da6966054699889076] 
\draw    (270,100) -- (298.34,81.11) ;
\draw [shift={(300,80)}, rotate = 146.31] [fill={rgb, 255:red, 0; green, 0; blue, 0 }  ][line width=0.08]  [draw opacity=0] (12,-3) -- (0,0) -- (12,3) -- cycle    ;
%Straight Lines [id:da7963945433712981] 
\draw    (270,220) -- (288.59,201.41) ;
\draw [shift={(290,200)}, rotate = 135] [fill={rgb, 255:red, 0; green, 0; blue, 0 }  ][line width=0.08]  [draw opacity=0] (12,-3) -- (0,0) -- (12,3) -- cycle    ;
%Straight Lines [id:da5286391760386177] 
\draw    (360,180) -- (388,180) ;
\draw [shift={(390,180)}, rotate = 180] [fill={rgb, 255:red, 0; green, 0; blue, 0 }  ][line width=0.08]  [draw opacity=0] (12,-3) -- (0,0) -- (12,3) -- cycle    ;
\draw [shift={(360,180)}, rotate = 360] [color={rgb, 255:red, 0; green, 0; blue, 0 }  ][fill={rgb, 255:red, 0; green, 0; blue, 0 }  ][line width=0.75]      (0, 0) circle [x radius= 3.35, y radius= 3.35]   ;
%Straight Lines [id:da04670237198943161] 
\draw    (480,100) -- (508.34,81.11) ;
\draw [shift={(510,80)}, rotate = 146.31] [fill={rgb, 255:red, 0; green, 0; blue, 0 }  ][line width=0.08]  [draw opacity=0] (12,-3) -- (0,0) -- (12,3) -- cycle    ;
%Straight Lines [id:da31554038366096426] 
\draw    (480,220) -- (498.59,201.41) ;
\draw [shift={(500,200)}, rotate = 135] [fill={rgb, 255:red, 0; green, 0; blue, 0 }  ][line width=0.08]  [draw opacity=0] (12,-3) -- (0,0) -- (12,3) -- cycle    ;
%Straight Lines [id:da35416821549272925] 
\draw    (120,95) .. controls (121.67,96.67) and (121.67,98.33) .. (120,100) .. controls (118.33,101.67) and (118.33,103.33) .. (120,105) .. controls (121.67,106.67) and (121.67,108.33) .. (120,110) -- (120,114) -- (120,122) ;
\draw [shift={(120,125)}, rotate = 270] [fill={rgb, 255:red, 0; green, 0; blue, 0 }  ][line width=0.08]  [draw opacity=0] (10.72,-5.15) -- (0,0) -- (10.72,5.15) -- (7.12,0) -- cycle    ;
%Straight Lines [id:da7356176822951471] 
\draw    (330,95) .. controls (331.67,96.67) and (331.67,98.33) .. (330,100) .. controls (328.33,101.67) and (328.33,103.33) .. (330,105) .. controls (331.67,106.67) and (331.67,108.33) .. (330,110) -- (330,114) -- (330,122) ;
\draw [shift={(330,125)}, rotate = 270] [fill={rgb, 255:red, 0; green, 0; blue, 0 }  ][line width=0.08]  [draw opacity=0] (10.72,-5.15) -- (0,0) -- (10.72,5.15) -- (7.12,0) -- cycle    ;
%Straight Lines [id:da9349628599863099] 
\draw    (540,95) .. controls (541.67,96.67) and (541.67,98.33) .. (540,100) .. controls (538.33,101.67) and (538.33,103.33) .. (540,105) .. controls (541.67,106.67) and (541.67,108.33) .. (540,110) -- (540,114) -- (540,122) ;
\draw [shift={(540,125)}, rotate = 270] [fill={rgb, 255:red, 0; green, 0; blue, 0 }  ][line width=0.08]  [draw opacity=0] (10.72,-5.15) -- (0,0) -- (10.72,5.15) -- (7.12,0) -- cycle    ;

% Text Node
\draw (111,38.4) node [anchor=north west][inner sep=0.75pt]    {$p$};
% Text Node
\draw (321,38.4) node [anchor=north west][inner sep=0.75pt]    {$p$};
% Text Node
\draw (531,38.4) node [anchor=north west][inner sep=0.75pt]    {$p$};
% Text Node
\draw (87,192.4) node [anchor=north west][inner sep=0.75pt]    {$p^{0}$};
% Text Node
\draw (141,182.4) node [anchor=north west][inner sep=0.75pt]    {$p^{-}$};
% Text Node
\draw (296,192.4) node [anchor=north west][inner sep=0.75pt]    {$p^{+}$};
% Text Node
\draw (350,182.4) node [anchor=north west][inner sep=0.75pt]    {$p^{-}$};
% Text Node
\draw (510,192.4) node [anchor=north west][inner sep=0.75pt]    {$p^{+}$};
% Text Node
\draw (562,187.4) node [anchor=north west][inner sep=0.75pt]    {$p^{0}$};
% Text Node
\draw (74,59.4) node [anchor=north west][inner sep=0.75pt]    {$I_{p}^{0}$};
% Text Node
\draw (154,34.4) node [anchor=north west][inner sep=0.75pt]    {$I_{p}^{-}$};
% Text Node
\draw (53,184.4) node [anchor=north west][inner sep=0.75pt]    {$\varPhi _{p^0}^{0}$};
% Text Node
\draw (163,154.4) node [anchor=north west][inner sep=0.75pt]    {$\varPhi _{p^-}^{-}$};
% Text Node
\draw (110,137.4) node [anchor=north west][inner sep=0.75pt]    {$O$};
% Text Node
\draw (258,184.4) node [anchor=north west][inner sep=0.75pt]    {$\varPhi _{p^+}^{+}$};
% Text Node
\draw (372,152.4) node [anchor=north west][inner sep=0.75pt]    {$\varPhi _{p^-}^{-}$};
% Text Node
\draw (321,138.4) node [anchor=north west][inner sep=0.75pt]    {$O$};
% Text Node
\draw (531,138.4) node [anchor=north west][inner sep=0.75pt]    {$O$};
% Text Node
\draw (471,182.4) node [anchor=north west][inner sep=0.75pt]    {$\varPhi _{p^+}^{+}$};
% Text Node
\draw (583,154.4) node [anchor=north west][inner sep=0.75pt]    {$\varPhi _{p^0}^{0}$};
% Text Node
\draw (277,59.4) node [anchor=north west][inner sep=0.75pt]    {$I_{p}^{+}$};
% Text Node
\draw (366,34.4) node [anchor=north west][inner sep=0.75pt]    {$I_{p}^{-}$};
% Text Node
\draw (487,59.4) node [anchor=north west][inner sep=0.75pt]    {$I_{p}^{+}$};
% Text Node
\draw (566,34.4) node [anchor=north west][inner sep=0.75pt]    {$I_{p}^{0}$};

\end{tikzpicture}

%% file: two_faced.tex
% Pattern Info

\tikzset{every picture/.style={line width=0.75pt}} %set default line width to 0.75pt        

\begin{tikzpicture}[x=0.75pt,y=0.75pt,yscale=-1,xscale=1]

%Shape: Polygon [id:ds16088943885203555] 
\draw   (50,50) -- (100,50) -- (120,100) -- (40,110) -- cycle ;
%Shape: Polygon [id:ds7688157871000191] 
\draw   (180,50) -- (230,50) -- (250,100) -- (170,110) -- cycle ;
%Straight Lines [id:da670094430051517] 
\draw    (40,110) -- (78.02,105.25) ;
\draw [shift={(80,105)}, rotate = 172.87] [fill={rgb, 255:red, 0; green, 0; blue, 0 }  ][line width=0.08]  [draw opacity=0] (12,-3) -- (0,0) -- (12,3) -- cycle    ;
%Straight Lines [id:da6017987239577217] 
\draw    (50,50) -- (78,50) ;
\draw [shift={(80,50)}, rotate = 180] [fill={rgb, 255:red, 0; green, 0; blue, 0 }  ][line width=0.08]  [draw opacity=0] (12,-3) -- (0,0) -- (12,3) -- cycle    ;
\draw [shift={(50,50)}, rotate = 0] [color={rgb, 255:red, 0; green, 0; blue, 0 }  ][fill={rgb, 255:red, 0; green, 0; blue, 0 }  ][line width=0.75]      (0, 0) circle [x radius= 3.35, y radius= 3.35]   ;
%Straight Lines [id:da9647008718562163] 
\draw    (40,110) -- (45.35,76.97) ;
\draw [shift={(45.67,75)}, rotate = 99.2] [fill={rgb, 255:red, 0; green, 0; blue, 0 }  ][line width=0.08]  [draw opacity=0] (12,-3) -- (0,0) -- (12,3) -- cycle    ;
%Straight Lines [id:da18844085579908854] 
\draw    (120,100) -- (108.84,71.31) ;
\draw [shift={(108.11,69.44)}, rotate = 68.74] [fill={rgb, 255:red, 0; green, 0; blue, 0 }  ][line width=0.08]  [draw opacity=0] (12,-3) -- (0,0) -- (12,3) -- cycle    ;
%Straight Lines [id:da3853363078731956] 
\draw    (170,110) -- (208.02,105.25) ;
\draw [shift={(210,105)}, rotate = 172.87] [fill={rgb, 255:red, 0; green, 0; blue, 0 }  ][line width=0.08]  [draw opacity=0] (12,-3) -- (0,0) -- (12,3) -- cycle    ;
%Straight Lines [id:da5734622998491674] 
\draw    (180,50) -- (208,50) ;
\draw [shift={(210,50)}, rotate = 180] [fill={rgb, 255:red, 0; green, 0; blue, 0 }  ][line width=0.08]  [draw opacity=0] (12,-3) -- (0,0) -- (12,3) -- cycle    ;
%Straight Lines [id:da10177831079881827] 
\draw    (180,50) -- (175.33,78.03) ;
\draw [shift={(175,80)}, rotate = 279.46] [fill={rgb, 255:red, 0; green, 0; blue, 0 }  ][line width=0.08]  [draw opacity=0] (12,-3) -- (0,0) -- (12,3) -- cycle    ;
%Shape: Triangle [id:dp04593850023288115] 
\draw   (336.03,100) -- (298,49.97) -- (377.99,48.91) -- cycle ;
%Straight Lines [id:da6191141040048072] 
\draw    (298,49.97) -- (338.65,49.36) ;
\draw [shift={(340.65,49.33)}, rotate = 179.13] [fill={rgb, 255:red, 0; green, 0; blue, 0 }  ][line width=0.08]  [draw opacity=0] (12,-3) -- (0,0) -- (12,3) -- cycle    ;
%Straight Lines [id:da6751168532559669] 
\draw    (298,49.97) -- (322.45,82.05) ;
\draw [shift={(323.67,83.64)}, rotate = 232.68] [fill={rgb, 255:red, 0; green, 0; blue, 0 }  ][line width=0.08]  [draw opacity=0] (12,-3) -- (0,0) -- (12,3) -- cycle    ;
%Straight Lines [id:da7466889013219975] 
\draw    (336.03,100) -- (359.07,71.86) ;
\draw [shift={(360.33,70.31)}, rotate = 129.3] [fill={rgb, 255:red, 0; green, 0; blue, 0 }  ][line width=0.08]  [draw opacity=0] (12,-3) -- (0,0) -- (12,3) -- cycle    ;
%Straight Lines [id:da9254423315450856] 
\draw    (230,50) -- (240.92,77.14) ;
\draw [shift={(241.67,79)}, rotate = 248.09] [fill={rgb, 255:red, 0; green, 0; blue, 0 }  ][line width=0.08]  [draw opacity=0] (12,-3) -- (0,0) -- (12,3) -- cycle    ;
%Shape: Triangle [id:dp9876409617594362] 
\draw   (464.37,100) -- (426.33,49.97) -- (506.32,48.91) -- cycle ;
%Straight Lines [id:da2093537415917439] 
\draw    (426.33,49.97) -- (466.82,49.47) ;
\draw [shift={(468.82,49.44)}, rotate = 179.29] [fill={rgb, 255:red, 0; green, 0; blue, 0 }  ][line width=0.08]  [draw opacity=0] (12,-3) -- (0,0) -- (12,3) -- cycle    ;
%Straight Lines [id:da08796683341439149] 
\draw    (464.37,100) -- (445.2,74.57) ;
\draw [shift={(444,72.97)}, rotate = 53] [fill={rgb, 255:red, 0; green, 0; blue, 0 }  ][line width=0.08]  [draw opacity=0] (12,-3) -- (0,0) -- (12,3) -- cycle    ;
%Straight Lines [id:da5934289486483239] 
\draw    (464.37,100) -- (487.21,72.05) ;
\draw [shift={(488.47,70.5)}, rotate = 129.25] [fill={rgb, 255:red, 0; green, 0; blue, 0 }  ][line width=0.08]  [draw opacity=0] (12,-3) -- (0,0) -- (12,3) -- cycle    ;
%Shape: Triangle [id:dp43326763233244414] 
\draw   (590.87,100) -- (552.83,49.97) -- (632.82,48.91) -- cycle ;
%Straight Lines [id:da8319444030541832] 
\draw    (552.83,49.97) -- (593.32,49.47) ;
\draw [shift={(595.32,49.44)}, rotate = 179.29] [fill={rgb, 255:red, 0; green, 0; blue, 0 }  ][line width=0.08]  [draw opacity=0] (12,-3) -- (0,0) -- (12,3) -- cycle    ;
%Straight Lines [id:da32627137087507085] 
\draw    (552.83,49.97) -- (573.46,77.21) ;
\draw [shift={(574.67,78.81)}, rotate = 232.87] [fill={rgb, 255:red, 0; green, 0; blue, 0 }  ][line width=0.08]  [draw opacity=0] (12,-3) -- (0,0) -- (12,3) -- cycle    ;
%Straight Lines [id:da9536701720764904] 
\draw    (631,50.81) -- (609.59,77.25) ;
\draw [shift={(608.33,78.81)}, rotate = 308.99] [fill={rgb, 255:red, 0; green, 0; blue, 0 }  ][line width=0.08]  [draw opacity=0] (12,-3) -- (0,0) -- (12,3) -- cycle    ;
%Straight Lines [id:da9954593202910825] 
\draw    (50,50) -- (100,50) ;
\draw [shift={(100,50)}, rotate = 0] [color={rgb, 255:red, 0; green, 0; blue, 0 }  ][fill={rgb, 255:red, 0; green, 0; blue, 0 }  ][line width=0.75]      (0, 0) circle [x radius= 3.35, y radius= 3.35]   ;
%Straight Lines [id:da8746687895532279] 
\draw    (180,50) -- (230,50) ;
\draw [shift={(230,50)}, rotate = 0] [color={rgb, 255:red, 0; green, 0; blue, 0 }  ][fill={rgb, 255:red, 0; green, 0; blue, 0 }  ][line width=0.75]      (0, 0) circle [x radius= 3.35, y radius= 3.35]   ;
\draw [shift={(180,50)}, rotate = 0] [color={rgb, 255:red, 0; green, 0; blue, 0 }  ][fill={rgb, 255:red, 0; green, 0; blue, 0 }  ][line width=0.75]      (0, 0) circle [x radius= 3.35, y radius= 3.35]   ;
%Straight Lines [id:da42525415027223057] 
\draw    (298,49.97) -- (377.99,48.91) ;
\draw [shift={(377.99,48.91)}, rotate = 359.24] [color={rgb, 255:red, 0; green, 0; blue, 0 }  ][fill={rgb, 255:red, 0; green, 0; blue, 0 }  ][line width=0.75]      (0, 0) circle [x radius= 3.35, y radius= 3.35]   ;
\draw [shift={(298,49.97)}, rotate = 359.24] [color={rgb, 255:red, 0; green, 0; blue, 0 }  ][fill={rgb, 255:red, 0; green, 0; blue, 0 }  ][line width=0.75]      (0, 0) circle [x radius= 3.35, y radius= 3.35]   ;
%Straight Lines [id:da6698402368941995] 
\draw    (426.33,49.97) -- (506.32,48.91) ;
\draw [shift={(506.32,48.91)}, rotate = 359.24] [color={rgb, 255:red, 0; green, 0; blue, 0 }  ][fill={rgb, 255:red, 0; green, 0; blue, 0 }  ][line width=0.75]      (0, 0) circle [x radius= 3.35, y radius= 3.35]   ;
\draw [shift={(426.33,49.97)}, rotate = 359.24] [color={rgb, 255:red, 0; green, 0; blue, 0 }  ][fill={rgb, 255:red, 0; green, 0; blue, 0 }  ][line width=0.75]      (0, 0) circle [x radius= 3.35, y radius= 3.35]   ;
%Straight Lines [id:da827604112613894] 
\draw    (552.83,49.97) -- (632.82,48.91) ;
\draw [shift={(632.82,48.91)}, rotate = 359.24] [color={rgb, 255:red, 0; green, 0; blue, 0 }  ][fill={rgb, 255:red, 0; green, 0; blue, 0 }  ][line width=0.75]      (0, 0) circle [x radius= 3.35, y radius= 3.35]   ;
\draw [shift={(552.83,49.97)}, rotate = 359.24] [color={rgb, 255:red, 0; green, 0; blue, 0 }  ][fill={rgb, 255:red, 0; green, 0; blue, 0 }  ][line width=0.75]      (0, 0) circle [x radius= 3.35, y radius= 3.35]   ;

% Text Node
\draw (34,35.4) node [anchor=north west][inner sep=0.75pt]    {$p$};
% Text Node
\draw (106,35.4) node [anchor=north west][inner sep=0.75pt]    {$q$};
% Text Node
\draw (164,38.4) node [anchor=north west][inner sep=0.75pt]    {$p$};
% Text Node
\draw (236,37.4) node [anchor=north west][inner sep=0.75pt]    {$q$};
% Text Node
\draw (281,38.37) node [anchor=north west][inner sep=0.75pt]    {$p$};
% Text Node
\draw (383,37.37) node [anchor=north west][inner sep=0.75pt]    {$q$};
% Text Node
\draw (411,38.37) node [anchor=north west][inner sep=0.75pt]    {$p$};
% Text Node
\draw (511,36.37) node [anchor=north west][inner sep=0.75pt]    {$q$};
% Text Node
\draw (538,36.37) node [anchor=north west][inner sep=0.75pt]    {$p$};
% Text Node
\draw (637,34.37) node [anchor=north west][inner sep=0.75pt]    {$q$};
% Text Node
\draw (71,122.4) node [anchor=north west][inner sep=0.75pt]    {$(\mathrm{i})$};
% Text Node
\draw (201,122.4) node [anchor=north west][inner sep=0.75pt]    {$\mathrm{( ii)}$};
% Text Node
\draw (326,122.4) node [anchor=north west][inner sep=0.75pt]    {$\mathrm{( iii)}$};
% Text Node
\draw (451,122.4) node [anchor=north west][inner sep=0.75pt]    {$\mathrm{( iv)}$};
% Text Node
\draw (581,122.4) node [anchor=north west][inner sep=0.75pt]    {$\mathrm{( v)}$};

\end{tikzpicture}

%% file: with_arrows.tex
\tikzset{every picture/.style={line width=0.75pt}} % set default line width

\begin{tikzpicture}[x=0.75pt,y=0.75pt,yscale=-1,xscale=1]

% ----------------------------
% Arrow style for order relation
% ----------------------------
\tikzset{
  ord/.style={
    postaction={decorate},
    decoration={markings, mark=at position 0.55 with {\arrow[scale=1.4]{Latex}}}
  }
}

%arrows = {-Latex[width'=0pt .5, length=10pt]}

\tikzset{
  ordo/.style={
    postaction={decorate},
    decoration={markings, mark=at position 0.45 with {\arrow[scale=1.4]{Latex}}}
  }
}

% ----------------------------
% Name the labeled vertices
% ----------------------------
\coordinate (H) at (223.65,322.85);
\coordinate (G) at (423.06,371.98);
\coordinate (I) at (368.67,257.33);
\coordinate (L) at (223.65,60.79);
\coordinate (J) at (423.06,240.95);
\coordinate (K) at (368.67,126.3);

%Straight Lines [id:da4902721703636367]
\draw    (223.65,322.85) -- (423.06,371.98) ;
%Straight Lines [id:da47540576202534]
\draw    (223.65,322.85) -- (368.67,257.33) ;
%Straight Lines [id:da3268796331043976]
\draw    (368.67,257.33) -- (423.06,371.98) ;
%Straight Lines [id:da2571330274097242]
\draw    (368.67,257.33) -- (368.67,126.3) ;
%Straight Lines [id:da7609840054271558]
\draw [line width=1.5]    (223.65,322.85) -- (223.65,191.82) ;
%Straight Lines [id:da5151140798101816]
\draw    (423.06,371.98) -- (423.06,240.95) ;
%Straight Lines [id:da07211966463142305]
\draw [line width=1.5]    (223.65,191.82) -- (223.65,60.79) ;
%Straight Lines [id:da3405461375844955]
\draw    (423.06,240.95) -- (368.67,126.3) ;
%Straight Lines [id:da6869470038576081]
\draw    (368.67,126.3) -- (223.65,60.79) ;
%Straight Lines [id:da6213148322803826]
\draw [line width=1.5]    (423.06,240.95) -- (223.65,60.79) ;
%Shape: Polygon [id:ds5267183311858671]
\draw  [fill={rgb, 255:red, 128; green, 128; blue, 128 }  ,fill opacity=0.37 ]
  (423.06,240.95) -- (423.06,371.98) -- (368.67,257.33) -- (368.67,126.3) -- cycle ;
%Straight Lines [id:da7647143771697386]
\draw [line width=1.5]    (423.06,371.98) -- (423.06,240.95) ;
%Straight Lines [id:da13776524756122766]
\draw [line width=1.5]    (423.06,240.95) -- (368.67,126.3) ;
%Straight Lines [id:da9741284578081841]
\draw [line width=1.5]    (423.06,371.98) -- (223.65,322.85) ;
%Straight Lines [id:da5607999762365043]
\draw [line width=1.5]    (368.67,126.3) -- (223.65,60.79) ;
%Curve Lines [id:da5177781947923937]
\draw [color={rgb, 255:red, 74; green, 144; blue, 226 }  ,draw opacity=1 ][line width=1.5]
  (423.11,370) .. controls (375.47,344.78) and (368.4,338.38) .. (360.76,310.65) ;
\draw [shift={(360.05,308.02)}, rotate = 75.15] [color={rgb, 255:red, 74; green, 144; blue, 226 }  ,draw opacity=1 ][line width=1.5]
  (14.21,-4.28) .. controls (9.04,-1.82) and (4.3,-0.39) .. (0,0)
  .. controls (4.3,0.39) and (9.04,1.82) .. (14.21,4.28)   ;
%Curve Lines [id:da23339448369922333]
\draw [color={rgb, 255:red, 74; green, 144; blue, 226 }  ,draw opacity=1 ][line width=1.5]
  (368.73,255.35) .. controls (359.05,281.02) and (360.05,300.02) .. (360.05,308.02) ;
%Shape: Circle [id:dp9096176540231814]
\draw  [color={rgb, 255:red, 0; green, 0; blue, 0 }  ,draw opacity=0 ][fill={rgb, 255:red, 74; green, 144; blue, 226 }  ,fill opacity=1 ]
  (361.51,257.33) .. controls (361.51,253.38) and (364.72,250.17) .. (368.67,250.17)
  .. controls (372.63,250.17) and (375.84,253.38) .. (375.84,257.33)
  .. controls (375.84,261.29) and (372.63,264.5) .. (368.67,264.5)
  .. controls (364.72,264.5) and (361.51,261.29) .. (361.51,257.33) -- cycle ;
%Shape: Circle [id:dp6662692734919173]
\draw  [color={rgb, 255:red, 255; green, 255; blue, 255 }  ,draw opacity=0 ][fill={rgb, 255:red, 74; green, 144; blue, 226 }  ,fill opacity=1 ]
  (415.9,371.98) .. controls (415.9,368.03) and (419.1,364.82) .. (423.06,364.82)
  .. controls (427.01,364.82) and (430.22,368.03) .. (430.22,371.98)
  .. controls (430.22,375.94) and (427.01,379.15) .. (423.06,379.15)
  .. controls (419.1,379.15) and (415.9,375.94) .. (415.9,371.98) -- cycle ;

% Text Node
\draw (195.3,315.25) node [anchor=north west][inner sep=0.75pt] {$H$};
% Text Node
\draw (431.56,382.4) node [anchor=north west][inner sep=0.75pt] {$G$};
% Text Node
\draw (339.92,233.35) node [anchor=north west][inner sep=0.75pt] {$I$};
% Text Node
\draw (201.33,36.81) node [anchor=north west][inner sep=0.75pt] {$L$};
% Text Node
\draw (437,216.98) node [anchor=north west][inner sep=0.75pt] {$J$};
% Text Node
\draw (376.58,102.32) node [anchor=north west][inner sep=0.75pt] {$K$};

% ----------------------------
% Arrows encoding: G < H < I < J < K < L
% (direction is "smaller -> larger")
% ----------------------------
\path[ord] (G) -- (H);
\path[ord] (H) -- (I);
\path[ord] (G) -- (I);
\path[ord] (G) -- (J);
\path[ord] (J) -- (K);
\path[ordo] (I) -- (K);
\path[ord] (K) -- (L);
\path[ord] (H) -- (L);
\path[ord] (J) -- (L);

\end{tikzpicture}

%% file: bad_richard.tex
\tdplotsetmaincoords{72}{128}
\begin{tikzpicture}[
  tdplot_main_coords,
  scale=2.15,
  line join=round,
  line cap=round,
  edge/.style={black, line width=0.95pt},
  orient/.style={
    edge,
    postaction={decorate},
    decoration={
      markings,
      mark=at position 0.5 with {\arrow{Latex}}
    }
  },
  bluecurve/.style={
    draw={rgb,255:red,80;green,150;blue,255},
    line width=1.15pt,
    postaction={decorate},
    decoration={
      markings,
      mark=at position 0.5 with {
        \arrow[color={rgb,255:red,80;green,150;blue,255}]{Latex}
      }
    }
  },
  vtx/.style={circle, fill=black, inner sep=1.35pt},
  bluevtx/.style={circle, fill={rgb,255:red,80;green,150;blue,255}, inner sep=1.35pt}
]

% vertices
\coordinate (A) at (2,0,0);   % (2,0,0)
\coordinate (B) at (2,1,0);   % (2,1,0)
\coordinate (C) at (2,0,1);   % (2,0,1)
\coordinate (D) at (0,0,2);   % (0,0,2)
\coordinate (E) at (1,0,2);   % (1,0,2)
\coordinate (F) at (0,1,2);   % (0,1,2)
\coordinate (G) at (0,0,0);   % (0,0,0)
\coordinate (H) at (0,3,0);   % (0,3,0)

% oriented edges
\foreach \u/\v in {G/A,A/B,A/C,B/C,B/H,C/E,G/D,D/E,D/F,E/F,G/H,H/F}
  \draw[orient] (\u)--(\v);

% blue curved arrow in the face x=2 from (2,0,0) to (2,0,1)
\draw[bluecurve]
  (A) .. controls (2,0.18,0.14) and (2,0.14,0.70) .. (C);

% vertices, all same size
\node[bluevtx] at (A) {};
\node[bluevtx] at (C) {};
\foreach \P in {B,D,E,F,G,H}
  \node[vtx] at (\P) {};

\end{tikzpicture}

%% file: Pop_Projective_Space.tex
\tikzset{every picture/.style={line width=0.75pt}} %set default line width to 0.75pt        

\begin{tikzpicture}[
    %--- 3d projection vectors ------------------------
    x={(-0.587802cm,-0.404450cm)},
    y={(0.809005cm,-0.293947cm)},
    z={(0.000078cm,0.866034cm)},
    scale=1,
    %--- arrow tip style -------------------------------
    >=stealth,
    %--- styles ----------------------------------------
    aedge/.style={
      color=black,
      thick,
      decoration={
        markings,
        mark=at position 0.5 with {\arrow{>}}
      },
      postaction={decorate}
    },
    aedge_red/.style={
      color=red,
      thick,
      decoration={
        markings,
        mark=at position 0.5 with {\arrow{>}}
      },
      postaction={decorate}
    },
    edge/.style={
      color=black,
      thick
    },
    back/.style={
      dashed,
      decoration={
        %markings,
        %mark=at position 0.5 with {\arrow{>}}
      },
      postaction={decorate}
    },
    facet/.style={fill=white,fill opacity=0.5},
    vertex/.style={inner sep=1pt,circle,draw=black,fill=black,thick}
  ]
%vertices
  \coordinate (A) at (0,0,1);
  \coordinate (B) at (0,0,2);
  \coordinate (C) at (0,1,0);
  \coordinate (D) at (0,1,2);
  \coordinate (E) at (0,2,0);
  \coordinate (F) at (0,2,1);
  \coordinate (G) at (1,0,0);
  \coordinate (H) at (1,0,2);
  \coordinate (I) at (1,2,0);
  \coordinate (J) at (2,0,0);
  \coordinate (K) at (2,0,1);
  \coordinate (L) at (2,1,0);

  %% back‐facing edges
  \draw[edge,back] (A) -- (B);
  \draw[edge,back] (C) -- (A);
  \draw[edge,back] (G) -- (A);
  \draw[edge,back] (C) -- (E);
  \draw[edge,back] (G) -- (C);
  \draw[edge,back] (G) -- (J);

  %% back vertices
  \node[vertex] at (C) {};
  \node[vertex] at (G) {};
  \node[vertex] at (A) {};

  %% facets
  \fill[facet] (L) -- (I) -- (F) -- (D) -- (H) -- (K) -- cycle;
  \fill[facet] (H) -- (B) -- (D) -- cycle;
  \fill[facet] (I) -- (E) -- (F) -- cycle;
  \fill[facet] (L) -- (J) -- (K) -- cycle;

  %% front‐facing edges
  \draw[edge] (B) -- (D);
  \draw[edge] (B) -- (H);
  \draw[edge] (F) -- (D);
  \draw[edge] (H) -- (D);
  \draw[edge] (E) -- (F);
  \draw[edge] (E) -- (I);
  \draw[edge] (I) -- (F);
  \draw[edge] (K) -- (H);
  \draw[edge] (L) -- (I);
  \draw[edge] (J) -- (K);
  \draw[edge] (L) -- (J);
  \draw[edge] (L) -- (K);

  %% front vertices
  \foreach \P in {B,D,E,F,H,I,J,K,L}
    \node[vertex] at (\P) {};

\end{tikzpicture}

%% file: pop_not_enough_1.tex
% Pattern Info

\tikzset{every picture/.style={line width=0.75pt}} %set default line width to 0.75pt        

\tikzset{every picture/.style={line width=0.75pt}} %set default line width to 0.75pt        

\begin{tikzpicture}%
	[x={(-0.341079cm, 0.092831cm)},
	y={(-0.494526cm, 0.828564cm)},
	z={(0.799443cm, 0.552146cm)},
	scale=1.000000,
	back/.style={dashed},
	edge/.style={color=black, thick},
	facet/.style={fill=white,fill opacity=0.500000},
	vertex/.style={inner sep=1pt,circle,draw=black,fill=black,thick}]
%
%
%% This TikZ-picture was produced with Sagemath version 10.6
%% with the command: ._tikz_3d_in_3d and parameters:
%% view = [-3.50340000000000, 20.9870000000000, 7.10530000000000]
%% angle = 112
%% scale = 1
%% edge_color = blue!95!black
%% facet_color = blue!95!black
%% opacity = 0.8
%% vertex_color = green
%% axis = False
%%
%% Coordinate of the vertices:
%%
\coordinate (3.00000, 0.00000, 0.00000) at (3.00000, 0.00000, 0.00000);
\coordinate (3.00000, 0.00000, 3.00000) at (3.00000, 0.00000, 3.00000);
\coordinate (0.00000, 3.00000, 3.00000) at (0.00000, 3.00000, 3.00000);
\coordinate (0.00000, 0.00000, 15.00000) at (0.00000, 0.00000, 15.00000);
\coordinate (0.00000, 3.00000, 0.00000) at (0.00000, 3.00000, 0.00000);
\coordinate (0.00000, 0.00000, 0.00000) at (0.00000, 0.00000, 0.00000);
%%
%%
%% Drawing edges in the back
%%
\draw[edge,back] (3.00000, 0.00000, 0.00000) -- (3.00000, 0.00000, 3.00000);
\draw[edge,back] (3.00000, 0.00000, 3.00000) -- (0.00000, 3.00000, 3.00000);
\draw[edge,back] (3.00000, 0.00000, 3.00000) -- (0.00000, 0.00000, 15.00000);
%%
%%
%% Drawing vertices in the back
%%
\node[vertex] at (3.00000, 0.00000, 3.00000)     {};
%%
%%
%% Drawing the facets
%%
\fill[facet] (0.00000, 0.00000, 0.00000) -- (0.00000, 0.00000, 15.00000) -- (0.00000, 3.00000, 3.00000) -- (0.00000, 3.00000, 0.00000) -- cycle {};
\fill[facet] (0.00000, 0.00000, 0.00000) -- (3.00000, 0.00000, 0.00000) -- (0.00000, 3.00000, 0.00000) -- cycle {};
%%
%%
%% Drawing edges in the front
%%
\draw[edge] (3.00000, 0.00000, 0.00000) -- (0.00000, 3.00000, 0.00000);
\draw[edge] (3.00000, 0.00000, 0.00000) -- (0.00000, 0.00000, 0.00000);
\draw[edge] (0.00000, 3.00000, 3.00000) -- (0.00000, 0.00000, 15.00000);
\draw[edge] (0.00000, 3.00000, 3.00000) -- (0.00000, 3.00000, 0.00000);
\draw[edge] (0.00000, 0.00000, 15.00000) -- (0.00000, 0.00000, 0.00000);
\draw[edge] (0.00000, 3.00000, 0.00000) -- (0.00000, 0.00000, 0.00000);
%%
%%
%% Drawing the vertices in the front
%%
\node[vertex] at (3.00000, 0.00000, 0.00000)     {};
\node[vertex] at (0.00000, 3.00000, 3.00000)     {};
\node[vertex] at (0.00000, 0.00000, 15.00000)     {};
\node[vertex] at (0.00000, 3.00000, 0.00000)     {};
\node[vertex] at (0.00000, 0.00000, 0.00000)     {};
\end{tikzpicture}

%% file: pop_not_enough_2.tex
% Pattern Info

\tikzset{every picture/.style={line width=0.75pt}} %set default line width to 0.75pt        

\tikzset{every picture/.style={line width=0.75pt}} %set default line width to 0.75pt        

\begin{tikzpicture}%
	[x={(-0.341079cm, 0.092831cm)},
	y={(-0.494526cm, 0.828564cm)},
	z={(0.799443cm, 0.552146cm)},
	scale=1.000000,
	back/.style={ dashed},
	edge/.style={color=black, thick},
	facet/.style={fill=white,fill opacity=0.500000},
	vertex/.style={inner sep=1pt,circle,draw=black,fill=black,thick}]
%
%
%% This TikZ-picture was produced with Sagemath version 10.6
%% with the command: ._tikz_3d_in_3d and parameters:
%% view = [-3.50340000000000, 20.9870000000000, 7.10530000000000]
%% angle = 112
%% scale = 1
%% edge_color = blue!95!black
%% facet_color = blue!95!black
%% opacity = 0.8
%% vertex_color = green
%% axis = False
%%
%% Coordinate of the vertices:
%%
\coordinate (0.00000, 0.00000, 1.00000) at (0.00000, 0.00000, 1.00000);
\coordinate (0.00000, 0.00000, 14.00000) at (0.00000, 0.00000, 14.00000);
\coordinate (0.00000, 1.00000, 0.00000) at (0.00000, 1.00000, 0.00000);
\coordinate (0.00000, 1.00000, 11.00000) at (0.00000, 1.00000, 11.00000);
\coordinate (0.00000, 2.00000, 0.00000) at (0.00000, 2.00000, 0.00000);
\coordinate (0.00000, 2.00000, 7.00000) at (0.00000, 2.00000, 7.00000);
\coordinate (0.00000, 3.00000, 1.00000) at (0.00000, 3.00000, 1.00000);
\coordinate (0.00000, 3.00000, 2.00000) at (0.00000, 3.00000, 2.00000);
\coordinate (1.00000, 0.00000, 0.00000) at (1.00000, 0.00000, 0.00000);
\coordinate (1.00000, 0.00000, 11.00000) at (1.00000, 0.00000, 11.00000);
\coordinate (1.00000, 2.00000, 0.00000) at (1.00000, 2.00000, 0.00000);
\coordinate (1.00000, 2.00000, 3.00000) at (1.00000, 2.00000, 3.00000);
\coordinate (2.00000, 0.00000, 0.00000) at (2.00000, 0.00000, 0.00000);
\coordinate (2.00000, 0.00000, 7.00000) at (2.00000, 0.00000, 7.00000);
\coordinate (2.00000, 1.00000, 0.00000) at (2.00000, 1.00000, 0.00000);
\coordinate (2.00000, 1.00000, 3.00000) at (2.00000, 1.00000, 3.00000);
\coordinate (3.00000, 0.00000, 1.00000) at (3.00000, 0.00000, 1.00000);
\coordinate (3.00000, 0.00000, 2.00000) at (3.00000, 0.00000, 2.00000);
%%
%%
%% Drawing edges in the back
%%
\draw[edge,back] (0.00000, 0.00000, 14.00000) -- (1.00000, 0.00000, 11.00000);
\draw[edge,back] (0.00000, 1.00000, 11.00000) -- (1.00000, 0.00000, 11.00000);
\draw[edge,back] (0.00000, 2.00000, 7.00000) -- (1.00000, 2.00000, 3.00000);
\draw[edge,back] (0.00000, 3.00000, 2.00000) -- (1.00000, 2.00000, 3.00000);
\draw[edge,back] (1.00000, 0.00000, 11.00000) -- (2.00000, 0.00000, 7.00000);
\draw[edge,back] (1.00000, 2.00000, 3.00000) -- (2.00000, 1.00000, 3.00000);
\draw[edge,back] (2.00000, 0.00000, 0.00000) -- (3.00000, 0.00000, 1.00000);
\draw[edge,back] (2.00000, 0.00000, 7.00000) -- (2.00000, 1.00000, 3.00000);
\draw[edge,back] (2.00000, 0.00000, 7.00000) -- (3.00000, 0.00000, 2.00000);
\draw[edge,back] (2.00000, 1.00000, 0.00000) -- (3.00000, 0.00000, 1.00000);
\draw[edge,back] (2.00000, 1.00000, 3.00000) -- (3.00000, 0.00000, 2.00000);
\draw[edge,back] (3.00000, 0.00000, 1.00000) -- (3.00000, 0.00000, 2.00000);
%%
%%
%% Drawing vertices in the back
%%
\node[vertex] at (1.00000, 2.00000, 3.00000)     {};
\node[vertex] at (2.00000, 1.00000, 3.00000)     {};
\node[vertex] at (3.00000, 0.00000, 1.00000)     {};
\node[vertex] at (3.00000, 0.00000, 2.00000)     {};
\node[vertex] at (1.00000, 0.00000, 11.00000)     {};
\node[vertex] at (2.00000, 0.00000, 7.00000)     {};
%%
%%
%% Drawing the facets
%%
\fill[facet] (0.00000, 3.00000, 2.00000) -- (0.00000, 2.00000, 7.00000) -- (0.00000, 1.00000, 11.00000) -- (0.00000, 0.00000, 14.00000) -- (0.00000, 0.00000, 1.00000) -- (0.00000, 1.00000, 0.00000) -- (0.00000, 2.00000, 0.00000) -- (0.00000, 3.00000, 1.00000) -- cycle {};
\fill[facet] (2.00000, 1.00000, 0.00000) -- (1.00000, 2.00000, 0.00000) -- (0.00000, 2.00000, 0.00000) -- (0.00000, 1.00000, 0.00000) -- (1.00000, 0.00000, 0.00000) -- (2.00000, 0.00000, 0.00000) -- cycle {};
\fill[facet] (1.00000, 2.00000, 0.00000) -- (0.00000, 2.00000, 0.00000) -- (0.00000, 3.00000, 1.00000) -- cycle {};
\fill[facet] (1.00000, 0.00000, 0.00000) -- (0.00000, 0.00000, 1.00000) -- (0.00000, 1.00000, 0.00000) -- cycle {};
%%
%%
%% Drawing edges in the front
%%
\draw[edge] (0.00000, 0.00000, 1.00000) -- (0.00000, 0.00000, 14.00000);
\draw[edge] (0.00000, 0.00000, 1.00000) -- (0.00000, 1.00000, 0.00000);
\draw[edge] (0.00000, 0.00000, 1.00000) -- (1.00000, 0.00000, 0.00000);
\draw[edge] (0.00000, 0.00000, 14.00000) -- (0.00000, 1.00000, 11.00000);
\draw[edge] (0.00000, 1.00000, 0.00000) -- (0.00000, 2.00000, 0.00000);
\draw[edge] (0.00000, 1.00000, 0.00000) -- (1.00000, 0.00000, 0.00000);
\draw[edge] (0.00000, 1.00000, 11.00000) -- (0.00000, 2.00000, 7.00000);
\draw[edge] (0.00000, 2.00000, 0.00000) -- (0.00000, 3.00000, 1.00000);
\draw[edge] (0.00000, 2.00000, 0.00000) -- (1.00000, 2.00000, 0.00000);
\draw[edge] (0.00000, 2.00000, 7.00000) -- (0.00000, 3.00000, 2.00000);
\draw[edge] (0.00000, 3.00000, 1.00000) -- (0.00000, 3.00000, 2.00000);
\draw[edge] (0.00000, 3.00000, 1.00000) -- (1.00000, 2.00000, 0.00000);
\draw[edge] (1.00000, 0.00000, 0.00000) -- (2.00000, 0.00000, 0.00000);
\draw[edge] (1.00000, 2.00000, 0.00000) -- (2.00000, 1.00000, 0.00000);
\draw[edge] (2.00000, 0.00000, 0.00000) -- (2.00000, 1.00000, 0.00000);
%%
%%
%% Drawing the vertices in the front
%%
\node[vertex] at (0.00000, 0.00000, 1.00000)     {};
\node[vertex] at (0.00000, 0.00000, 14.00000)     {};
\node[vertex] at (0.00000, 1.00000, 0.00000)     {};
\node[vertex] at (0.00000, 1.00000, 11.00000)     {};
\node[vertex] at (0.00000, 2.00000, 0.00000)     {};
\node[vertex] at (0.00000, 2.00000, 7.00000)     {};
\node[vertex] at (0.00000, 3.00000, 1.00000)     {};
\node[vertex] at (0.00000, 3.00000, 2.00000)     {};
\node[vertex] at (1.00000, 0.00000, 0.00000)     {};
\node[vertex] at (1.00000, 2.00000, 0.00000)     {};
\node[vertex] at (2.00000, 0.00000, 0.00000)     {};
\node[vertex] at (2.00000, 1.00000, 0.00000)     {};
\end{tikzpicture}

%% file: product_part1.tex
% Pattern Info

\tikzset{every picture/.style={line width=0.75pt}} %set default line width to 0.75pt        

\begin{tikzpicture}[x=0.75pt,y=0.75pt,yscale=-1,xscale=1]

%Straight Lines [id:da752609958411949] 
\draw    (474.69,226.61) -- (565.71,257.23) ;
\draw [shift={(565.71,257.23)}, rotate = 18.59] [color={rgb, 255:red, 0; green, 0; blue, 0 }  ][fill={rgb, 255:red, 0; green, 0; blue, 0 }  ][line width=0.75]      (0, 0) circle [x radius= 3.35, y radius= 3.35]   ;
\draw [shift={(474.69,226.61)}, rotate = 18.59] [color={rgb, 255:red, 0; green, 0; blue, 0 }  ][fill={rgb, 255:red, 0; green, 0; blue, 0 }  ][line width=0.75]      (0, 0) circle [x radius= 3.35, y radius= 3.35]   ;
%Straight Lines [id:da22889825958872145] 
\draw    (553.41,200.83) -- (565.71,257.23) ;
\draw [shift={(565.71,257.23)}, rotate = 77.7] [color={rgb, 255:red, 0; green, 0; blue, 0 }  ][fill={rgb, 255:red, 0; green, 0; blue, 0 }  ][line width=0.75]      (0, 0) circle [x radius= 3.35, y radius= 3.35]   ;
\draw [shift={(553.41,200.83)}, rotate = 77.7] [color={rgb, 255:red, 0; green, 0; blue, 0 }  ][fill={rgb, 255:red, 0; green, 0; blue, 0 }  ][line width=0.75]      (0, 0) circle [x radius= 3.35, y radius= 3.35]   ;
%Straight Lines [id:da9997773151108145] 
\draw    (474.69,226.61) -- (553.41,200.83) ;
\draw [shift={(553.41,200.83)}, rotate = 341.86] [color={rgb, 255:red, 0; green, 0; blue, 0 }  ][fill={rgb, 255:red, 0; green, 0; blue, 0 }  ][line width=0.75]      (0, 0) circle [x radius= 3.35, y radius= 3.35]   ;
\draw [shift={(474.69,226.61)}, rotate = 341.86] [color={rgb, 255:red, 0; green, 0; blue, 0 }  ][fill={rgb, 255:red, 0; green, 0; blue, 0 }  ][line width=0.75]      (0, 0) circle [x radius= 3.35, y radius= 3.35]   ;
%Straight Lines [id:da20484680503396357] 
\draw    (444.35,132.33) -- (535.37,162.95) ;
\draw [shift={(535.37,162.95)}, rotate = 18.59] [color={rgb, 255:red, 0; green, 0; blue, 0 }  ][fill={rgb, 255:red, 0; green, 0; blue, 0 }  ][line width=0.75]      (0, 0) circle [x radius= 3.35, y radius= 3.35]   ;
\draw [shift={(444.35,132.33)}, rotate = 18.59] [color={rgb, 255:red, 0; green, 0; blue, 0 }  ][fill={rgb, 255:red, 0; green, 0; blue, 0 }  ][line width=0.75]      (0, 0) circle [x radius= 3.35, y radius= 3.35]   ;
%Straight Lines [id:da8252482370384508] 
\draw    (523.07,106.55) -- (535.37,162.95) ;
\draw [shift={(535.37,162.95)}, rotate = 77.7] [color={rgb, 255:red, 0; green, 0; blue, 0 }  ][fill={rgb, 255:red, 0; green, 0; blue, 0 }  ][line width=0.75]      (0, 0) circle [x radius= 3.35, y radius= 3.35]   ;
\draw [shift={(523.07,106.55)}, rotate = 77.7] [color={rgb, 255:red, 0; green, 0; blue, 0 }  ][fill={rgb, 255:red, 0; green, 0; blue, 0 }  ][line width=0.75]      (0, 0) circle [x radius= 3.35, y radius= 3.35]   ;
%Straight Lines [id:da04419063608611751] 
\draw    (444.35,132.33) -- (523.07,106.55) ;
\draw [shift={(523.07,106.55)}, rotate = 341.86] [color={rgb, 255:red, 0; green, 0; blue, 0 }  ][fill={rgb, 255:red, 0; green, 0; blue, 0 }  ][line width=0.75]      (0, 0) circle [x radius= 3.35, y radius= 3.35]   ;
\draw [shift={(444.35,132.33)}, rotate = 341.86] [color={rgb, 255:red, 0; green, 0; blue, 0 }  ][fill={rgb, 255:red, 0; green, 0; blue, 0 }  ][line width=0.75]      (0, 0) circle [x radius= 3.35, y radius= 3.35]   ;
%Straight Lines [id:da14149457826100864] 
\draw  [dash pattern={on 0.75pt off 0.75pt}]  (445.78,131.87) .. controls (447.87,132.95) and (448.38,134.54) .. (447.31,136.63) .. controls (446.23,138.73) and (446.74,140.32) .. (448.84,141.39) .. controls (450.94,142.46) and (451.45,144.05) .. (450.37,146.15) .. controls (449.3,148.25) and (449.81,149.84) .. (451.91,150.91) .. controls (454.01,151.98) and (454.52,153.57) .. (453.44,155.67) .. controls (452.36,157.77) and (452.87,159.36) .. (454.97,160.43) .. controls (457.07,161.5) and (457.58,163.09) .. (456.5,165.19) .. controls (455.42,167.29) and (455.93,168.88) .. (458.03,169.95) .. controls (460.13,171.02) and (460.64,172.61) .. (459.56,174.71) .. controls (458.49,176.81) and (459,178.4) .. (461.1,179.47) .. controls (463.2,180.54) and (463.71,182.13) .. (462.63,184.23) .. controls (461.55,186.33) and (462.06,187.92) .. (464.16,188.99) .. controls (466.26,190.06) and (466.77,191.65) .. (465.69,193.75) .. controls (464.61,195.85) and (465.12,197.44) .. (467.22,198.51) .. controls (469.32,199.58) and (469.83,201.17) .. (468.75,203.27) .. controls (467.68,205.37) and (468.19,206.96) .. (470.29,208.03) .. controls (472.39,209.1) and (472.9,210.69) .. (471.82,212.79) .. controls (470.74,214.89) and (471.25,216.48) .. (473.35,217.55) .. controls (475.45,218.62) and (475.96,220.21) .. (474.88,222.31) -- (476.12,226.15) -- (476.12,226.15)(442.92,132.79) .. controls (445.02,133.87) and (445.53,135.46) .. (444.45,137.55) .. controls (443.38,139.65) and (443.89,141.24) .. (445.99,142.31) .. controls (448.09,143.38) and (448.6,144.97) .. (447.52,147.07) .. controls (446.44,149.17) and (446.95,150.76) .. (449.05,151.83) .. controls (451.15,152.9) and (451.66,154.49) .. (450.58,156.59) .. controls (449.5,158.69) and (450.01,160.28) .. (452.11,161.35) .. controls (454.21,162.42) and (454.72,164.01) .. (453.64,166.11) .. controls (452.57,168.21) and (453.08,169.8) .. (455.18,170.87) .. controls (457.28,171.94) and (457.79,173.53) .. (456.71,175.63) .. controls (455.63,177.73) and (456.14,179.32) .. (458.24,180.39) .. controls (460.34,181.46) and (460.85,183.05) .. (459.77,185.15) .. controls (458.69,187.25) and (459.2,188.84) .. (461.3,189.91) .. controls (463.4,190.98) and (463.91,192.57) .. (462.84,194.67) .. controls (461.76,196.77) and (462.27,198.36) .. (464.37,199.43) .. controls (466.47,200.5) and (466.98,202.09) .. (465.9,204.19) .. controls (464.82,206.29) and (465.33,207.88) .. (467.43,208.95) .. controls (469.53,210.02) and (470.04,211.61) .. (468.96,213.71) .. controls (467.88,215.81) and (468.39,217.4) .. (470.49,218.47) .. controls (472.59,219.54) and (473.1,221.13) .. (472.03,223.23) -- (473.26,227.07) -- (473.26,227.07) ;
%Straight Lines [id:da7870690738291825] 
\draw    (444.35,132.33) -- (444.35,65.45) ;
\draw [shift={(444.35,65.45)}, rotate = 270] [color={rgb, 255:red, 0; green, 0; blue, 0 }  ][fill={rgb, 255:red, 0; green, 0; blue, 0 }  ][line width=0.75]      (0, 0) circle [x radius= 3.35, y radius= 3.35]   ;
\draw [shift={(444.35,132.33)}, rotate = 270] [color={rgb, 255:red, 0; green, 0; blue, 0 }  ][fill={rgb, 255:red, 0; green, 0; blue, 0 }  ][line width=0.75]      (0, 0) circle [x radius= 3.35, y radius= 3.35]   ;
%Straight Lines [id:da9429706462818134] 
\draw  [dash pattern={on 4.5pt off 4.5pt}]  (535.37,162.95) -- (535.37,98.42) ;
\draw [shift={(535.37,96.07)}, rotate = 270] [color={rgb, 255:red, 0; green, 0; blue, 0 }  ][line width=0.75]      (0, 0) circle [x radius= 3.35, y radius= 3.35]   ;
\draw [shift={(535.37,162.95)}, rotate = 270] [color={rgb, 255:red, 0; green, 0; blue, 0 }  ][fill={rgb, 255:red, 0; green, 0; blue, 0 }  ][line width=0.75]      (0, 0) circle [x radius= 3.35, y radius= 3.35]   ;
%Straight Lines [id:da18217665965563523] 
\draw  [dash pattern={on 4.5pt off 4.5pt}]  (523.07,106.55) -- (523.07,42.02) ;
\draw [shift={(523.07,39.67)}, rotate = 270] [color={rgb, 255:red, 0; green, 0; blue, 0 }  ][line width=0.75]      (0, 0) circle [x radius= 3.35, y radius= 3.35]   ;
\draw [shift={(523.07,106.55)}, rotate = 270] [color={rgb, 255:red, 0; green, 0; blue, 0 }  ][fill={rgb, 255:red, 0; green, 0; blue, 0 }  ][line width=0.75]      (0, 0) circle [x radius= 3.35, y radius= 3.35]   ;
%Straight Lines [id:da27541225781975953] 
\draw  [dash pattern={on 4.5pt off 4.5pt}]  (444.35,65.45) -- (520.84,40.4) ;
\draw [shift={(523.07,39.67)}, rotate = 341.86] [color={rgb, 255:red, 0; green, 0; blue, 0 }  ][line width=0.75]      (0, 0) circle [x radius= 3.35, y radius= 3.35]   ;
\draw [shift={(444.35,65.45)}, rotate = 341.86] [color={rgb, 255:red, 0; green, 0; blue, 0 }  ][fill={rgb, 255:red, 0; green, 0; blue, 0 }  ][line width=0.75]      (0, 0) circle [x radius= 3.35, y radius= 3.35]   ;
%Straight Lines [id:da6907557427193569] 
\draw  [dash pattern={on 4.5pt off 4.5pt}]  (444.35,65.45) -- (533.15,95.32) ;
\draw [shift={(535.37,96.07)}, rotate = 18.59] [color={rgb, 255:red, 0; green, 0; blue, 0 }  ][line width=0.75]      (0, 0) circle [x radius= 3.35, y radius= 3.35]   ;
\draw [shift={(444.35,65.45)}, rotate = 18.59] [color={rgb, 255:red, 0; green, 0; blue, 0 }  ][fill={rgb, 255:red, 0; green, 0; blue, 0 }  ][line width=0.75]      (0, 0) circle [x radius= 3.35, y radius= 3.35]   ;
%Straight Lines [id:da4816346704193779] 
\draw    (523.57,41.96) -- (534.87,93.78) ;
\draw [shift={(535.37,96.07)}, rotate = 77.7] [color={rgb, 255:red, 0; green, 0; blue, 0 }  ][line width=0.75]      (0, 0) circle [x radius= 3.35, y radius= 3.35]   ;
\draw [shift={(523.07,39.67)}, rotate = 77.7] [color={rgb, 255:red, 0; green, 0; blue, 0 }  ][line width=0.75]      (0, 0) circle [x radius= 3.35, y radius= 3.35]   ;
%Straight Lines [id:da7590712895090702] 
\draw    (160.19,216) -- (251.22,246.62) ;
\draw [shift={(251.22,246.62)}, rotate = 18.59] [color={rgb, 255:red, 0; green, 0; blue, 0 }  ][fill={rgb, 255:red, 0; green, 0; blue, 0 }  ][line width=0.75]      (0, 0) circle [x radius= 3.35, y radius= 3.35]   ;
\draw [shift={(160.19,216)}, rotate = 18.59] [color={rgb, 255:red, 0; green, 0; blue, 0 }  ][fill={rgb, 255:red, 0; green, 0; blue, 0 }  ][line width=0.75]      (0, 0) circle [x radius= 3.35, y radius= 3.35]   ;
%Straight Lines [id:da22721673730965697] 
\draw    (238.92,190.21) -- (251.22,246.62) ;
\draw [shift={(251.22,246.62)}, rotate = 77.7] [color={rgb, 255:red, 0; green, 0; blue, 0 }  ][fill={rgb, 255:red, 0; green, 0; blue, 0 }  ][line width=0.75]      (0, 0) circle [x radius= 3.35, y radius= 3.35]   ;
\draw [shift={(238.92,190.21)}, rotate = 77.7] [color={rgb, 255:red, 0; green, 0; blue, 0 }  ][fill={rgb, 255:red, 0; green, 0; blue, 0 }  ][line width=0.75]      (0, 0) circle [x radius= 3.35, y radius= 3.35]   ;
%Straight Lines [id:da5636126288155222] 
\draw    (160.19,216) -- (238.92,190.21) ;
\draw [shift={(238.92,190.21)}, rotate = 341.86] [color={rgb, 255:red, 0; green, 0; blue, 0 }  ][fill={rgb, 255:red, 0; green, 0; blue, 0 }  ][line width=0.75]      (0, 0) circle [x radius= 3.35, y radius= 3.35]   ;
\draw [shift={(160.19,216)}, rotate = 341.86] [color={rgb, 255:red, 0; green, 0; blue, 0 }  ][fill={rgb, 255:red, 0; green, 0; blue, 0 }  ][line width=0.75]      (0, 0) circle [x radius= 3.35, y radius= 3.35]   ;
%Curve Lines [id:da43198199037044505] 
\draw  [dash pattern={on 0.75pt off 0.75pt}]  (474.38,228.08) .. controls (471.88,229.23) and (470.05,228.8) .. (468.9,226.8) .. controls (467.85,224.79) and (466.38,224.38) .. (464.48,225.56) .. controls (462.07,226.54) and (460.46,226) .. (459.63,223.95) .. controls (458.5,221.74) and (456.79,221.05) .. (454.5,221.9) .. controls (452.57,222.85) and (451.19,222.19) .. (450.36,219.91) .. controls (449.71,217.66) and (448.09,216.72) .. (445.52,217.11) .. controls (443.24,217.58) and (441.96,216.67) .. (441.67,214.38) .. controls (441.31,211.92) and (439.9,210.68) .. (437.45,210.66) .. controls (435.21,210.72) and (434.13,209.51) .. (434.22,207.03) .. controls (434.55,204.66) and (433.56,203.21) .. (431.25,202.68) .. controls (428.94,201.9) and (428.17,200.35) .. (428.94,198.03) .. controls (429.91,195.88) and (429.36,194.23) .. (427.28,193.09) .. controls (425.25,191.7) and (424.91,189.96) .. (426.26,187.89) .. controls (427.74,185.97) and (427.61,184.33) .. (425.88,182.98) .. controls (424.21,181.07) and (424.25,179.37) .. (425.99,177.89) .. controls (427.82,176.14) and (428.02,174.37) .. (426.58,172.6) .. controls (425.23,170.68) and (425.54,169.06) .. (427.51,167.75) .. controls (429.52,166.52) and (429.96,164.86) .. (428.81,162.77) .. controls (427.74,160.56) and (428.29,158.85) .. (430.46,157.64) .. controls (432.5,156.93) and (433.08,155.4) .. (432.2,153.05) .. controls (431.41,150.61) and (432.08,149.05) .. (434.21,148.36) .. controls (436.34,147.71) and (437.1,146.12) .. (436.48,143.58) .. controls (435.7,141.43) and (436.42,140.03) .. (438.63,139.4) .. controls (440.86,138.79) and (441.64,137.38) .. (440.97,135.15) -- (443.06,131.56)(475,225.15) .. controls (472.55,226.31) and (470.77,225.89) .. (469.65,223.9) .. controls (468.64,221.9) and (467.21,221.5) .. (465.35,222.69) .. controls (463,223.69) and (461.44,223.17) .. (460.65,221.13) .. controls (459.58,218.94) and (457.94,218.28) .. (455.71,219.15) .. controls (453.84,220.13) and (452.52,219.5) .. (451.75,217.25) .. controls (451.17,215.04) and (449.64,214.15) .. (447.15,214.59) .. controls (444.94,215.11) and (443.73,214.25) .. (443.52,212.02) .. controls (443.25,209.63) and (441.94,208.47) .. (439.58,208.54) .. controls (437.41,208.67) and (436.41,207.55) .. (436.58,205.17) .. controls (436.99,202.92) and (436.08,201.58) .. (433.84,201.17) .. controls (431.59,200.5) and (430.88,199.08) .. (431.71,196.89) .. controls (432.73,194.86) and (432.22,193.34) .. (430.18,192.33) .. controls (428.18,191.06) and (427.86,189.45) .. (429.23,187.49) .. controls (430.72,185.69) and (430.6,184.16) .. (428.87,182.9) .. controls (427.21,181.09) and (427.25,179.49) .. (428.98,178.09) .. controls (430.81,176.44) and (431,174.77) .. (429.55,173.06) .. controls (428.18,171.21) and (428.47,169.66) .. (430.44,168.41) .. controls (432.44,167.24) and (432.86,165.64) .. (431.69,163.61) .. controls (430.61,161.45) and (431.14,159.79) .. (433.29,158.64) .. controls (435.32,157.96) and (435.88,156.47) .. (434.98,154.17) .. controls (434.18,151.78) and (434.83,150.25) .. (436.94,149.6) .. controls (439.07,148.98) and (439.81,147.42) .. (439.17,144.91) .. controls (438.37,142.78) and (439.07,141.42) .. (441.28,140.81) .. controls (443.49,140.23) and (444.25,138.84) .. (443.58,136.64) -- (445.64,133.1) ;
%Straight Lines [id:da6917355737702283] 
\draw  [dash pattern={on 0.75pt off 0.75pt}]  (131.28,120.54) .. controls (133.38,121.62) and (133.89,123.21) .. (132.81,125.3) .. controls (131.73,127.4) and (132.24,128.99) .. (134.34,130.06) .. controls (136.44,131.13) and (136.95,132.72) .. (135.88,134.82) .. controls (134.8,136.92) and (135.31,138.51) .. (137.41,139.58) .. controls (139.51,140.65) and (140.02,142.24) .. (138.94,144.34) .. controls (137.86,146.44) and (138.37,148.03) .. (140.47,149.1) .. controls (142.57,150.17) and (143.08,151.76) .. (142,153.86) .. controls (140.92,155.96) and (141.43,157.55) .. (143.53,158.62) .. controls (145.63,159.69) and (146.14,161.28) .. (145.07,163.38) .. controls (143.99,165.48) and (144.5,167.07) .. (146.6,168.14) .. controls (148.7,169.21) and (149.21,170.8) .. (148.13,172.9) .. controls (147.05,175) and (147.56,176.59) .. (149.66,177.66) .. controls (151.76,178.73) and (152.27,180.32) .. (151.19,182.42) .. controls (150.11,184.52) and (150.62,186.11) .. (152.72,187.18) .. controls (154.82,188.25) and (155.33,189.84) .. (154.26,191.94) .. controls (153.19,194.03) and (153.7,195.62) .. (155.79,196.69) .. controls (157.89,197.76) and (158.4,199.35) .. (157.32,201.45) .. controls (156.24,203.55) and (156.75,205.14) .. (158.85,206.21) .. controls (160.95,207.28) and (161.46,208.87) .. (160.38,210.97) -- (161.62,214.82) -- (161.62,214.82)(128.42,121.46) .. controls (130.52,122.53) and (131.03,124.12) .. (129.96,126.22) .. controls (128.88,128.32) and (129.39,129.91) .. (131.49,130.98) .. controls (133.59,132.05) and (134.1,133.64) .. (133.02,135.74) .. controls (131.94,137.84) and (132.45,139.43) .. (134.55,140.5) .. controls (136.65,141.57) and (137.16,143.16) .. (136.08,145.26) .. controls (135,147.36) and (135.51,148.95) .. (137.61,150.02) .. controls (139.71,151.09) and (140.22,152.68) .. (139.15,154.78) .. controls (138.07,156.88) and (138.58,158.47) .. (140.68,159.54) .. controls (142.78,160.61) and (143.29,162.2) .. (142.21,164.3) .. controls (141.13,166.4) and (141.64,167.99) .. (143.74,169.06) .. controls (145.84,170.13) and (146.35,171.72) .. (145.27,173.82) .. controls (144.2,175.92) and (144.71,177.51) .. (146.81,178.58) .. controls (148.91,179.65) and (149.42,181.24) .. (148.34,183.34) .. controls (147.27,185.43) and (147.78,187.02) .. (149.87,188.09) .. controls (151.97,189.16) and (152.48,190.75) .. (151.4,192.85) .. controls (150.32,194.95) and (150.83,196.54) .. (152.93,197.61) .. controls (155.03,198.68) and (155.54,200.27) .. (154.46,202.37) .. controls (153.39,204.47) and (153.9,206.06) .. (156,207.13) .. controls (158.1,208.2) and (158.61,209.79) .. (157.53,211.89) -- (158.77,215.74) -- (158.77,215.74) ;
%Curve Lines [id:da8418188657902628] 
\draw  [dash pattern={on 0.75pt off 0.75pt}]  (159.88,216.75) .. controls (157.39,217.9) and (155.56,217.47) .. (154.41,215.47) .. controls (153.36,213.46) and (151.89,213.05) .. (149.98,214.23) .. controls (147.57,215.21) and (145.96,214.67) .. (145.13,212.62) .. controls (144,210.41) and (142.29,209.72) .. (140,210.57) .. controls (138.07,211.52) and (136.69,210.86) .. (135.86,208.58) .. controls (135.21,206.33) and (133.6,205.39) .. (131.02,205.78) .. controls (128.74,206.25) and (127.46,205.34) .. (127.17,203.05) .. controls (126.82,200.59) and (125.41,199.35) .. (122.96,199.33) .. controls (120.72,199.39) and (119.64,198.18) .. (119.72,195.69) .. controls (120.05,193.33) and (119.07,191.88) .. (116.76,191.35) .. controls (114.44,190.56) and (113.67,189.01) .. (114.44,186.7) .. controls (115.41,184.55) and (114.86,182.9) .. (112.78,181.76) .. controls (110.75,180.37) and (110.41,178.63) .. (111.76,176.55) .. controls (113.24,174.63) and (113.11,173) .. (111.38,171.65) .. controls (109.71,169.73) and (109.75,168.03) .. (111.49,166.55) .. controls (113.33,164.8) and (113.53,163.04) .. (112.09,161.27) .. controls (110.73,159.35) and (111.04,157.73) .. (113.01,156.42) .. controls (115.03,155.19) and (115.46,153.52) .. (114.31,151.43) .. controls (113.24,149.22) and (113.8,147.52) .. (115.97,146.31) .. controls (118,145.59) and (118.58,144.06) .. (117.7,141.72) .. controls (116.91,139.28) and (117.58,137.72) .. (119.71,137.03) .. controls (121.84,136.38) and (122.6,134.78) .. (121.98,132.24) .. controls (121.2,130.09) and (121.92,128.7) .. (124.13,128.06) .. controls (126.36,127.46) and (127.14,126.05) .. (126.47,123.82) -- (128.56,120.23)(160.5,213.81) .. controls (158.05,214.97) and (156.27,214.55) .. (155.16,212.56) .. controls (154.15,210.57) and (152.72,210.17) .. (150.85,211.36) .. controls (148.5,212.36) and (146.94,211.84) .. (146.15,209.8) .. controls (145.09,207.61) and (143.44,206.95) .. (141.21,207.82) .. controls (139.34,208.79) and (138.02,208.16) .. (137.25,205.92) .. controls (136.67,203.71) and (135.14,202.82) .. (132.65,203.26) .. controls (130.44,203.77) and (129.23,202.91) .. (129.02,200.68) .. controls (128.75,198.3) and (127.44,197.14) .. (125.08,197.21) .. controls (122.91,197.34) and (121.91,196.22) .. (122.08,193.84) .. controls (122.49,191.59) and (121.58,190.25) .. (119.34,189.83) .. controls (117.09,189.17) and (116.39,187.74) .. (117.22,185.55) .. controls (118.24,183.52) and (117.73,182) .. (115.68,180.99) .. controls (113.68,179.72) and (113.37,178.11) .. (114.74,176.16) .. controls (116.23,174.36) and (116.11,172.83) .. (114.38,171.57) .. controls (112.71,169.76) and (112.75,168.15) .. (114.48,166.76) .. controls (116.31,165.11) and (116.5,163.43) .. (115.05,161.73) .. controls (113.68,159.88) and (113.98,158.33) .. (115.94,157.08) .. controls (117.94,155.91) and (118.36,154.31) .. (117.19,152.27) .. controls (116.11,150.12) and (116.64,148.46) .. (118.79,147.31) .. controls (120.82,146.63) and (121.38,145.14) .. (120.49,142.84) .. controls (119.68,140.45) and (120.34,138.92) .. (122.45,138.26) .. controls (124.57,137.64) and (125.31,136.08) .. (124.67,133.58) .. controls (123.88,131.45) and (124.58,130.09) .. (126.78,129.48) .. controls (128.99,128.9) and (129.76,127.51) .. (129.08,125.3) -- (131.14,121.77) ;
%Straight Lines [id:da012617133578799677] 
\draw    (129.85,121) -- (129.85,54.12) ;
\draw [shift={(129.85,54.12)}, rotate = 270] [color={rgb, 255:red, 0; green, 0; blue, 0 }  ][fill={rgb, 255:red, 0; green, 0; blue, 0 }  ][line width=0.75]      (0, 0) circle [x radius= 3.35, y radius= 3.35]   ;
\draw [shift={(129.85,121)}, rotate = 270] [color={rgb, 255:red, 0; green, 0; blue, 0 }  ][fill={rgb, 255:red, 0; green, 0; blue, 0 }  ][line width=0.75]      (0, 0) circle [x radius= 3.35, y radius= 3.35]   ;
%Straight Lines [id:da8526023518713186] 
\draw    (129.85,121) -- (220.88,151.62) ;
\draw [shift={(220.88,151.62)}, rotate = 18.59] [color={rgb, 255:red, 0; green, 0; blue, 0 }  ][fill={rgb, 255:red, 0; green, 0; blue, 0 }  ][line width=0.75]      (0, 0) circle [x radius= 3.35, y radius= 3.35]   ;
\draw [shift={(129.85,121)}, rotate = 18.59] [color={rgb, 255:red, 0; green, 0; blue, 0 }  ][fill={rgb, 255:red, 0; green, 0; blue, 0 }  ][line width=0.75]      (0, 0) circle [x radius= 3.35, y radius= 3.35]   ;
%Straight Lines [id:da08254226258710207] 
\draw    (129.85,54.12) -- (220.88,84.74) ;
\draw [shift={(220.88,84.74)}, rotate = 18.59] [color={rgb, 255:red, 0; green, 0; blue, 0 }  ][fill={rgb, 255:red, 0; green, 0; blue, 0 }  ][line width=0.75]      (0, 0) circle [x radius= 3.35, y radius= 3.35]   ;
\draw [shift={(129.85,54.12)}, rotate = 18.59] [color={rgb, 255:red, 0; green, 0; blue, 0 }  ][fill={rgb, 255:red, 0; green, 0; blue, 0 }  ][line width=0.75]      (0, 0) circle [x radius= 3.35, y radius= 3.35]   ;
%Straight Lines [id:da6018622832099214] 
\draw    (220.88,151.62) -- (220.88,84.74) ;
\draw [shift={(220.88,84.74)}, rotate = 270] [color={rgb, 255:red, 0; green, 0; blue, 0 }  ][fill={rgb, 255:red, 0; green, 0; blue, 0 }  ][line width=0.75]      (0, 0) circle [x radius= 3.35, y radius= 3.35]   ;
\draw [shift={(220.88,151.62)}, rotate = 270] [color={rgb, 255:red, 0; green, 0; blue, 0 }  ][fill={rgb, 255:red, 0; green, 0; blue, 0 }  ][line width=0.75]      (0, 0) circle [x radius= 3.35, y radius= 3.35]   ;

% Text Node
\draw (416.01,47.88) node [anchor=north west][inner sep=0.75pt]    {$w_{2}$};
% Text Node
\draw (415.19,122.02) node [anchor=north west][inner sep=0.75pt]    {$w_{1}$};
% Text Node
\draw (475.76,96.23) node [anchor=north west][inner sep=0.75pt]    {$e_{1} '$};
% Text Node
\draw (483.14,152.64) node [anchor=north west][inner sep=0.75pt]    {$e_{0} '$};
% Text Node
\draw (539.62,121.81) node [anchor=north west][inner sep=0.75pt]    {$e'$};
% Text Node
\draw (501.43,192.93) node [anchor=north west][inner sep=0.75pt]    {$e_{1}$};
% Text Node
\draw (517.83,242.89) node [anchor=north west][inner sep=0.75pt]    {$e_{0}$};
% Text Node
\draw (566.57,215.28) node [anchor=north west][inner sep=0.75pt]    {$e$};
% Text Node
\draw (539.38,54.12) node [anchor=north west][inner sep=0.75pt]    {$e''$};
% Text Node
\draw (203.2,232.2) node [anchor=north west][inner sep=0.75pt]    {$e$};
% Text Node
\draw (527.63,217.6) node [anchor=north west][inner sep=0.75pt]    {$J$};
% Text Node
\draw (438.3,177.98) node [anchor=north west][inner sep=0.75pt]    {$C$};
% Text Node
\draw (213.14,206.26) node [anchor=north west][inner sep=0.75pt]    {$J$};
% Text Node
\draw (123.8,166.65) node [anchor=north west][inner sep=0.75pt]    {$C$};
% Text Node
\draw (101.41,109.24) node [anchor=north west][inner sep=0.75pt]    {$w_{1}$};
% Text Node
\draw (102.23,41.59) node [anchor=north west][inner sep=0.75pt]    {$w_{2}$};
% Text Node
\draw (170.49,139.29) node [anchor=north west][inner sep=0.75pt]    {$e'$};
% Text Node
\draw (175.21,49.27) node [anchor=north west][inner sep=0.75pt]    {$e''$};

\end{tikzpicture}

%% file: product_part1_inductive.tex
% Pattern Info

\tikzset{every picture/.style={line width=0.75pt}} %set default line width to 0.75pt        

\begin{tikzpicture}[x=0.75pt,y=0.75pt,yscale=-1,xscale=1]

%Straight Lines [id:da6934980007315071] 
\draw    (186.26,218.63) -- (280.83,250.77) ;
\draw [shift={(280.83,250.77)}, rotate = 18.77] [color={rgb, 255:red, 0; green, 0; blue, 0 }  ][fill={rgb, 255:red, 0; green, 0; blue, 0 }  ][line width=0.75]      (0, 0) circle [x radius= 3.35, y radius= 3.35]   ;
\draw [shift={(186.26,218.63)}, rotate = 18.77] [color={rgb, 255:red, 0; green, 0; blue, 0 }  ][fill={rgb, 255:red, 0; green, 0; blue, 0 }  ][line width=0.75]      (0, 0) circle [x radius= 3.35, y radius= 3.35]   ;
%Straight Lines [id:da3528918760718237] 
\draw    (268.05,191.58) -- (280.83,250.77) ;
\draw [shift={(280.83,250.77)}, rotate = 77.82] [color={rgb, 255:red, 0; green, 0; blue, 0 }  ][fill={rgb, 255:red, 0; green, 0; blue, 0 }  ][line width=0.75]      (0, 0) circle [x radius= 3.35, y radius= 3.35]   ;
\draw [shift={(268.05,191.58)}, rotate = 77.82] [color={rgb, 255:red, 0; green, 0; blue, 0 }  ][fill={rgb, 255:red, 0; green, 0; blue, 0 }  ][line width=0.75]      (0, 0) circle [x radius= 3.35, y radius= 3.35]   ;
%Straight Lines [id:da15393934245900165] 
\draw    (186.26,218.63) -- (268.05,191.58) ;
\draw [shift={(268.05,191.58)}, rotate = 341.69] [color={rgb, 255:red, 0; green, 0; blue, 0 }  ][fill={rgb, 255:red, 0; green, 0; blue, 0 }  ][line width=0.75]      (0, 0) circle [x radius= 3.35, y radius= 3.35]   ;
\draw [shift={(186.26,218.63)}, rotate = 341.69] [color={rgb, 255:red, 0; green, 0; blue, 0 }  ][fill={rgb, 255:red, 0; green, 0; blue, 0 }  ][line width=0.75]      (0, 0) circle [x radius= 3.35, y radius= 3.35]   ;
%Straight Lines [id:da5836652407099755] 
\draw  [dash pattern={on 0.75pt off 0.75pt}]  (156.17,118.49) .. controls (158.26,119.57) and (158.77,121.16) .. (157.69,123.25) .. controls (156.6,125.34) and (157.11,126.93) .. (159.2,128.02) .. controls (161.29,129.1) and (161.8,130.69) .. (160.72,132.78) .. controls (159.64,134.87) and (160.15,136.46) .. (162.24,137.54) .. controls (164.33,138.63) and (164.84,140.22) .. (163.76,142.31) .. controls (162.68,144.4) and (163.19,145.99) .. (165.28,147.07) .. controls (167.37,148.16) and (167.88,149.75) .. (166.79,151.84) .. controls (165.71,153.93) and (166.22,155.52) .. (168.31,156.6) .. controls (170.4,157.68) and (170.91,159.27) .. (169.83,161.36) .. controls (168.75,163.45) and (169.26,165.04) .. (171.35,166.13) .. controls (173.44,167.21) and (173.95,168.8) .. (172.86,170.89) .. controls (171.78,172.98) and (172.29,174.57) .. (174.38,175.66) .. controls (176.47,176.74) and (176.98,178.33) .. (175.9,180.42) .. controls (174.82,182.51) and (175.33,184.1) .. (177.42,185.18) .. controls (179.51,186.27) and (180.02,187.86) .. (178.94,189.95) .. controls (177.85,192.04) and (178.36,193.63) .. (180.45,194.71) .. controls (182.54,195.8) and (183.05,197.39) .. (181.97,199.48) .. controls (180.89,201.57) and (181.4,203.16) .. (183.49,204.24) .. controls (185.58,205.32) and (186.09,206.91) .. (185.01,209) .. controls (183.93,211.09) and (184.44,212.68) .. (186.53,213.77) -- (187.69,217.42) -- (187.69,217.42)(153.31,119.4) .. controls (155.4,120.48) and (155.91,122.07) .. (154.83,124.16) .. controls (153.75,126.25) and (154.26,127.84) .. (156.35,128.93) .. controls (158.44,130.01) and (158.95,131.6) .. (157.86,133.69) .. controls (156.78,135.78) and (157.29,137.37) .. (159.38,138.45) .. controls (161.47,139.54) and (161.98,141.13) .. (160.9,143.22) .. controls (159.82,145.31) and (160.33,146.9) .. (162.42,147.98) .. controls (164.51,149.07) and (165.02,150.66) .. (163.93,152.75) .. controls (162.85,154.84) and (163.36,156.43) .. (165.45,157.51) .. controls (167.54,158.59) and (168.05,160.18) .. (166.97,162.27) .. controls (165.89,164.36) and (166.4,165.95) .. (168.49,167.04) .. controls (170.58,168.12) and (171.09,169.71) .. (170.01,171.8) .. controls (168.92,173.89) and (169.43,175.48) .. (171.52,176.57) .. controls (173.61,177.65) and (174.12,179.24) .. (173.04,181.33) .. controls (171.96,183.42) and (172.47,185.01) .. (174.56,186.09) .. controls (176.65,187.18) and (177.16,188.77) .. (176.08,190.86) .. controls (175,192.95) and (175.51,194.54) .. (177.6,195.62) .. controls (179.69,196.71) and (180.2,198.3) .. (179.11,200.39) .. controls (178.03,202.48) and (178.54,204.07) .. (180.63,205.15) .. controls (182.72,206.23) and (183.23,207.82) .. (182.15,209.91) .. controls (181.07,212) and (181.58,213.59) .. (183.67,214.68) -- (184.83,218.33) -- (184.83,218.33) ;
%Curve Lines [id:da282485005597114] 
\draw  [dash pattern={on 0.75pt off 0.75pt}]  (185.95,219.35) .. controls (183.92,220.6) and (182.02,220.15) .. (180.26,218) .. controls (179.16,215.98) and (177.63,215.55) .. (175.66,216.71) .. controls (173.67,217.82) and (172,217.26) .. (170.63,215.02) .. controls (169.87,212.95) and (168.31,212.32) .. (165.94,213.14) .. controls (163.97,214.08) and (162.52,213.4) .. (161.6,211.09) .. controls (160.87,208.8) and (159.35,207.95) .. (157.05,208.52) .. controls (154.7,208.98) and (153.33,208.05) .. (152.94,205.72) .. controls (152.75,203.44) and (151.38,202.3) .. (148.85,202.29) .. controls (146.56,202.37) and (145.38,201.13) .. (145.31,198.57) .. controls (145.71,196.4) and (144.72,195.07) .. (142.33,194.58) .. controls (139.93,193.85) and (139.04,192.27) .. (139.66,189.84) .. controls (140.62,187.89) and (140.01,186.37) .. (137.84,185.3) .. controls (135.71,183.97) and (135.24,182.2) .. (136.45,179.99) .. controls (137.85,178.28) and (137.62,176.61) .. (135.76,174.97) .. controls (134,173.5) and (133.94,171.76) .. (135.59,169.74) .. controls (137.32,168.25) and (137.41,166.64) .. (135.87,164.93) .. controls (134.41,163.06) and (134.63,161.4) .. (136.54,159.96) .. controls (138.49,158.6) and (138.85,156.89) .. (137.6,154.84) .. controls (136.44,152.65) and (136.92,150.9) .. (139.04,149.57) .. controls (141.05,148.75) and (141.57,147.18) .. (140.6,144.85) .. controls (139.55,142.89) and (140.16,141.28) .. (142.45,140.02) .. controls (144.57,139.27) and (145.17,137.87) .. (144.25,135.8) .. controls (143.39,133.67) and (144.18,131.99) .. (146.61,130.78) .. controls (148.82,130.08) and (149.57,128.62) .. (148.84,126.4) .. controls (148.17,124.12) and (148.98,122.63) .. (151.27,121.94) -- (153.45,118.18)(186.57,216.41) .. controls (184.58,217.68) and (182.73,217.24) .. (181.02,215.1) .. controls (179.96,213.09) and (178.47,212.67) .. (176.54,213.84) .. controls (174.59,214.97) and (172.97,214.42) .. (171.66,212.2) .. controls (170.95,210.15) and (169.45,209.54) .. (167.14,210.39) .. controls (165.21,211.35) and (163.82,210.69) .. (162.97,208.42) .. controls (162.3,206.17) and (160.86,205.36) .. (158.63,205.97) .. controls (156.36,206.48) and (155.06,205.59) .. (154.74,203.32) .. controls (154.62,201.1) and (153.34,200.03) .. (150.9,200.1) .. controls (148.68,200.25) and (147.58,199.1) .. (147.6,196.64) .. controls (148.07,194.55) and (147.15,193.32) .. (144.84,192.94) .. controls (142.51,192.32) and (141.68,190.86) .. (142.37,188.55) .. controls (143.37,186.7) and (142.81,185.3) .. (140.68,184.36) .. controls (138.59,183.15) and (138.16,181.51) .. (139.39,179.42) .. controls (140.81,177.81) and (140.6,176.24) .. (138.75,174.72) .. controls (137,173.34) and (136.94,171.7) .. (138.59,169.79) .. controls (140.31,168.38) and (140.4,166.85) .. (138.86,165.22) .. controls (137.39,163.43) and (137.6,161.84) .. (139.5,160.46) .. controls (141.44,159.17) and (141.78,157.53) .. (140.52,155.54) .. controls (139.35,153.41) and (139.81,151.72) .. (141.91,150.45) .. controls (143.9,149.66) and (144.41,148.13) .. (143.43,145.86) .. controls (142.36,143.93) and (142.96,142.36) .. (145.23,141.15) .. controls (147.33,140.44) and (147.92,139.06) .. (146.99,137.02) .. controls (146.12,134.92) and (146.89,133.28) .. (149.3,132.1) .. controls (151.5,131.43) and (152.23,130) .. (151.5,127.8) .. controls (150.81,125.55) and (151.61,124.09) .. (153.89,123.42) -- (156.03,119.71) ;
%Straight Lines [id:da7025330583030641] 
\draw    (154.74,118.94) -- (154.74,48.76) ;
\draw [shift={(154.74,48.76)}, rotate = 270] [color={rgb, 255:red, 0; green, 0; blue, 0 }  ][fill={rgb, 255:red, 0; green, 0; blue, 0 }  ][line width=0.75]      (0, 0) circle [x radius= 3.35, y radius= 3.35]   ;
\draw [shift={(154.74,118.94)}, rotate = 270] [color={rgb, 255:red, 0; green, 0; blue, 0 }  ][fill={rgb, 255:red, 0; green, 0; blue, 0 }  ][line width=0.75]      (0, 0) circle [x radius= 3.35, y radius= 3.35]   ;
%Straight Lines [id:da5384313400700488] 
\draw    (154.74,118.94) -- (249.31,151.08) ;
\draw [shift={(249.31,151.08)}, rotate = 18.77] [color={rgb, 255:red, 0; green, 0; blue, 0 }  ][fill={rgb, 255:red, 0; green, 0; blue, 0 }  ][line width=0.75]      (0, 0) circle [x radius= 3.35, y radius= 3.35]   ;
\draw [shift={(154.74,118.94)}, rotate = 18.77] [color={rgb, 255:red, 0; green, 0; blue, 0 }  ][fill={rgb, 255:red, 0; green, 0; blue, 0 }  ][line width=0.75]      (0, 0) circle [x radius= 3.35, y radius= 3.35]   ;
%Straight Lines [id:da16690421882543116] 
\draw    (154.74,48.76) -- (249.31,80.89) ;
\draw [shift={(249.31,80.89)}, rotate = 18.77] [color={rgb, 255:red, 0; green, 0; blue, 0 }  ][fill={rgb, 255:red, 0; green, 0; blue, 0 }  ][line width=0.75]      (0, 0) circle [x radius= 3.35, y radius= 3.35]   ;
\draw [shift={(154.74,48.76)}, rotate = 18.77] [color={rgb, 255:red, 0; green, 0; blue, 0 }  ][fill={rgb, 255:red, 0; green, 0; blue, 0 }  ][line width=0.75]      (0, 0) circle [x radius= 3.35, y radius= 3.35]   ;
%Straight Lines [id:da4465774079984246] 
\draw    (249.31,151.08) -- (249.31,80.89) ;
\draw [shift={(249.31,80.89)}, rotate = 270] [color={rgb, 255:red, 0; green, 0; blue, 0 }  ][fill={rgb, 255:red, 0; green, 0; blue, 0 }  ][line width=0.75]      (0, 0) circle [x radius= 3.35, y radius= 3.35]   ;
\draw [shift={(249.31,151.08)}, rotate = 270] [color={rgb, 255:red, 0; green, 0; blue, 0 }  ][fill={rgb, 255:red, 0; green, 0; blue, 0 }  ][line width=0.75]      (0, 0) circle [x radius= 3.35, y radius= 3.35]   ;
%Straight Lines [id:da7581077397817662] 
\draw    (154.74,48.76) -- (113.33,82.56) ;
\draw [shift={(113.33,82.56)}, rotate = 140.77] [color={rgb, 255:red, 0; green, 0; blue, 0 }  ][fill={rgb, 255:red, 0; green, 0; blue, 0 }  ][line width=0.75]      (0, 0) circle [x radius= 3.35, y radius= 3.35]   ;
\draw [shift={(154.74,48.76)}, rotate = 140.77] [color={rgb, 255:red, 0; green, 0; blue, 0 }  ][fill={rgb, 255:red, 0; green, 0; blue, 0 }  ][line width=0.75]      (0, 0) circle [x radius= 3.35, y radius= 3.35]   ;
%Straight Lines [id:da6027951135098035] 
\draw    (154.74,118.94) -- (113.33,152.75) ;
\draw [shift={(113.33,152.75)}, rotate = 140.77] [color={rgb, 255:red, 0; green, 0; blue, 0 }  ][fill={rgb, 255:red, 0; green, 0; blue, 0 }  ][line width=0.75]      (0, 0) circle [x radius= 3.35, y radius= 3.35]   ;
\draw [shift={(154.74,118.94)}, rotate = 140.77] [color={rgb, 255:red, 0; green, 0; blue, 0 }  ][fill={rgb, 255:red, 0; green, 0; blue, 0 }  ][line width=0.75]      (0, 0) circle [x radius= 3.35, y radius= 3.35]   ;
%Straight Lines [id:da4558782310599374] 
\draw  [dash pattern={on 4.5pt off 4.5pt}]  (113.33,152.75) -- (249.31,151.08) ;
\draw [shift={(249.31,151.08)}, rotate = 359.3] [color={rgb, 255:red, 0; green, 0; blue, 0 }  ][fill={rgb, 255:red, 0; green, 0; blue, 0 }  ][line width=0.75]      (0, 0) circle [x radius= 3.35, y radius= 3.35]   ;
\draw [shift={(113.33,152.75)}, rotate = 359.3] [color={rgb, 255:red, 0; green, 0; blue, 0 }  ][fill={rgb, 255:red, 0; green, 0; blue, 0 }  ][line width=0.75]      (0, 0) circle [x radius= 3.35, y radius= 3.35]   ;
%Straight Lines [id:da7226035049221524] 
\draw  [dash pattern={on 4.5pt off 4.5pt}]  (113.33,82.56) -- (249.31,80.89) ;
\draw [shift={(249.31,80.89)}, rotate = 359.3] [color={rgb, 255:red, 0; green, 0; blue, 0 }  ][fill={rgb, 255:red, 0; green, 0; blue, 0 }  ][line width=0.75]      (0, 0) circle [x radius= 3.35, y radius= 3.35]   ;
\draw [shift={(113.33,82.56)}, rotate = 359.3] [color={rgb, 255:red, 0; green, 0; blue, 0 }  ][fill={rgb, 255:red, 0; green, 0; blue, 0 }  ][line width=0.75]      (0, 0) circle [x radius= 3.35, y radius= 3.35]   ;
%Straight Lines [id:da6297571649311656] 
\draw  [dash pattern={on 4.5pt off 4.5pt}]  (113.33,82.56) -- (113.33,152.75) ;
\draw [shift={(113.33,152.75)}, rotate = 90] [color={rgb, 255:red, 0; green, 0; blue, 0 }  ][fill={rgb, 255:red, 0; green, 0; blue, 0 }  ][line width=0.75]      (0, 0) circle [x radius= 3.35, y radius= 3.35]   ;
\draw [shift={(113.33,82.56)}, rotate = 90] [color={rgb, 255:red, 0; green, 0; blue, 0 }  ][fill={rgb, 255:red, 0; green, 0; blue, 0 }  ][line width=0.75]      (0, 0) circle [x radius= 3.35, y radius= 3.35]   ;
%Straight Lines [id:da9646748753895977] 
\draw    (475.79,222.83) -- (570.36,254.97) ;
\draw [shift={(570.36,254.97)}, rotate = 18.77] [color={rgb, 255:red, 0; green, 0; blue, 0 }  ][fill={rgb, 255:red, 0; green, 0; blue, 0 }  ][line width=0.75]      (0, 0) circle [x radius= 3.35, y radius= 3.35]   ;
\draw [shift={(475.79,222.83)}, rotate = 18.77] [color={rgb, 255:red, 0; green, 0; blue, 0 }  ][fill={rgb, 255:red, 0; green, 0; blue, 0 }  ][line width=0.75]      (0, 0) circle [x radius= 3.35, y radius= 3.35]   ;
%Straight Lines [id:da05798579726844655] 
\draw    (557.58,195.77) -- (570.36,254.97) ;
\draw [shift={(570.36,254.97)}, rotate = 77.82] [color={rgb, 255:red, 0; green, 0; blue, 0 }  ][fill={rgb, 255:red, 0; green, 0; blue, 0 }  ][line width=0.75]      (0, 0) circle [x radius= 3.35, y radius= 3.35]   ;
\draw [shift={(557.58,195.77)}, rotate = 77.82] [color={rgb, 255:red, 0; green, 0; blue, 0 }  ][fill={rgb, 255:red, 0; green, 0; blue, 0 }  ][line width=0.75]      (0, 0) circle [x radius= 3.35, y radius= 3.35]   ;
%Straight Lines [id:da9734436278797667] 
\draw    (475.79,222.83) -- (557.58,195.77) ;
\draw [shift={(557.58,195.77)}, rotate = 341.69] [color={rgb, 255:red, 0; green, 0; blue, 0 }  ][fill={rgb, 255:red, 0; green, 0; blue, 0 }  ][line width=0.75]      (0, 0) circle [x radius= 3.35, y radius= 3.35]   ;
\draw [shift={(475.79,222.83)}, rotate = 341.69] [color={rgb, 255:red, 0; green, 0; blue, 0 }  ][fill={rgb, 255:red, 0; green, 0; blue, 0 }  ][line width=0.75]      (0, 0) circle [x radius= 3.35, y radius= 3.35]   ;
%Straight Lines [id:da010294313210437944] 
\draw  [dash pattern={on 0.75pt off 0.75pt}]  (445.69,122.69) .. controls (447.78,123.77) and (448.29,125.36) .. (447.21,127.45) .. controls (446.13,129.54) and (446.64,131.13) .. (448.73,132.21) .. controls (450.82,133.3) and (451.33,134.89) .. (450.25,136.98) .. controls (449.16,139.07) and (449.67,140.66) .. (451.76,141.74) .. controls (453.85,142.83) and (454.36,144.42) .. (453.28,146.51) .. controls (452.2,148.6) and (452.71,150.19) .. (454.8,151.27) .. controls (456.89,152.35) and (457.4,153.94) .. (456.32,156.03) .. controls (455.24,158.12) and (455.75,159.71) .. (457.84,160.8) .. controls (459.93,161.88) and (460.44,163.47) .. (459.35,165.56) .. controls (458.27,167.65) and (458.78,169.24) .. (460.87,170.33) .. controls (462.96,171.41) and (463.47,173) .. (462.39,175.09) .. controls (461.31,177.18) and (461.82,178.77) .. (463.91,179.85) .. controls (466,180.94) and (466.51,182.53) .. (465.42,184.62) .. controls (464.34,186.71) and (464.85,188.3) .. (466.94,189.38) .. controls (469.03,190.47) and (469.54,192.06) .. (468.46,194.15) .. controls (467.38,196.24) and (467.89,197.83) .. (469.98,198.91) .. controls (472.07,199.99) and (472.58,201.58) .. (471.5,203.67) .. controls (470.41,205.76) and (470.92,207.35) .. (473.01,208.44) .. controls (475.1,209.52) and (475.61,211.11) .. (474.53,213.2) .. controls (473.45,215.29) and (473.96,216.88) .. (476.05,217.97) -- (477.21,221.62) -- (477.21,221.62)(442.83,123.6) .. controls (444.92,124.68) and (445.43,126.27) .. (444.35,128.36) .. controls (443.27,130.45) and (443.78,132.04) .. (445.87,133.12) .. controls (447.96,134.21) and (448.47,135.8) .. (447.39,137.89) .. controls (446.3,139.98) and (446.81,141.57) .. (448.9,142.65) .. controls (450.99,143.74) and (451.5,145.33) .. (450.42,147.42) .. controls (449.34,149.51) and (449.85,151.1) .. (451.94,152.18) .. controls (454.03,153.26) and (454.54,154.85) .. (453.46,156.94) .. controls (452.38,159.03) and (452.89,160.62) .. (454.98,161.71) .. controls (457.07,162.79) and (457.58,164.38) .. (456.49,166.47) .. controls (455.41,168.56) and (455.92,170.15) .. (458.01,171.24) .. controls (460.1,172.32) and (460.61,173.91) .. (459.53,176) .. controls (458.45,178.09) and (458.96,179.68) .. (461.05,180.76) .. controls (463.14,181.85) and (463.65,183.44) .. (462.57,185.53) .. controls (461.48,187.62) and (461.99,189.21) .. (464.08,190.29) .. controls (466.17,191.38) and (466.68,192.97) .. (465.6,195.06) .. controls (464.52,197.15) and (465.03,198.74) .. (467.12,199.82) .. controls (469.21,200.9) and (469.72,202.49) .. (468.64,204.58) .. controls (467.56,206.67) and (468.07,208.26) .. (470.16,209.35) .. controls (472.25,210.43) and (472.76,212.02) .. (471.67,214.11) .. controls (470.59,216.2) and (471.1,217.79) .. (473.19,218.88) -- (474.36,222.53) -- (474.36,222.53) ;
%Curve Lines [id:da8561870012168837] 
\draw  [dash pattern={on 0.75pt off 0.75pt}]  (475.47,223.54) .. controls (473.44,224.8) and (471.55,224.35) .. (469.78,222.2) .. controls (468.68,220.17) and (467.15,219.74) .. (465.19,220.9) .. controls (463.2,222.02) and (461.52,221.46) .. (460.15,219.21) .. controls (459.4,217.14) and (457.83,216.52) .. (455.46,217.34) .. controls (453.49,218.28) and (452.04,217.6) .. (451.12,215.29) .. controls (450.39,213) and (448.87,212.14) .. (446.57,212.72) .. controls (444.22,213.18) and (442.86,212.24) .. (442.47,209.91) .. controls (442.27,207.64) and (440.9,206.5) .. (438.37,206.49) .. controls (436.08,206.57) and (434.9,205.33) .. (434.84,202.77) .. controls (435.23,200.6) and (434.24,199.27) .. (431.85,198.78) .. controls (429.45,198.04) and (428.56,196.46) .. (429.18,194.04) .. controls (430.14,192.08) and (429.53,190.57) .. (427.36,189.5) .. controls (425.23,188.17) and (424.77,186.4) .. (425.97,184.19) .. controls (427.38,182.48) and (427.15,180.81) .. (425.28,179.17) .. controls (423.53,177.7) and (423.47,175.95) .. (425.12,173.94) .. controls (426.84,172.45) and (426.93,170.84) .. (425.39,169.12) .. controls (423.93,167.25) and (424.16,165.6) .. (426.07,164.15) .. controls (428.02,162.8) and (428.38,161.09) .. (427.13,159.03) .. controls (425.97,156.85) and (426.45,155.1) .. (428.56,153.77) .. controls (430.57,152.95) and (431.1,151.38) .. (430.13,149.05) .. controls (429.08,147.08) and (429.69,145.47) .. (431.97,144.22) .. controls (434.09,143.47) and (434.69,142.07) .. (433.78,140) .. controls (432.92,137.87) and (433.7,136.19) .. (436.13,134.98) .. controls (438.34,134.28) and (439.09,132.82) .. (438.37,130.59) .. controls (437.7,128.32) and (438.51,126.83) .. (440.8,126.14) -- (442.97,122.38)(476.1,220.61) .. controls (474.1,221.87) and (472.25,221.43) .. (470.54,219.3) .. controls (469.48,217.29) and (467.99,216.86) .. (466.07,218.03) .. controls (464.12,219.16) and (462.49,218.62) .. (461.18,216.4) .. controls (460.47,214.35) and (458.97,213.74) .. (456.66,214.59) .. controls (454.73,215.55) and (453.34,214.89) .. (452.49,212.62) .. controls (451.83,210.37) and (450.39,209.55) .. (448.16,210.17) .. controls (445.89,210.68) and (444.59,209.79) .. (444.27,207.52) .. controls (444.15,205.3) and (442.87,204.23) .. (440.42,204.3) .. controls (438.2,204.45) and (437.1,203.29) .. (437.13,200.83) .. controls (437.59,198.75) and (436.67,197.52) .. (434.36,197.13) .. controls (432.03,196.51) and (431.21,195.05) .. (431.89,192.75) .. controls (432.9,190.9) and (432.34,189.5) .. (430.21,188.55) .. controls (428.12,187.35) and (427.69,185.7) .. (428.92,183.61) .. controls (430.34,182) and (430.12,180.44) .. (428.27,178.92) .. controls (426.52,177.54) and (426.47,175.9) .. (428.12,173.99) .. controls (429.84,172.58) and (429.93,171.05) .. (428.38,169.41) .. controls (426.91,167.62) and (427.13,166.03) .. (429.02,164.66) .. controls (430.96,163.37) and (431.3,161.73) .. (430.04,159.74) .. controls (428.87,157.61) and (429.33,155.91) .. (431.43,154.64) .. controls (433.43,153.86) and (433.94,152.33) .. (432.95,150.06) .. controls (431.88,148.13) and (432.48,146.56) .. (434.75,145.35) .. controls (436.86,144.64) and (437.45,143.26) .. (436.52,141.22) .. controls (435.65,139.12) and (436.42,137.48) .. (438.83,136.3) .. controls (441.03,135.63) and (441.76,134.19) .. (441.02,132) .. controls (440.34,129.75) and (441.14,128.28) .. (443.41,127.61) -- (445.55,123.9) ;
%Straight Lines [id:da0913364658700343] 
\draw    (444.26,123.14) -- (444.26,52.96) ;
\draw [shift={(444.26,52.96)}, rotate = 270] [color={rgb, 255:red, 0; green, 0; blue, 0 }  ][fill={rgb, 255:red, 0; green, 0; blue, 0 }  ][line width=0.75]      (0, 0) circle [x radius= 3.35, y radius= 3.35]   ;
\draw [shift={(444.26,123.14)}, rotate = 270] [color={rgb, 255:red, 0; green, 0; blue, 0 }  ][fill={rgb, 255:red, 0; green, 0; blue, 0 }  ][line width=0.75]      (0, 0) circle [x radius= 3.35, y radius= 3.35]   ;
%Straight Lines [id:da4113430644531171] 
\draw    (444.26,123.14) -- (538.83,155.27) ;
\draw [shift={(538.83,155.27)}, rotate = 18.77] [color={rgb, 255:red, 0; green, 0; blue, 0 }  ][fill={rgb, 255:red, 0; green, 0; blue, 0 }  ][line width=0.75]      (0, 0) circle [x radius= 3.35, y radius= 3.35]   ;
\draw [shift={(444.26,123.14)}, rotate = 18.77] [color={rgb, 255:red, 0; green, 0; blue, 0 }  ][fill={rgb, 255:red, 0; green, 0; blue, 0 }  ][line width=0.75]      (0, 0) circle [x radius= 3.35, y radius= 3.35]   ;
%Straight Lines [id:da2627849105127581] 
\draw    (444.26,52.96) -- (538.83,85.09) ;
\draw [shift={(538.83,85.09)}, rotate = 18.77] [color={rgb, 255:red, 0; green, 0; blue, 0 }  ][fill={rgb, 255:red, 0; green, 0; blue, 0 }  ][line width=0.75]      (0, 0) circle [x radius= 3.35, y radius= 3.35]   ;
\draw [shift={(444.26,52.96)}, rotate = 18.77] [color={rgb, 255:red, 0; green, 0; blue, 0 }  ][fill={rgb, 255:red, 0; green, 0; blue, 0 }  ][line width=0.75]      (0, 0) circle [x radius= 3.35, y radius= 3.35]   ;
%Straight Lines [id:da6550434073196226] 
\draw  [dash pattern={on 4.5pt off 4.5pt}]  (538.83,155.27) -- (538.83,85.09) ;
\draw [shift={(538.83,85.09)}, rotate = 270] [color={rgb, 255:red, 0; green, 0; blue, 0 }  ][fill={rgb, 255:red, 0; green, 0; blue, 0 }  ][line width=0.75]      (0, 0) circle [x radius= 3.35, y radius= 3.35]   ;
\draw [shift={(538.83,155.27)}, rotate = 270] [color={rgb, 255:red, 0; green, 0; blue, 0 }  ][fill={rgb, 255:red, 0; green, 0; blue, 0 }  ][line width=0.75]      (0, 0) circle [x radius= 3.35, y radius= 3.35]   ;

% Text Node
\draw (231.16,235.98) node [anchor=north west][inner sep=0.75pt]    {$e$};
% Text Node
\draw (241.48,208.77) node [anchor=north west][inner sep=0.75pt]    {$J$};
% Text Node
\draw (148.7,167.2) node [anchor=north west][inner sep=0.75pt]    {$C$};
% Text Node
\draw (160.92,99.31) node [anchor=north west][inner sep=0.75pt]    {$w_{1}$};
% Text Node
\draw (158.31,26.92) node [anchor=north west][inner sep=0.75pt]    {$w_{2}$};
% Text Node
\draw (201.35,116.8) node [anchor=north west][inner sep=0.75pt]    {$e'$};
% Text Node
\draw (202.2,44.02) node [anchor=north west][inner sep=0.75pt]    {$e''$};
% Text Node
\draw (122.79,45.76) node [anchor=north west][inner sep=0.75pt]    {$f$};
% Text Node
\draw (127.64,112.92) node [anchor=north west][inner sep=0.75pt]    {$g$};
% Text Node
\draw (531,212.97) node [anchor=north west][inner sep=0.75pt]    {$J$};
% Text Node
\draw (438.22,171.39) node [anchor=north west][inner sep=0.75pt]    {$C$};
% Text Node
\draw (417.2,108.41) node [anchor=north west][inner sep=0.75pt]    {$w_{1}$};
% Text Node
\draw (418.05,39.51) node [anchor=north west][inner sep=0.75pt]    {$w_{2}$};
% Text Node
\draw (579.63,249.77) node [anchor=north west][inner sep=0.75pt]    {$v_{1}$};
% Text Node
\draw (484.64,143.15) node [anchor=north west][inner sep=0.75pt]    {$e'$};
% Text Node
\draw (487.51,48) node [anchor=north west][inner sep=0.75pt]    {$e''$};

\end{tikzpicture}